\documentclass[a4paper,oneside,10pt]{article}%
\usepackage{amsfonts}
\usepackage{amsmath,amssymb}
\usepackage{amsmath}
\usepackage{amssymb}
\usepackage{graphicx}
\usepackage{hyperref}
\usepackage{color}%
\providecommand{\U}[1]{\protect\rule{.1in}{.1in}}
\newtheorem{theorem}{Theorem}[section]

\newtheorem{definition}[theorem]{Definition}
\newtheorem{example}[theorem]{Example}
\newtheorem{lemma}[theorem]{Lemma}

\newtheorem{proposition}[theorem]{Proposition}
\newtheorem{remark}[theorem]{Remark}
\numberwithin{equation}{section}

\newenvironment{proof}[1][Proof]{\noindent \textbf{#1.} }{\ \rule{0.5em}{0.5em}}

\def\b{\big}
\def\B{\Big}
\begin{document}

\title{Relationship between maximum principle and dynamic programming principle for
recursive optimal control problem of stochastic evolution equations}
\author{Ying Hu\thanks{Institut de Recherche Math\'ematique de Rennes, Universit\'e
Rennes 1, 35042 Rennes Cedex, France. ying.hu@univ-rennes1.fr. This author's
research is partially supported by Lebesgue Center of Mathematics
``Investissements d'avenir" Program (No. ANR-11-LABX-0020-01), by ANR CAESARS
(No. ANR-15-CE05-0024) and by ANR MFG (No. ANR-16-CE40-0015-01).}
\and Guomin Liu\thanks{School of Mathematical Sciences, Nankai University, Tianjin,
China. gmliu@nankai.edu.cn. Research supported by National Natural Science
Foundation of China (No. 12201315 and No. 12571479).}
\and Shanjian Tang\thanks{School of Mathematical Sciences, Fudan University,
Shanghai, China. sjtang@fudan.edu.cn. Research supported by National Key R\&D
Program of China (No. 2018YFA0703900) and National Natural Science Foundation
of China (No. 11631004 and No. 12031009).}}
\date{}
\maketitle

\begin{abstract}
This paper investigates the relationship between the maximum principle (MP)
and the dynamic programming principle (DPP) for the recursive optimal control
problem of stochastic evolution equations, allowing for nonconvex control domains and nonsmooth value functions. Using the notion of conditionally expected operator-valued backward stochastic integral equations, we establish the connection between the first- and
second-order adjoint processes in MP and the generalized derivatives of the value
function in DPP. Under certain additional assumptions, we also show that  the value function  is  $C^{1,1}$-regular. Furthermore,  the smooth case and  several applications illustrating our results are provided.

 \smallskip


\end{abstract}

\vspace{-10pt}

\section{Introduction}

Pontryagin's maximum principle (MP) and Bellman's dynamic programming
principle (DPP) are two fundamental approaches in  optimal control. While they are often studied separately, 
it is crucial
to understand their relationship, particularly the connection between the
adjoint processes in MP and the value function in DPP, which play central
roles in their respective theories.

The relationship between the MP and the DPP for controlled ordinary
differential equations was  given by Pontryagin et al.~\cite{PBGM62}, assuming that the value function is
continuously differentiable. By employing the notion of viscosity solutions, subsequent works extended this connection to nonsmooth value functions; see Barron and Jensen \cite{BJ86},
Clarke and Vinter \cite{CV87} and Zhou \cite{Zhou90}.  Cannarsa and
Frankowska \cite{CF92,CF96} and Cernea and Frankowska \cite{CF05} later obtained the corresponding results for control systems governed by partial differential equations.

In the stochastic setting, Peng \cite{Pe90} derived a general MP for controlled stochastic differential equations with possibly nonconvex control domains by introducing a second-order adjoint process governed by a matrix-valued backward stochastic differential equation (BSDE). The corresponding
relationship was obtained in the smooth case by Bensoussan \cite{Ben82} and
in the nonsmooth case by Zhou \cite{Zhou90-2,Zhou91}. For infinite
dimensional stochastic systems  with nonconvex control domain,
\cite{DM13,FHT13,LZ14} studied the MP, and Chen and L\"{u} \cite{CL23}
recently established the connection between the MP and the DPP for stochastic
evolution equations (SEEs), and Stannat and Wessels \cite{SW24} investigated the
connection for semilinear SPDEs.

The main objective of this paper is to establish the connection between the MP
and the DPP for the following controlled SEEs:
\begin{equation}
	\left\{
	\begin{aligned}
		dX(t)= &\,\, \big[A(t)X(t)+a(t,X(t),u(t))\big]dt + \big[B(t)X(t)+b(t,X(t),u(t))\big]dw(t),\quad t\in[0,T], \\
		X(0)= &\,\, x.
	\end{aligned}
	\right.
	\label{eq0-1}
\end{equation}
where $w(\cdot)$ is a  Brownian motion, $A$ and $B$ are unbounded
linear operators, $a$ and $b$ are nonlinear functions, and $u(\cdot)$ is a control
process taking values in a given metric space. The diffusion coefficient $b$ 
depends on the control variable, the control domain is not necessarily convex,
and the value function is not assumed to be smooth. The value of the cost functional at an admissible control $u$  is
defined by
\[
J(x;u(\cdot)):=Y(0),
\]
where $Y(\cdot)$ is the recursive utility subject to a BSDE:
\begin{equation}
Y(t)=h(X(T))+\int_{t}^{T}k(s,X(s),Y(s),Z(s),u(s))dr-\int_{t}^{T}%
Z(s)dw({s}),\text{\quad}t\in\lbrack 0,T]. \label{eq0-2}
\end{equation}
The notion of a recursive utility in continuous time was introduced by Duffie
and Epstein \cite{DE92} and further generalized by Peng \cite{Pe93} and El
Karoui et al.~\cite{KPQ97}. Stochastic recursive optimal control
problems have found important applications in mathematical economics,
mathematical finance and engineering (see, e.g., El Karoui et al. 
\cite{KPQ97}). When $k$ is invariant with $(y,z)$, we have by taking expectation on
both sides of (\ref{eq0-2}) that
\[
J(x;u(\cdot))=\mathbb{E}\big[h(X(T))+\int_{0}^{T}k(t,X(t),u(t))dt\big],
\]
and thus reduce  the problem to the conventional stochastic optimal control
problem studied in \cite{DM13,FHT13,LZ14,CL23,SW24}.

As for the stochastic recursive optimal control problems for finite
dimensional systems, Peng \cite{Pe93} first derived a local MP when the
control domain is convex. Recently, Hu \cite{Hu-17} obtained a general MP for
the stochastic recursive optimal control problem, solving a long-standing
open problem proposed by Peng \cite{Pe98}. On the
other hand, DPP for recursive control systems and the associated generalized Hamilton-Jacobi-Bellman (HJB) equations were developed in Peng \cite{Pe92} (see also \cite{YPFW97}). Concerning the connection between the MP and the DPP, Shi \cite{Shi10} and Shi and Yu \cite{SY13} investigated the local
case with convex control domain and smooth value function; while Nie et al.~\cite{NSW16} extended this to the local case with a convex control domain
in the viscosity solution framework. The general case when the domain of
the control is nonconvex was addressed by Nie et al.~\cite{NSW17}.

Recently,  the last two authors \cite{LT21} established the MP for the recursive
optimal control problem (\ref{eq0-2}) of SEEs (\ref{eq0-1}), where the control domain is a general metric
space (not necessarily convex). The present paper is a sequel, 
focusing on the connection between the MP and the DPP for the  recursive
control system (\ref{eq0-1}) and (\ref{eq0-2}). 
Our analysis builds on the notion of conditionally expected operator-valued backward stochastic integral equations (BSIEs) introduced in \cite{LT21}, which characterize the second-order adjoint processes in the MP for infinite-dimensional systems. A key component of our approach is the derivation of a  new It\^{o}'s formula for the second-order BSIE and the variational equations associated with perturbations of the initial state and time (see Theorems \ref{Myth3-7}, \ref{diff-space}, \ref{rel-time}), which plays a crucial role in establishing the  connection between MP and DPP.
Under additional
assumptions, we also obtain by a probabilistic argument the $C^{1,1}$-regularity of
the value function $V$, which offers further insights
on the first-order derivatives of $V$ and extends the  result  for conventional optimal control problems obtained in \cite{FSW23}; see Remarks \ref{Rm3-16} and     \ref{rm4-1}.

The rest of this paper is organized as follows. In Section 2, we recall the MP and derive the DPP
for recursive control problem of SEEs. In Section 3, we present the relationship between the MP and the  DPP  in the nonsmooth case. Section 4 is devoted to  the
$C^{1,1}$-regularity of the value function. We discuss the special smooth case and present two applications of results
in Section 4. In the appendix, some
technical results are proved.
\section{Preliminaries and formulation of the problem}

In this section, we first introduce the concept of conditionally
expected operator-valued BSIEs, which serve  as the second-adjoint equations for optimally
controlled SEEs. We then recall the MP  for the recursive optimal control problem of SEEs  and present the DPP.

\subsection{Conditional expected operator-valued BSIEs}
Consider a probability space $(\Omega,\mathcal{F},\mathbb{P})$ with a filtration $\mathbb{F}:=\{\mathcal{F}_{t}\}_{0\leq t\leq T}$ satisfying the usual
conditions, for fixed $T>0.$ 
We denote by $\Vert\cdot\Vert_{X}$ the norm on a Banach space $X$.
The space of all bounded linear operators from $X$ to another Banach space $Y$
is denoted by $\mathfrak{L}(X;Y)$, equipped with the operator norm, and we
write $\mathfrak{L}(X)$ when $X=Y$. Let $H$ be a separable
Hilbert space with inner product $\langle\cdot,\cdot\rangle$. 
We identify
$\mathfrak{L}(H;\mathbb{R})$ with $H.$ The adjoint of an operator $M$ is
denoted by $M^{\ast}$, $I_{d}$ denotes the identity operator on $H.$  

Denote by $\mathfrak{L}_{2}(H\times H;X)$ the  space of all bounded bilinear operators from $H\times H$ to $X$, equipped with the operator norm.
Recall that
\[
\mathfrak{L}_{2}(H\times H;X)=\mathfrak{L}(H;\mathfrak{L}(H;X))
\]
by identifying $\tilde{\varphi}\in\mathfrak{L}_{2}(H\times H;X)$ with
$\varphi\in\mathfrak{L}(H;\mathfrak{L}(H;X))$ through
$
\tilde{\varphi}(u,v):=\varphi(u)v,$ for $(u,v)\in H\times H,
$
In particular, we have
\[
\mathfrak{L}_{2}(H\times H):=\mathfrak{L}_{2}(H\times H;\mathbb{R}%
)=\mathfrak{L}(H;\mathfrak{L}(H;\mathbb{R}))=\mathfrak{L}(H;H)=\mathfrak{L}%
(H).
\]

For a sub-$\sigma$-algebra $\mathcal{G}$ of $\mathcal{F}$ and $\alpha
\geq1,$ we denote by $L^{\alpha}(\mathcal{G},H)$ the space of $H$-valued
$\mathcal{G}$-measurable mapping $y$ with norm $\Vert y\Vert_{L^{\alpha
}(\mathcal{G},H)}=\{{\mathbb{E}}[\Vert y\Vert_{H}^{\alpha}]\}^{\frac
{1}{\alpha}}$. Define $L_{\mathbb{F}}^{\alpha}(0,T;H)$ (resp. $L_{\mathbb{F}%
}^{2,\alpha}(0,T;H),$ $L_{\mathbb{F}}^{1,\alpha}(0,T;H)$) as the space of
$H$-valued progressively measurable processes $y(\cdot)$ with norm $\Vert
y\Vert_{L_{\mathbb{F}}^{\alpha}(0,T;H)}=\{{\mathbb{E}}[\int_{0}%
^{T}\Vert y(t)\Vert_{H}^{\alpha}{d}t]\}^{\frac{1}{\alpha}}$ (resp. $\Vert
y\Vert_{L_{\mathbb{F}}^{2,\alpha}(0,T;H)}=\{{\mathbb{E}}[(\int_{0}%
^{T}\Vert y(t)\Vert_{H}^{2}{d}t)^{\frac{\alpha}{2}}]\}^{\frac{1}{\alpha}}$,
$\Vert y\Vert_{L_{\mathbb{F}}^{1,\alpha}(0,T;H)}=\{{\mathbb{E}}%
[(\int_{0}^{T}\Vert y(t)\Vert_{H}{d}t)^{\alpha}]\}^{\frac{1}{\alpha}}$). We
write $L^{\alpha}(\mathcal{G})$, $L_{\mathbb{F}}^{\alpha}(0,T)$,
$L_{\mathbb{F}}^{2,\alpha}(0,T)$ and $L_{\mathbb{F}}^{1,\alpha}(0,T)$ for
$L^{\alpha}(\mathcal{G},\mathbb{R})$, $L_{\mathbb{F}}^{\alpha}(0,T;\mathbb{R}%
)$, $L_{\mathbb{F}}^{2,\alpha}(0,T;\mathbb{R})$ and $L_{\mathbb{F}}%
^{1,\alpha}(0,T;\mathbb{R})$, respectively. Note that
\begin{equation}\label{Myeq1-19}
	\mathfrak{L}_{2}(H\times
	H;L^{1}(\mathcal{G}))=\mathfrak{L}(H;\mathfrak{L}(H;L^{1}(\mathcal{G}))).
\end{equation}

An operator-valued random variable $Z:\Omega\rightarrow\mathfrak{L}(H)$ is called weakly $\mathcal{G}%
$-measurable if for each $(u,v)\in H\times H,$ $\langle Zu,v\rangle
:\Omega\rightarrow\mathbb{R}$ is $\mathcal{G}$-measurable. An operator-valued process
$Y:\Omega\times\lbrack0,T]\rightarrow\mathfrak{L}(H)$ is said to be weakly
progressively measurable if for each $(u,v)\in H\times H,$ the process
$\langle Yu,v\rangle:\Omega\times\lbrack0,T]\rightarrow$ $\mathbb{R}$ is
progressively measurable.
We denote by $L_{w}$ the weak $\sigma$-algebra on $\mathfrak{L}(H)$ generated by
all the sets in the form of
\[
\big\{z\in\mathfrak{L}(H):\langle zu,v\rangle\in A\big\},\quad u,v\in
H,\ A\in\mathcal{B}(\mathbb{R}).
\] 

We denote by $L_{w}^{\alpha}(\mathcal{G},\mathfrak{L}(H))$ the space of
$\mathfrak{L}(H)$-valued weakly $\mathcal{G}$-measurable mapping $F$ with norm
$\Vert F\Vert_{L_{w}^{\alpha}(\mathcal{G},\mathfrak{L}(H))}%
=\{{\mathbb{E}}[\Vert F\Vert_{\mathfrak{L}(H)}^{\alpha}]\}^{\frac
{1}{\alpha}}$. We define $L_{\mathbb{F},w}^{\alpha}(0,T;\mathfrak{L}(H))$
(resp. $L_{\mathbb{F},w}^{2,\alpha}(0,T;\mathfrak{L}(H))$) as the space of
$\mathfrak{L}(H)$-valued weakly progressively measurable processes $F(\cdot)$
with norm $\Vert F\Vert_{L_{\mathbb{F},w}^{\alpha}(0,T;\mathfrak{L}%
(H))}=\{{\mathbb{E}}[\int_{0}^{T}\Vert F(t)\Vert_{\mathfrak{L}%
(H)}^{\alpha}{d}t]\}^{\frac{1}{\alpha}}$ (resp. $\Vert F\Vert_{L_{\mathbb{F}%
,w}^{2,\alpha}(0,T;\mathfrak{L}(H))}=\{{\mathbb{E}}[(\int_{0}^{T}\Vert
F(t)\Vert_{\mathfrak{L}(H)}^{2}{d}t)^{\frac{\alpha}{2}}]\}^{\frac{1}{\alpha}}$).

For $Y\in\mathfrak{L}_{2}(H\times H;L^{1}(\mathcal{F})),$  its
conditional expectation $\mathbb{E}[Y|\mathcal{G}]$ with respect to
$\mathcal{G}$ is a weakly $\mathcal{G}$-measurable mapping $Z:\Omega\rightarrow\mathfrak{L}(H)$ satisfying
\[
\langle Zu,v\rangle=\mathbb{E}[Y(u,v)|\mathcal{G}],\quad P\text{-a.s., }\forall(u,v)\in
H\times H. \label{Myeq2-15}%
\]

\begin{theorem} [{\cite[Theorem 2.1]{LT21}}]
\label{Myth2-17} For $Y\in\mathfrak{L}_{2}(H\times H;L^{1}(\mathcal{F}))$, there exists $\mathbb{E}[Y|\mathcal{G}]\in L_{w}^{1}(\mathcal{G}%
,\mathfrak{L}(H))$ if and only if the mapping $(u,v)\longmapsto
\mathbb{E}[Y(u,v)|\mathcal{G}]\in\mathfrak{L}_{2}(H\times H;L^{1}%
(\mathcal{G}))$ satisfies the domination condition
\[
\big|\mathbb{E}[Y(u,v)|\mathcal{G}]\big|\leq g\Vert u\Vert_{H}\Vert v\Vert
_{H},\quad P\text{-a.s.},\ \forall(u,v)\in H\times H, \label{Myeq2-29}%
\]
for some $0\leq g\in L^{1}(\mathcal{G})$. Moreover, $\mathbb{E}[Y|\mathcal{G}%
]$ is unique (up to $P$-a.s. equality) and satisfies
\[
\big\Vert\mathbb{E}[Y|\mathcal{G}]\big\Vert_{\mathfrak{L}(H)}\leq
g,\quad P\text{-a.s.} \label{Myeq2-11}%
\]

\end{theorem}

We define a stochastic evolution operator on $H$ as a family of mappings
\[
\big\{L(t,s)\in\mathfrak{L}(L^{2}(\mathcal{F}_{t},H);L^{2}(\mathcal{F}%
_{s},H)):(t,s)\in\Delta\big\}
\]
with $\Delta=\{(t,s):0\leq t\leq s\leq T\}$. For each $(t,s)\in\Delta,$ the
formal adjoint $L^{\ast}(t,s)$ of $L(t,s)$ is defined for $u\in L^{1}%
(\mathcal{F}_{s},H)$ by%
\[
\big(L^{\ast}(t,s)u\big)(v):=\langle u,L(t,s)v\rangle\ \ P\text{-a.s.}%
,\quad\forall v\in L^{2}(\mathcal{F}_{t},H).
\]

We consider a conditionally expected $\mathfrak{L}(H)$-valued BSIE
\begin{equation}
P(t)=\mathbb{E}\Big[L^{\ast}(t,T)\xi L(t,T)+\int_{t}^{T}L^{\ast}%
(t,s)f(s,P(s))L(t,s)ds\Big|\mathcal{F}_{t}\Big],\quad t\in\lbrack0,T],
\label{Myeq2-1}%
\end{equation}
where $\xi$, $f$ and $L$ satisfy the following assumptions:

\begin{description}
\item[$(H1)$] There exists some constant $\Lambda\geq0$ such that for each
$(t,s)\in\Delta$ and $u\in L^{4}(\mathcal{F}_{t},H),$
\[
\mathbb{E}\big[\left\Vert L(t,s)u\right\Vert _{H}^{4}\big|\mathcal{F}%
_{t}\big]\leq\Lambda\Vert u\Vert_{H}^{4},\quad P\text{-a.s}.,
\]
and the mapping $(\omega,t,s)\mapsto(L(t,s)u)(\omega)$ is jointly measurable.

\item[$(H2)$] $\xi\in L_{w}^{2}(\mathcal{F}_{T},\mathfrak{L}(H))$; the
function $f(\omega,t,p):\Omega\times\lbrack0,T]\times\mathfrak{L}(H)\rightarrow
\mathfrak{L}(H)$ is $\mathcal{P}\otimes L_{w}/L_{w}$-measurable, Lipschitz continuous in
$p$ with some constant $\lambda\geq0$; $f(\cdot,\cdot,0)\in L_{\mathbb{F},w}%
^{2}(0,T;\mathfrak{L}(H)).$
\end{description}

For any $\eta\in L_{w}^{2}(\mathcal{F}_{s},\mathfrak{L}(H))$, we have $L^{\ast}(t,s)\eta L(t,s)\in\mathfrak{L}%
(H;\mathfrak{L}(H;L^{1}(\mathcal{F}_{s})))=\mathfrak{L}_{2}(H\times
H;L^{1}(\mathcal{F}_{s}))$ and we can write $(L^{\ast}(t,s)\eta
L(t,s)u)(v)=L^{\ast}(t,s)\eta L(t,s)(u,v)$, for
$(u,v)\in H\times H$.
For $g\in\mathfrak{L}_{2}(H\times H;L_{\mathbb{F}}^{1}(t,T)),$ we define its
weak integral $\int_{t}^{T}g(s)ds$ by%
\[
\Big(\int_{t}^{T}g(s)ds\Big)(u,v):=\int_{t}^{T}g(s)(u,v)ds\ \ P\text{-a.s}%
,\quad\forall(u,v)\in H\times H,
\]
and have $\int_{t}^{T}g(s)ds\in\mathfrak{L}_{2}(H\times H;L^{1}(\mathcal{F}_{T}))$.
Then for $h\in L_{\mathbb{F},w}^{2}(t,T;\mathfrak{L}(H))$, we derive $[t,T]\ni
s\mapsto L^{\ast}(t,s)h(s)L(t,s)\in\mathfrak{L}_{2}(H\times H;L_{\mathbb{F}%
}^{1}(t,T))$ and $\int_{t}^{T}L^{\ast}(t,s)h(s)L(t,s)ds\in\mathfrak{L}%
_{2}(H\times H;L^{1}(\mathcal{F}_{T}))$. Given any $P\in L_{\mathbb{F},w}%
^{2}(0,T;\mathfrak{L}(H))$, we have $f(\cdot,P(\cdot))\in L_{\mathbb{F},w}%
^{2}(0,T;\mathfrak{L}(H))$, so
\[
L^{\ast}(t,T)\xi L(t,T)+\int_{t}^{T}L^{\ast}(t,s)f(s,P(s))L(t,s)ds\in
\mathfrak{L}_{2}(H\times H;L^{1}(\mathcal{F}_{T})). \label{Myeq2-12}%
\]
By a solution of (\ref{Myeq2-1}), we mean a process $P\in L_{\mathbb{F},w}%
^{2}(0,T;\mathfrak{L}(H))$ satisfying, for each $t\in\lbrack0,T],$
\[
P(t)=\mathbb{E}\Big[L^{\ast}(t,T)\xi L(t,T)+\int_{t}^{T}L^{\ast}%
(t,s)f(s,P(s))L(t,s)ds\Big|\mathcal{F}_{t}\Big],\quad P\text{-a.s}.
\label{Myeq2-6}%
\]

In the following, the constant $C$  may vary from line to line.
\begin{theorem}[{\cite[Theorem 2.7]{LT21}}]
\label{Mainthm-1} Under assumptions $(H1)$ and $(H2)$, BSIE (\ref{Myeq2-1}) has a
unique solution $P$, with, for each $t\in
\lbrack0,T],$
\begin{equation}
\Vert P(t)\Vert_{\mathfrak{L}(H)}^{2}\leq C\mathbb{E}\Big[\Vert\xi
\Vert_{\mathfrak{L}(H)}^{2}+\int_{t}^{T}\Vert f(s,0)\Vert_{\mathfrak{L}%
(H)}^{2}ds\Big|\mathcal{F}_{t}\Big],\quad P\text{-a.s}., \label{Myeq3-23}%
\end{equation}
for some constant $C$ depending on $\Lambda$ and $\lambda$.
\end{theorem}

The proof of the following result is similar to that of Proposition 2.11 in \cite{LT21},  and thus is omitted.
\begin{proposition}
\label{Myth2-7-2} Let $\alpha\geq1.$ Suppose $(H1)$, $(H2)$ and
\begin{description}
\item[$(H3)$] $(\xi,f(\cdot,\cdot,0))\in L_{w}^{2\alpha}(\mathcal{F}%
_{T},\mathfrak{L}(H))\times L_{\mathbb{F},w}^{2,2\alpha}(0,T;\mathfrak{L}%
(H)),$ and there exists a constant $\Lambda_{\alpha}\geq0$ such that for each
$0\leq t\leq r\leq s\leq T$ and $u\in L^{4\alpha}(\mathcal{F}_{t},H),$ it holds
$L(t,s)=L(t,r)L(r,s)$,
\[
\mathbb{E}[\left\Vert L(t,s)u\right\Vert _{H}^{4\alpha}|\mathcal{F}_{t}%
]\leq\Lambda_{\alpha}\Vert u\Vert_{H}^{4\alpha}\ P\text{-a.s.}\ \text{and
}[t,T]\ni s\mapsto L(t,s)u\text{ is strongly continuous in }L^{4\alpha
}(\mathcal{F}_{T},H)\text{.}%
\]
\end{description}
Let $P$ be the solution to BSIE (\ref{Myeq2-1}). Then, for each $t\in
\lbrack0,T)$ and $u,v\in L^{4\alpha}(\mathcal{F}_{t},H),$
\[
\lim_{\delta\downarrow0}\mathbb{E}\big[|\langle P(t+\delta)u,v\rangle
-\langle P(t)u,v\rangle|^{\alpha}\big|\mathcal{F}_{t}%
\big]=0,\quad P\text{-a.s.}%
\]
\end{proposition}

We next present a typical example arising from linear stochastic evolution equations for $L(t,s)$.
Let $\mathcal{V}$ be a separable Hilbert space densely embedded in $H$, with dual space ${\mathcal{V}}^{\ast}:=\mathfrak{L}({\mathcal{V}};\mathbb{R})$.
Then $\mathcal{V}\subset H\subset\mathcal{V}^{\ast}$ form a Gelfand triple.
We denote by $\langle\cdot,\cdot\rangle_{\ast}$ the duality pairing between $\mathcal{V}^{\ast}$ and $\mathcal{V}$.
Let $w=\{w(t)\}_{t\in\lbrack0,T]}$ be a one-dimensional Brownian motion on $(\Omega,\mathcal{F},P)$. We remark that,  we restrict our discussion to the one-dimensional Brownian motion case for simplicity of presentation. Nevertheless, all results in this paper remain valid, after straightforward modifications, when $w$ is a $K$-valued cylindrical $Q$-Wiener process for some separable Hilbert space $K$ and a nonnegative self-adjoint operator $Q\in\mathfrak{L}(K)$, with stochastic integrands taking values in the space of Hilbert-Schmidt operators from $K$ to $H$; see, e.g., \cite{lsw25,lr15}.

In the following, we always assume that the filtration $\mathbb{F}%
=(\mathcal{F}_{t})_{0\leq t\leq T}$ is the augmented natural filtration of
$w$.

Consider the following linear homogeneous SEE on $[t,T]$:
\begin{equation}
	\left\{
	\begin{aligned}
		du^{t,u_{0}}(s)= &\,\, A(s)u^{t,u_{0}}(s)\,ds 
		+ B(s)u^{t,u_{0}}(s)\,dw(s),\quad s\in[t,T], \\
		u^{t,u_{0}}(t)= &\,\, u_{0}.
	\end{aligned}
	\right.
	\label{Myeq2-17}
\end{equation}
where $u_{0}\in L^{2}(\mathcal{F}_{t},H)$ and $(A,B):[0,T]\times
\Omega\rightarrow\mathfrak{L}(\mathcal{V};\mathcal{V}^{\ast}\times
H).$ 

We make the following assumption.
\begin{description}
\item[$(H4)$] For each $u\in\mathcal{V},$ $A(t,\omega)u$ and $B(t,\omega)u$
are progressively measurable and satisfy the following conditions for some constants
$\delta>0$ and $K_1\geq0$: for each
$t,\omega$ and $u\in\mathcal{V}$,

\begin{description}
\item[\rm{{(i)}}] Coercivity condition:
\[
2\langle A(t,\omega)u,u\rangle_{\ast}+\Vert B(t,\omega)u\Vert_H^{2}\leq-\delta\Vert u\Vert_{\mathcal{V}}^{2}+K_1\Vert u\Vert_{H}%
^{2}\quad\text{and}\quad\Vert A(t,\omega)u\Vert_{\mathcal{V}^{\ast}}\leq
K_1\Vert u\Vert_{\mathcal{V}};
\]
\item[\rm{{(ii)}}] Quasi-skew-symmetry condition: 
\[     
\b| \langle B(t,\omega)u,u\rangle\b|
\leq K_1\Vert u\Vert_{H}^{2}.
\]
\end{description}
\end{description}

We define a stochastic evolution operator $L_{A,B}$ as:
\begin{equation}
L_{A,B}(t,s)(u_{0}):=u^{t,u_{0}}(s)\in L^{2}(\mathcal{F}_{s},H),\quad
\text{for}\ t\leq s\leq T\ \text{and}\ u_{0}\in L^{2}(\mathcal{F}%
_{t},H)\text{.} \label{Myeq1-5}%
\end{equation}
From Theorem \ref{Mainthm-1}, the $\mathfrak{L}(H)$-valued BSIE
\[
P(t)=\mathbb{E}\Big[L_{A,B}^{\ast}(t,T)\xi L_{A,B}(t,T)+\int_{t}^{T}%
L_{A,B}^{\ast}(t,s)f(s,P(s))L_{A,B}(t,s)ds\Big|\mathcal{F}_{t}\Big],\quad
t\in\lbrack0,T], \label{Myeq3-2}%
\]
has a unique solution $P\in L_{\mathbb{F},w}^{2}(0,T;\mathfrak{L}(H))$.

\subsection{A priori estimates for SEEs}

Consider SEE
\begin{equation}
	\left\{
	\begin{aligned}
		dz(s)= &\,\, \big[A(s)z(s)+\tilde{a}(s,z(s))\big]ds  + \big[B(s)z(s)+\tilde{b}(s,z(s))\big]dw(s),\quad s\in[t,T], \\
		z(t)= &\,\, z_{0}.
	\end{aligned}
	\right.
	\label{Eq6-1}
\end{equation}
We have the following estimates depending on whether $\tilde{a}$ takes values in
$\mathcal{V}^{\ast}$ or in $H,$ whose proofs are deferred to   the appendix.
\begin{lemma}
\label{apriori-see} Assume $(H4)$  and $z_{0}\in
L^{2\alpha}(\mathcal{F}_{t},H),$  for  $\alpha\geq1$.
\begin{itemize}
\item[(i)] Suppose $(\tilde{a},\tilde{b}):[0,T]\times\Omega\times H\rightarrow
H\times H$ satisfies:    $\tilde{a}%
(\cdot,\cdot,z)$, $\tilde{b}(\cdot,\cdot,z)$ are progressively measurable for each $z\in H,$  
there exists a constant $L>0$ such that for all $(t,\omega)\in
\lbrack0,T]\times\Omega$ and $z,z^{\prime}\in H$,
\[
\Vert\tilde{a}(t,z)-\tilde{a}(t,z^{\prime})\Vert_{H}+\Vert\tilde
{b}(t,z)-\tilde{b}(t,z^{\prime})\Vert_H\leq L\Vert
z-z^{\prime}\Vert_{H},
\]
$\tilde{a}(\cdot,\cdot,0)\in L_{\mathbb{F}}^{1,2\alpha}(t,T;H)$, $\tilde
{b}(\cdot,\cdot,0)\in L_{\mathbb{F}}^{2,2\alpha}(t,T;H)$.
Then there exists a constant $C>0$ depending on $\delta$, $K_1$, $\alpha$ and $L$
such that  for
$t\in\lbrack0,T],$
\begin{align*}
&  \mathbb{E}\Big[\sup_{s\in\lbrack t,T]}\left\Vert z(s)\right\Vert _{H}^{2\alpha
}\Big|\mathcal{F}_{t}\Big]+\mathbb{E}\Big[\Big(\int_{t}^{T}\Vert z(s)\Vert
_{\mathcal{V}}^{2}\,{d}s\Big)^{\alpha}\Big|\mathcal{F}_{t}\Big]\\
&  \leq C\Big\{\Vert z_{0}\Vert_{H}^{2\alpha}+\mathbb{E}\Big[\Big(\int_{t}^{T}\Vert
\tilde{a}(s,0)\Vert_{H}ds\Big)^{2\alpha}+\Big(\int_{t}^{T}\Vert\tilde{b}%
(s,0)\Vert_H^{2}ds\Big)^{\alpha}\Big|\mathcal{F}_{t}\Big]\Big\},\quad P\text{-a.s.};
\end{align*}

\item[(ii)]Suppose $(\tilde{a},\tilde{b}):[0,T]\times\Omega\times H\rightarrow
\mathcal{V}^{\ast}\times H$ satisfies: 
$\tilde{a}(\cdot,\cdot,z)$, $\tilde{b}(\cdot,\cdot,z)$ are progressively
measurable for each $z\in H,$ there exists a constant $L>0$ such that for all 
$(t,\omega)\in\lbrack0,T]\times\Omega$ and $z,z^{\prime}\in H$,
\[
\Vert\tilde{a}(t,z)-\tilde{a}(t,z^{\prime})\Vert_{\mathcal{V}^{\ast}}%
+\Vert\tilde{b}(t,z)-\tilde{b}(t,z^{\prime})\Vert_H\leq
L\Vert z-z^{\prime}\Vert_{H},
\]
$\tilde{a}(\cdot,\cdot,0)\in L_{\mathbb{F}}^{2,2\alpha}(t,T;\mathcal{V}^{\ast
})$, $\tilde{b}(\cdot,\cdot,0)\in L_{\mathbb{F}}^{2,2\alpha}(t,T;H)$. Then there is a constant $C>0$ depending on   $\delta$, $K_1$, $\alpha$ and $L$ such that  for
$t\in\lbrack0,T],$
\begin{align*}
&  \mathbb{E}\Big[\sup_{s\in\lbrack t,T]}\left\Vert z(s)\right\Vert
_{H}^{2\alpha}\Big|\mathcal{F}_{t}\Big]+\mathbb{E}\Big[\Big(\int_{t}^{T}\Vert
z(s)\Vert_{\mathcal{V}}^{2}\,{d}s\Big)^{\alpha}\Big|\mathcal{F}_{t}\Big]\\
&  \leq C\Big\{\Vert z_{0}\Vert_{H}^{2\alpha}+\mathbb{E}\Big[\Big(\int_{t}%
^{T}\Vert\tilde{a}(s,0)\Vert_{\mathcal{V}^{\ast}}^{2}ds\Big)^{\alpha
}+\Big(\int_{t}^{T}\Vert\tilde{b}(s,0)\Vert_H%
^{2}ds\Big)^{\alpha}\Big|\mathcal{F}_{t}\Big]\Big\},\quad P\text{-a.s.}
\end{align*}

\end{itemize}
\end{lemma}

The following is a continuity estimate.
\begin{lemma}
\label{conti-est-1} Suppose $(H4)$, the conditions in Lemma \ref{apriori-see} (ii),
 $\tilde{a}%
(\cdot,\cdot,0)$ and $\tilde{b}(\cdot,\cdot,0)$ are bounded, and  $z_{0}\in L^{2\alpha}(\mathcal{F}_{t},\mathcal{V})$. Then there exists a constant $C>0$ depending on $\delta$, $K_1$, $\alpha$ and $L$  such that for
$t\in\lbrack0,T)$ and $\rho\leq T-t,$
\[
\mathbb{E}\Big[\sup_{t\leq s\leq t+\rho}\left\Vert z(s)-z_{0}\right\Vert
_{H}^{2\alpha}\Big|\mathcal{F}_{t}\Big]\leq C\big(1+\left\Vert z_{0}%
\right\Vert _{\mathcal{V}}^{2\alpha}\big)\rho^{\alpha},\quad P\text{-a.s.}
\]
\end{lemma}

\begin{proof}
We denote
$\hat{z}(s):=z(s)-z_{0},\ s\in\lbrack t,T].$
On $[t,T],$ we have
\begin{align*}
\hat{z}(s)  &  =\int_{t}^{s}\b[A(r)z(r)+\tilde{a}(r,z(r))\b]{d}r+\int_{t}%
^{s}\b[B(r)z(r)+\tilde{b}(r,z(r))\b]{d}w(r)\\
&  =\int_{t}^{s}\b[A(r)\hat{z}(r)+Az_{0}+\tilde{a}(r,\hat{z}(r)+z_{0})\b]{d}%
r+\int_{t}^{s}\b[B(r)\hat{z}(r)+Bz_{0}+\tilde{b}(r,\hat{z}(r)+z_{0})\b]{d}w(r).
\end{align*}
Then from  Lemma \ref{apriori-see} (ii),
\begin{align*}
\mathbb{E}\Big[\sup_{t\leq s\leq t+\rho}\left\Vert \hat{z}(s)\right\Vert
_{H}^{2\alpha}ds\Big|\mathcal{F}_{t}\Big]  &  \leq C\mathbb{E}\Big[\Big (\int%
_{t}^{t+\rho}\Vert A(r)z_{0}+\tilde{a}(r,z_{0})\Vert_{\mathcal{V}^{\ast}}%
^{2}{dr}\Big)^{\alpha}\Big|\mathcal{F}_{t}\Big]\\
& \ \ \ +C\mathbb{E}\Big[\Big (\int_{t}^{t+\rho}\Vert B(r)z_{0}+\tilde{b}%
(r,z_{0})\Vert_H^{2}{dr}\Big)^{\alpha}\Big|\mathcal{F}%
_{t}\Big]\\
&  \leq C\rho^{\alpha-1}\int_{t}^{t+\rho}(1+\left\Vert z_{0}\right\Vert
_{\mathcal{V}}^{2\alpha}){d}r\\
&  = C(1+\left\Vert z_{0}\right\Vert _{\mathcal{V}}^{2\alpha})\rho^{\alpha
}, \quad P\text{-a.s.},
\end{align*}
which completes the proof.
\end{proof}
\begin{remark}
\label{Rm3-6}
Due to the infinite-dimensional setting, establishing such an estimate appears difficult when $z_{0} \in L^{2\alpha}(\mathcal{F}_{t},H)$. In this situation, one can only show that
\[
\mathbb{E}\Big[\sup_{t\leq s\leq t+\rho}\left\Vert z(s)-z_{0}\right\Vert
_{H}^{2\alpha}\Big|\mathcal{F}_{t}\Big]\rightarrow0, \quad \text{as
}\rho\downarrow0, \ P\text{-a.s.},
\]
which follows directly from the continuity of $z$ in  $s$ under the
$H$-norm  together with the dominated convergence theorem.
\end{remark}

\subsection{Stochastic maximum principle}
Consider the controlled SEE:
\begin{equation}\label{State-1}
	\left\{
	\begin{aligned}
		dX(t)= &\,\, \big[A(t)X(t)+a(t,X(t),u(t))\big]dt + \big[B(t)X(t)+b(t,X(t),u(t))\big]dw(t),  \quad t\in [0,T], \\[4pt]
		X(0)= &\,\, x.
	\end{aligned}
	\right.
\end{equation}
where $x\in H$, $(A,B):[0,T]\rightarrow\mathfrak{L}(\mathcal{V}%
;\mathcal{V}^{\ast}\times H)$ are unbounded linear
operators satisfying the coercivity and quasi-skew-symmetry condition $(H4)$,
while $(a,b):[0,T]\times H\times U\rightarrow H\times H$
are nonlinear functions. The cost functional is defined as
\[
J(x;u(\cdot)):=Y(0),
\]
where $Y$ is the recursive utility governed by a BSDE:%
\begin{equation}
Y(t)=h(X(T))+\int_{t}^{T}k(s,X(s),Y(s),Z(s),u(s))ds-\int_{t}^{T}Z(s)dw({s}),\quad 0\leq t\leq T,
\label{utibsde}%
\end{equation}
with $k:[0,T]\times H\times\mathbb{R}\times\mathbb{R}\times U\rightarrow\mathbb{R}$ and $h:H\rightarrow\mathbb{R}$. The
control domain $U$ is a separable metric space with distance ${d}(\cdot
,\cdot)$  and length $|u|_{U}:={d}(u,0)$ for any fixed $0\in U$. The admissible control set is
\[
\mathcal{U}[0,T]:=\Big\{u:[0,T]\times\Omega\rightarrow U\ \text{is
progressively measurable and}\ {\mathbb{E}}\Big[\int_{0}^{T}%
|u(t)|_{U}^{\alpha}{d}t\Big]<\infty,\text{ for each }\alpha\geq1\Big\}.
\]
For $t\in (0,T),$ we define similarly $\mathcal{U}[t,T].$

The objective of the optimal control problem $(S_{x})$ is to find an admissible control $\bar{u}(\cdot)\in\mathcal{U}[0,T]$ that minimizes the cost functional
$J(x;u(\cdot))$:
\[
J(x;\bar{u}(\cdot))=\inf_{u(\cdot)\in\mathcal{U}[0,T]}J(x;u(\cdot)).
\]

We assume the following conditions for $a$, $b$, $h$ and $k$.

\begin{description}
\item[$(H5)$] $a$, $b$, $h$, $k$ are twice Fr\'{e}chet differentiable with
respect to $(x,y,z)$; $a$, $b$, $k$, $a_{x}$, $b_{x}$, $Dk$, $a_{xx}$,
$b_{xx}$, $D^{2}k$ are continuous in $(x,y,z,u)$, where $Dk$ and $D^{2}k$ are
the Gradient and Hessian of $k$ with respect to $(x,y,z)$,
respectively; $a_{x}$, $b_{x}$, $Dk$, $a_{xx}$, $b_{xx}$, $D^{2}k$, $h_{xx}$
are bounded; $a$, $b$ are bounded by $L_1(1+\Vert x\Vert_{H}+|u|_{U})$ and $k$
is bounded by $L_2(1+\Vert x\Vert_{H}+|y|+| z|+|u|_{U})$, for some constants $L_1,L_2>0$.
\end{description}

For $\psi=a,b,a_{x},b_{x},a_{xx}%
,b_{xx}$ and $v\in U$, define
\[
\bar{\psi}(t):=\psi(t,\bar{X}(t),\bar{u}(t)),\ \delta\psi(t;v):=\psi
(t,\bar{X}(t),v)-\bar{\psi}(t),\  \bar{A}:=A+\bar{a}_{x},\ \bar{B}:=B+\bar{b}_{x}.
\]
To derive the maximum principle, we introduce the following first-order $H$-valued adjoint backward stochastic
evolution equation (BSEE for short):
\begin{equation}\label{adjoint1}
	\left\{
	\begin{aligned}
		-dp(t)=&\,\, \Big\{ \big[\bar{A}^{\ast}(t)+k_{y}(t)+k_{z}(t)\bar{B}^{\ast}(t)\big]p(t)
		+ \big[\bar{B}^{\ast}(t)+k_{z}(t)\big]q(t) + k_{x}(t) \Big\} dt \\
		&\quad - q(t)\,dw(t), \\[4pt]
		p(T)=&\,\, h_{x}(\bar{X}(T)),
	\end{aligned}
	\right.
\end{equation}
and the following second-order $\mathfrak{L}(H)$-valued adjoint BSIE:
\begin{equation}
P(t)=\mathbb{E}\Big[\tilde{L}^{\ast}(t,T)h_{xx}(\bar{X}(T))\tilde{L}%
(t,T)+\int_{t}^{T}\tilde{L}^{\ast}(t,s)(k_{y}(s)P(s)+G(s))\tilde
{L}(t,s)ds\Big|\mathcal{F}_{t}\Big],\quad0\leq t\leq T, \label{adjoint2}%
\end{equation}
where%
\begin{align*}
&  \phi(t):=\phi(t,\bar{X}(t),\bar{Y}(t),\bar{Z}(t),\bar{u}(t)),\quad
\text{for}\ \phi=k_{x},k_{y,},k_{z},D^{2}k,\\
&  \tilde{L}(t,s):=L_{\tilde{A},\tilde{B}}(t,s),\quad\text{for}\ \ \tilde
{A}(s):=\bar{A}(s)+\frac{k_{z}(s)}{2}\bar{B}(s)-\frac{(k_{z}(s))^{2}}{8}%
I_{d}\ \ \text{and}\ \ \tilde{B}(s):=\bar{B}(s)+\frac{k_{z}(s)}{2}I_{d},\\
&  G(t):=D^{2}k(t)\Big(\big[I_{d},p(t),\bar{B}^{\ast}%
(t)p(t)+q(t)\big],\big[I_{d},p(t),\bar{B}^{\ast}(t)p(t)+q(t)\big]\Big)+\langle
p(t),\bar{a}_{xx}(t)\rangle\\
&  \ \ \ \ \ \ \ \ \ \ +k_{z}(t){\langle p(t),\bar{b}_{xx}(t)\rangle}
+\langle q(t),\bar{b}_{xx}(t)\rangle.
\end{align*}

\begin{theorem}[{\cite[Theorem 3.2]{LT21}}]
\label{SMP} Suppose $(H4)$ and $(H5)$. Let $\bar{X}(\cdot)\in L_{\mathbb{F}%
}^{2}(0,T;V)$ and $(\bar{Y}(\cdot),\bar{Z}(\cdot))\in L_{\mathbb{F}}%
^{2}(0,T;\mathbb{R}\times \mathbb{R})$ be the
solutions of SEE (\ref{SEE1-1}) and BSDE (\ref{BSDE1-2}) corresponding to the
optimal control $\bar{u}(\cdot)$. Let $(p,q)\in L_{\mathbb{F}}^{2}%
(0,T;V\times H)$ and $P\in L_{\mathbb{F},w}^{2}%
(0,T;\mathfrak{L}(H))$ be the solutions of BSEE (\ref{adjoint1}) and BSIE
(\ref{adjoint2}), respectively. Then
\begin{equation}%
\begin{split}
&  \inf_{v\in U}\Big\{\mathcal{H}\big(t,\bar{X}(t),\bar{Y}(t),\bar
{Z}(t),v,p(t),q(t)\big)-\mathcal{H}\big(t,\bar{X}(t),\bar{Y}(t),\bar
{Z}(t),\bar{u}(t),p(t),q(t)\big)\\
&  \ \ \ \ \ \,+\frac{1}{2}\big\langle P(t)(b(t,\bar{X}(t),v)-b(t,\bar
{x}(t),\bar{u}(t))),b(t,\bar{X}(t),v)-b(t,\bar{X}(t),\bar{u}%
(t))\big\rangle\Big\}=0,\quad P\text{-a.s. a.e.,}%
\end{split}
\label{MP}%
\end{equation}
where the Hamiltonian
\begin{align*}
\mathcal{H}(t,x,y,z,v,p,q)  &  :=\langle p,a(t,x,v)\rangle+\langle
q,b(t,x,v)\rangle+k(t,x,y,z+{\langle
p,b(t,x,v)-b(t,\bar{X}(t),\bar{u}(t))\rangle},v),\\
&  \quad \quad\quad  (t,x,y,z,v,p,q)\in
\lbrack0,T]\times H\times\mathbb{R}\times \mathbb{R}\times U\times
H\times H.
\end{align*}

\end{theorem}

\begin{remark}
Because the coefficients of the BSIE \eqref{adjoint2} are self-adjoint, it follows that $P$ is self-adjoint and thus lies in $\mathcal{S}(H)$,  the space of bounded self-adjoint linear operators on $H$.
\end{remark}

\subsection{Dynamic programming principle}

Given any  $(t,\xi)\in\lbrack0,T]\times L^{2}(\mathcal{F}_{t},H),$ we consider the following
controlled SEE with different initial time and values:
\begin{equation}\label{SEE1-1}
	\begin{cases}
		dX^{t,\xi;u}(s)
		= \big[A(s)X^{t,\xi;u}(s) + a(s, X^{t,\xi;u}(s), u(s))\big]\,ds \\
		\qquad\quad
		+ \big[B(s)X^{t,\xi;u}(s) + b(s, X^{t,\xi;u}(s), u(s))\big]\,dw(s), \quad s\in [t,T],\\[4pt]
		X^{t,\xi;u}(t) = \xi .
	\end{cases}
\end{equation}
The cost functional is defined by
\[
J(t,\xi;u(\cdot)):=Y^{t,\xi;u}(t),
\]
with $Y^{t,\xi;u}$ solves the BSDE
\begin{equation}
Y^{t,\xi;u}(s)=h(X^{t,\xi;u}(T))+\int_{s}^{T}k(r,X^{t,\xi;u}(r),Y^{t,x;u}%
(r),Z^{t,\xi;u}(r),u(r))dr-\int_{s}^{T}Z^{t,\xi;u}(r)dw({r}),\quad 
s\in\lbrack t,T]. \label{BSDE1-2}%
\end{equation}

For each given $(t,x)\in\lbrack0,T]\times H$, the goal of the optimal control problem  $(S_{t,x})$ is to find an
admissible control $\bar{u}(\cdot)$ such that the cost functional
$J(t,x;\bar{u}(\cdot))$ is minimized at $\bar{u}(\cdot)$ over the control set
$\mathcal{U}[t,T]:$
\[
J(t,x;\bar{u}(\cdot))=\underset{u(\cdot)\in\mathcal{U}[t,T]}{ess\inf
}J(t,x;u(\cdot)).
\]
Note that when $t=0,$ $(S_{t,x})$ reduces to $(S_{x}),$ and $(X^{t,x;u}%
(s),Y^{t,x;u}(s),Z^{t,x;u}(s))=(X(s),Y(s),Z(s)),$ which are defined in the last subsection.

We define the value function
\begin{equation}
\label{value-f}V(t,x):=\underset{u(\cdot)\in\mathcal{U}[t,T]}{ess\inf
}J(t,x;u(\cdot)),\quad (t,x)\in\lbrack0,T]\times H.
\end{equation}

We denote by $\mathcal{U}^{t}[t,T]$ the space of all $U$-valued $(\mathcal{F}%
_{s}^{t})_{t\leq s\leq T}$-progressively measurable processes in $\mathcal{U}[t,T]$,
with $\mathcal{F}_{s}^{t}$ being the augmented natural filtration of
$(w(s)-w(t))_{s\geq t}$. Then for each $u\in\mathcal{U}^{t}[t,T],$ it is easy
to verify that the solution $(X^{t,x;u}(s),Y^{t,x;u}(s),Z^{t,x;u}(s))_{t\leq
s\leq T}$ of the system is $(\mathcal{F}_{s}^{t})_{t\leq s\leq T}$-adapted. In
particular, $Y^{t,x;u}(t)\in\mathcal{F}_{t}^{t},$ so it is deterministic. From
standard arguments (see Proposition \ref{Le3-2} for more details), we can see that  $V(t,x)$ is
a deterministic function.

Given a positive constant $\delta\leq T-t$ and
control $u\in\mathcal{U}[t,t+\delta],$ for each $\eta\in L^{2}(\mathcal{F}%
_{t+\delta}),$ we define the backward semigroup%
\[
G_{t,t+\delta}^{t,x;u}[\eta]:= \tilde Y(t),
\]
where $(\tilde Y(s),\tilde Z(s))$ solves the following BSDE
\begin{equation}
\label{EQ6}\tilde Y(s)=\eta+\int_{s}^{t+\delta}k(r,X^{t,x;u}
(r),\tilde Y(r),\tilde Z(r),u(r))dr-\int_{s}^{t+\delta}\tilde Z(r)dw({r}),\quad s\in\lbrack
t,t+\delta].
\end{equation}

We make the following assumption.
\begin{description}
\item[$(H6)$] $a,b,k,h$ are continuous in $(t,x,y,z,v)$ and satisfy for some constants $L_3,L_4,L_5>0$  that
\begin{align*}
&\big\Vert a(t,x_{1},v)-a(t,x_{2},v)\big\Vert_{H}+\big\Vert b(t,x_{1},u)-b(t,x_{2},v)\big\Vert_H\leq L_3\Vert x_{1}-x_{2}\Vert_{H},
\\
&\big|k(t,x_{1},y_{1},z_{1},v)-k(t,x_{2},y_{2},z_{2},v)\big|\leq L_3\big(\Vert x_{1}%
-x_{2}\Vert_{H}+|y_{1}-y_{2}|+| z_{1}-z_{2}|\big),
\\
&\big\Vert h(x_{1})-h(x_{2})\big\Vert_{H}\leq L_3\Vert x_{1}-x_{2}\Vert_{H},
\end{align*}
for all $t\in\lbrack0,T],(x_{i},y_{i},z_{i})\in H\times\mathbb{R}%
\times\mathbb{R},$ $i=1,2$, $v\in U$, and the functions $a,b,h$, $k$ are bounded by $L_4(1+\Vert
x\Vert_{H}+|y|+| z|+|v|_{U});$
$a(t,0,v),b(t,0,v),k(t,0,0,0,v)$ are  bounded by $L_5$.
\end{description}

We have the following DPP for Problem
$(S_{t,x})$. The proof is standard and is presented in the Appendix.
\begin{theorem}
\label{DPP} Assume $(H4)$ and $(H6)$ for $(S_{t,x})$. For each $(t,x)\in\lbrack0,T]\times H$
and $0\leq\delta\leq T-t,$ we have
\[
V(t,x)=\underset{u(\cdot)\in\mathcal{U}[t,t+\delta]}{ess\inf}G_{t,t+\delta
}^{t,x;u}[V(t+\delta,X^{t,x;u}(t+\delta))]=\inf_{u(\cdot)\in\mathcal{U}%
^{t}[t,t+\delta]}G_{t,t+\delta}^{t,x;u}[V(t+\delta,X^{t,x;u}(t+\delta))].
\]
\end{theorem}
\begin{remark}
	The DPP for infinite-dimensional stochastic systems in the nonrecursive case, under the weak formulation, was established in \cite{FGS17}. For the recursive case, the DPP in the strong formulation and within the mild-solution framework was obtained in \cite{CL23,TZ24-1}. For completeness, we establish here the DPP in the strong formulation under the variational-solution framework. In contrast to \cite{CL23,TZ24-1}, our state equation allows an unbounded diffusion operator, and all unbounded operators may be time dependent.
\end{remark}
\section{Relationship between MP and DPP}
In this section, we study the relationship between the MP and the DPP in the nonsmooth case.  
To carry out our purpose, we first derive an It\^{o}  formula for the
operator-valued BSIEs.
\subsection{An It\^{o}'s  formula for second adjoint equations}
Given any $t\in\lbrack0,T]$ and $x_{0}\in L^{2\alpha}(\mathcal{F}_{t},H),$
consider the operator-valued BSIE
\begin{equation}
P(s)=\mathbb{E}\Big[\tilde{L}^{\ast}(s,T)\xi\tilde{L}(s,T)+\int_{s}^{T}%
\tilde{L}^{\ast}(s,r)f(r,P(r))\tilde{L}(s,r)dr\Big|\mathcal{F}_{s}\Big],\quad
s\in\lbrack t,T], \label{Myeq4-19}%
\end{equation}
and the forward SEEs in the form of
\begin{equation}\label{Myeq3-9}
	\begin{cases}
		dx(s)
		= \big[A(s)x(s) + \gamma_{1}(s)\big]\,ds
		+ \big[B(s)x(s) + \gamma_{2}(s)\big]\,dw(s),
		\quad s\in[t,T], \\[4pt]
		x(t) = x_{0}.
	\end{cases}
\end{equation}
where, for some $\beta\in L_{\mathbb{F}}^{\infty}(t,T)$,
\[
\tilde{L}(s,r):=L_{\tilde{A},\tilde{B}}(s,r)\quad\text{with}\ \ \tilde
{A}(s):=A(s)+\frac{\beta(s)}{2}B(s)-\frac{\beta^{2}(s)}{8}I_{d}\ \ \text{and}%
\ \ \tilde{B}(s):=B(s)+\frac{\beta(s)}{2}I_{d}.
\]

We have the following It\^{o}  formula.

\begin{theorem}
\label{Myth3-7} Let Assumptions $(H2)$ and $(H4)$ be satisfied and for some
$\alpha>1$,
\[
(\xi,f(\cdot,\cdot,0),x_{0},\gamma_{1},\gamma_{2})\in L_{w}^{2\alpha
}(\mathcal{F}_{T},\mathfrak{L}(H))\times L_{\mathbb{F},w}^{2,2\alpha
}(t,T;\mathfrak{L}(H))\times L^{4\alpha}(\mathcal{F}_{t},H)\times
L_{\mathbb{F}}^{1,4\alpha}(t,T;H)\times L_{\mathbb{F}}^{2,4\alpha
}(t,T; H). \label{Myeq3-13}%
\]
Then, for $s\in\lbrack t,T]$,
\begin{equation}
\langle P(s)x(s),x(s)\rangle+\sigma(s)=\langle\xi x(T),x(T)\rangle+\int%
_{s}^{T}\big[\langle f(r,P(r))x(r),x(r)\rangle+\beta(r)\mathcal{Z}%
(r)\big ]ds-\int_{s}^{T}\mathcal{Z}(r)dw({r}), \label{Myeq3-5}%
\end{equation}
for a unique couple of processes $(\sigma,\mathcal{Z})\in L_{\mathbb{F}%
}^{\alpha}(t,T)\times L_{\mathbb{F}}^{2,\alpha}(t,T)$ satisfying
\[
  \sup_{t\leq s\leq T}\big(\mathbb{E}\big[|\sigma(s)|^{\alpha}%
\big|\mathcal{F}_{t}\big]\big)^{\frac{1}{\alpha}}\leq M(t)\mu
_{1}(t) \ \text{and}\ \ \Big\{\mathbb{E}\Big[\Big(\int_{t}^{T}
|\mathcal{Z}(s)|^{2}
ds\Big)^{\frac{\alpha}{2}%
}
ds\Big|\mathcal{F}_{t}\Big]\Big\}^{\frac{1}{\alpha}}\leq
M(t)\mu_{2}(t),\quad P\text{-a.s.},
\]
where
\begin{align*}
\mu_{1}(t)  &  :=\Big\{\mathbb{E}\Big[\Big(\int_{t}^{T}\left\Vert \gamma
_{1}(s)\right\Vert _{H}ds\Big)^{4\alpha}\Big|\mathcal{F}_{t}\Big]+\mathbb{E}%
\Big[\Big(\int_{t}^{T}\left\Vert \gamma_{2}(s)\right\Vert _H^{2}ds\Big)^{2\alpha}\Big|\mathcal{F}_{t}\Big]\Big\}^{\frac
{1}{2\alpha}}\\
&  +\Big\{\mathbb{E}\Big[\Big(\int_{t}^{T}\left\Vert \gamma_{1}(s)\right\Vert
_{H}ds\Big)^{4\alpha}\Big|\mathcal{F}_{t}\Big]+\mathbb{E}\Big[\Big(\int%
_{t}^{T}\left\Vert \gamma_{2}(s)\right\Vert _H%
^{2}ds\Big)^{2\alpha}\Big|\mathcal{F}_{t}\Big]\Big\}^{\frac{1}{4\alpha}%
}\left\Vert x_{0}\right\Vert _{H},
\end{align*}%
\[
\mu_{2}(t):=\left\Vert x_{0}\right\Vert _{H}^{2}+\Big\{\mathbb{E}%
\Big[\Big(\int_{t}^{T}\left\Vert \gamma_{1}(s)\right\Vert _{H}ds\Big)^{4\alpha
}\Big|\mathcal{F}_{t}\Big]+\mathbb{E}\Big[\Big(\int_{t}^{T}\left\Vert
\gamma_{2}(s)\right\Vert _H^{2}ds\Big)^{2\alpha
}\Big|\mathcal{F}_{t}\Big]\Big\}^{\frac{1}{2\alpha}},
\]
and
\[
M(t):=C\Big\{\big(\mathbb{E}[\Vert\xi\Vert_{\mathfrak{L}%
(H)}^{2\alpha}|\mathcal{F}_{t}]\big)^{\frac{1}{2\alpha}}+\Big(\mathbb{E}%
\Big[\Big(\int_{t}^{T}\Vert f(s,0)\Vert_{\mathfrak{L}(H)}^{2}ds\Big)^{\alpha
}\Big|\mathcal{F}_{t}\Big]\Big)^{\frac{1}{2\alpha}}\Big\}.
\]

\end{theorem}

\begin{proof}
We introduce a new SEE
\begin{equation}
	\left\{
	\begin{aligned}
		{d}\tilde{x}(s) = &\,\, A(s)\tilde{x}(s){d}s+B(s)\tilde{x}(s){d}w(s),\quad
		s\in\lbrack t,T], \\
		x(t)= &x_{0}.
	\end{aligned}
	\right.
\end{equation}
Then from Assertion (i) of Lemma \ref{apriori-see},
\begin{equation}
\mathbb{E}\Big[\sup_{s\in\lbrack t,T]}\left\Vert x(s)-\tilde{x}(s)\right\Vert
_{H}^{4\alpha}\Big|\mathcal{F}_{t}\Big]\leq C\Big\{\mathbb{E}\Big[\Big(\int%
_{t}^{T}\left\Vert \gamma_{1}(s)\right\Vert _{H}ds\Big)^{4\alpha
}\Big|\mathcal{F}_{t}\Big]+\mathbb{E}\Big[\Big(\int_{t}^{T}\left\Vert
\gamma_{2}(s)\right\Vert _H^{2}ds\Big)^{2\alpha
}\Big|\mathcal{F}_{t}\Big]\Big\}, \label{Eq3-17}%
\end{equation}%
\begin{equation}
\mathbb{E}\Big[\sup_{s\in\lbrack t,T]}\left\Vert x(s)\right\Vert _{H}%
^{4\alpha}\Big|\mathcal{F}_{t}\Big]\leq C\Big\{\left\Vert x_{0}\right\Vert
_{H}^{4\alpha}+\mathbb{E}\Big[\Big(\int_{t}^{T}\left\Vert \gamma
_{1}(s)\right\Vert _{H}ds\Big)^{4\alpha}\Big|\mathcal{F}_{t}\Big]+\mathbb{E}%
\Big[\Big(\int_{t}^{T}\left\Vert \gamma_{2}(s)\right\Vert _H^{2}ds\Big)^{2\alpha}\Big|\mathcal{F}_{t}\Big]\Big\}, \label{Eq3-18}%
\end{equation}
\begin{equation}
\mathbb{E}\Big[\sup_{s\in\lbrack t,T]}\left\Vert \tilde{x}(s)\right\Vert
_{H}^{4\alpha}\Big|\mathcal{F}_{t}\Big]\leq C\left\Vert x_{0}\right\Vert
_{H}^{4\alpha}. \label{Eq3-19}%
\end{equation}

We denote
$
\lambda(s):=e^{-\int_{0}^{s}\frac{1}{2}{\beta^{2}(r)}dr+\int_{0}^{s}\beta(r)dw({r})}.
$
For any $t\leq s\leq r\leq T,$ according to  \cite[Lemma 4.3]{LT21}, 
\[
\tilde{L}(s,r)=\frac{\lambda_{1}(r)}{\lambda_{1}(s)}L(s,r),
\]
with $L(s,r):=L_{A,B}(s,r) $ and $\lambda_{1}(s):=e^{-\int_{0}^{s}\frac{1}%
{4}\beta^{2}(r)dr+\int_{0}^{s}\frac{1}{2}\beta(r)dw({r})}.$ Then, for  
$s\in\lbrack t,T]$,
\[
P(s)=\mathbb{E}\Big[\frac{\lambda(T)}{\lambda(s)}L^{\ast}(s,T)\xi
L(s,T)+\int_{s}^{T}\frac{\lambda(r)}{\lambda(s)}L^{\ast}(s,r)f(r,P(r))L^{\ast
}(s,r)dr\Big|\mathcal{F}_{s}\Big].
\]
From the definition of $L(s,r),$ we have $L(s,r)\tilde{x}(s)=\tilde{x}(r)$.
Thus,
\begin{align*}
\langle P(s)\tilde{x}(s),\tilde{x}(s)\rangle &  =\mathbb{E}\Big[\frac
{\lambda(T)}{\lambda(s)}\langle\xi L(s,T)\tilde{x}(s),L(s,T)\tilde
{x}(s)\rangle+\int_{s}^{T}\frac{\lambda(r)}{\lambda(s)}\langle
f(r,P(s))L(s,r)\tilde{x}(s),L(s,r)\tilde{x}(s)\rangle dr\Big|\mathcal{F}%
_{s}\Big]\label{Myeq5-6}\\
&  =\mathbb{E}\Big[\frac{\lambda(T)}{\lambda(s)}\langle\xi\tilde{x}%
(T),\tilde{x}(T)\rangle+\int_{s}^{T}\frac{\lambda(r)}{\lambda(s)}\langle
f(r,P(s))\tilde{x}(r),\tilde{x}(r)\rangle dr\Big|\mathcal{F}_{s}\Big].
\end{align*}
Then
\begin{equation}
\langle P(s)x(s),x(s)\rangle+\sigma(s)=\mathbb{E}\Big[\frac{\lambda
(T)}{\lambda(s)}\langle\xi x(T),x(T)\rangle+\int_{s}^{T}\frac{\lambda
(r)}{\lambda(s)}\langle f(r,P(s))x(r),x(r)\rangle dr\Big|\mathcal{F}_{s}\Big],
\label{eq4-9}%
\end{equation}
where
\begin{align*}
\sigma(s)    =&\Big(\mathbb{E}\Big[\frac{\lambda(T)}{\lambda(s)}\langle\xi
x(T),x(T)\rangle\Big|\mathcal{F}_{s}\Big]-\mathbb{E}\Big[\frac{\lambda
(T)}{\lambda(s)}\langle\xi\tilde{x}(T),\tilde{x}(T)\rangle\Big|\mathcal{F}%
_{s}\Big]\Big)\\
&  +\Big(\mathbb{E}\Big[\int_{s}^{T}\frac{\lambda(r)}{\lambda(s)}\langle
f(r,P(s))x(r),x(r)\rangle dr\Big|\mathcal{F}_{s}\Big]-\mathbb{E}\Big[\int%
_{s}^{T}\frac{\lambda(r)}{\lambda(s)}\langle f(r,P(s))\tilde{x}(r),\tilde
{x}(r)\rangle dr\Big|\mathcal{F}_{s}\Big]\Big)\\
&  -\big(\langle P(s)x(s),x(s)\rangle-\langle P(s)\tilde{x}(s),\tilde
{x}(s)\rangle\big)\\
  =&:I_{1}(s)+I_{2}(s)+I_{3}(s).
\end{align*}
Note that, with denoting by $\alpha^{\prime}$ the H\"{o}lder conjugate of
$\alpha,$
\begin{align*}
&  \mathbb{E}\Big[\Big|\frac{\lambda(T)}{\lambda(s)}\Big|\b|\langle\xi
x(T),x(T)\rangle-\langle\xi\tilde{x}(T),\tilde{x}(T)\rangle\b|\Big|\mathcal{F}%
_{s}\Big]\\
&  \leq\Big(\mathbb{E}\Big[\Big|\frac{\lambda(T)}{\lambda(s)}\Big|^{\alpha
^{\prime}}\Big|\mathcal{F}_{s}\Big]\Big)^{\frac{1}{\alpha^{\prime}}%
}\Big(\mathbb{E}\Big[\b|\langle\xi x(T),x(T)\rangle-\langle\xi\tilde
{x}(T),\tilde{x}(T)\rangle\b|^{\alpha}\Big|\mathcal{F}_{s}\Big]\Big)^{\frac
{1}{\alpha}}\\
&  \leq C\Big(\mathbb{E}\Big[\b|\langle\xi x(T),x(T)\rangle-\langle\xi\tilde
{x}(T),\tilde{x}(T)\rangle\b|^{\alpha}\Big|\mathcal{F}_{s}\Big]\Big)^{\frac
{1}{\alpha}}.
\end{align*}
Then from (\ref{Eq3-17}), (\ref{Eq3-18}) and (\ref{Eq3-19}),
\begin{align*}
&  \big(\mathbb{E}[|I_{1}(s)|^{\alpha}|\mathcal{F}_{t}%
]\big)^{\frac{1}{\alpha}}\leq C\Big(\mathbb{E}\Big[\b|\langle\xi
x(T),x(T)\rangle-\langle\xi\tilde{x}(T),\tilde{x}(T)\rangle\b|^{\alpha
}\Big|\mathcal{F}_{t}\Big]\Big)^{\frac{1}{\alpha}}\\
&  \ \ \ \ \leq C\big(\mathbb{E}\big[\Vert\xi\Vert_{\mathfrak{L}(H)}^{2\alpha
}\big|\mathcal{F}_{t}\big]\big)^{\frac{1}{2\alpha}}\big(\mathbb{E}[\Vert
x(T)-\tilde{x}(T)\Vert_{H}^{4\alpha}|\mathcal{F}_{t}]\big)^{\frac{1}{4\alpha}%
}\Big\{\big(\mathbb{E}[\Vert x(T)\Vert_{H}^{4\alpha}|\mathcal{F}%
_{t}]\big)^{\frac{1}{4\alpha}}+\big(\mathbb{E}[\Vert\tilde{x}(T)\Vert
_{H}^{4\alpha}\big|\mathcal{F}_{t}]\big)^{\frac{1}{4\alpha}}\Big\}\\
&  \ \ \ \ \leq C\big(\mathbb{E}[\Vert\xi\Vert_{\mathfrak{L}(H)}^{2\alpha
}|\mathcal{F}_{t}]\big)^{\frac{1}{2\alpha}}\Big\{\mathbb{E}\Big[\Big(\int%
_{t}^{T}\left\Vert \gamma_{1}(s)\right\Vert _{H}ds\Big)^{4\alpha
}\Big|\mathcal{F}_{t}\Big]+\mathbb{E}\Big[\Big(\int_{t}^{T}\left\Vert
\gamma_{2}(s)\right\Vert _H^{2}ds\Big)^{2\alpha
}\Big|\mathcal{F}_{t}\Big]\Big\}^{\frac{1}{4\alpha}}\\
&  \ \ \ \quad  \times \Big\{\Big\{\mathbb{E}\Big[\Big(\int_{t}^{T}\left\Vert
\gamma_{1}(s)\right\Vert _{H}ds\Big)^{4\alpha}\Big|\mathcal{F}_{t}%
\Big]+\mathbb{E}\Big[\Big(\int_{t}^{T}\left\Vert \gamma_{2}(s)\right\Vert
_H^{2}ds\Big)^{2\alpha}\Big|\mathcal{F}_{t}%
\Big]\Big\}^{\frac{1}{4\alpha}}+\left\Vert x_{0}\right\Vert _{H}\Big\}\\
&  \ \ \ \ \leq C\big(\mathbb{E}[\Vert\xi\Vert_{\mathfrak{L}(H)}^{2\alpha
}|\mathcal{F}_{t}]\big)^{\frac{1}{2\alpha}}\Big\{\mathbb{E}\Big[\Big(\int%
_{t}^{T}\left\Vert \gamma_{1}(s)\right\Vert _{H}ds\Big)^{4\alpha
}\Big|\mathcal{F}_{t}\Big]+\mathbb{E}\Big[\Big(\int_{t}^{T}\left\Vert
\gamma_{2}(s)\right\Vert _H^{2}ds\Big)^{2\alpha
}\Big|\mathcal{F}_{t}\Big]\Big\}^{\frac{1}{2\alpha}}\\
&  \ \ \ \ \ \ \ +C\big(\mathbb{E}[\Vert\xi\Vert_{\mathfrak{L}(H)}^{2\alpha
}|\mathcal{F}_{t}]\big)^{\frac{1}{2\alpha}}\Big\{\mathbb{E}\Big[\Big(\int%
_{t}^{T}\left\Vert \gamma_{1}(s)\right\Vert _{H}ds\Big)^{4\alpha
}\Big|\mathcal{F}_{t}\Big]+\mathbb{E}\Big[\Big(\int_{t}^{T}\left\Vert
\gamma_{2}(s)\right\Vert _H^{2}ds\Big)^{2\alpha
}\Big|\mathcal{F}_{t}\Big]\Big\}^{\frac{1}{4\alpha}}\left\Vert x_{0}%
\right\Vert _{H}\\
&  \ \ \ \ \leq M(t)\mu_{1}(t).
\end{align*}
Similarly, we also have from Theorem \ref{Mainthm-1} that
\begin{align*}
&  \big(\mathbb{E}[|I_{2}(s)|^{\alpha}|\mathcal{F}_{t}%
]\big)^{\frac{1}{\alpha}}\leq C_{1}\Big(\mathbb{E}\Big[\B(\int_{s}%
^{T}\big|\langle f(r,P(s))x(r),x(r)\rangle dr-\langle f(r,P(s))\tilde
{x}(r),\tilde{x}(r)\rangle\b| dr\Big)^{\alpha}\Big|\mathcal{F}_{t}%
\Big]\Big)^{\frac{1}{\alpha}}\\
&  \ \ \ \ \leq\Big(\mathbb{E}\Big[\Big(\int_{s}^{T}\Vert f(r,P(s))\Vert
_{\mathfrak{L}(H)}^{2}ds\Big)^{\alpha}\Big|\mathcal{F}_{t}\Big]\Big)^{\frac
{1}{2\alpha}}\Big(\mathbb{E}\Big[\int_{s}^{T}\Vert x(r)-\tilde{x}(r)\Vert
_{H}^{4\alpha}ds\Big|\mathcal{F}_{t}\Big]\Big)^{\frac{1}{4\alpha}}\\
&  \ \ \ \ \ \ \quad \times \Big\{\Big(\mathbb{E}\Big[\int_{s}^{T}\Vert
x(r)\Vert_{H}^{4\alpha}ds\Big]\Big)^{\frac{1}{4\alpha}}+\Big(\mathbb{E}%
\Big[\int_{s}^{T}\Vert\tilde{x}(r)\Vert_{H}^{4\alpha}ds\Big|\mathcal{F}%
_{t}\Big]\Big)^{\frac{1}{4\alpha}}\Big\}\\
&  \ \ \ \ \leq M(t)\mu_{1}(t),
\end{align*}
and
\[
\big(\mathbb{E}[|I_{3}(s)|^{\alpha}|\mathcal{F}_{t}%
]\big)^{\frac{1}{\alpha}} \leq M(t)\mu_{1}(t).
\]
Thus,%
\[
\big(\mathbb{E}[|\sigma(s)|^{\alpha}|\mathcal{F}_{t}]\big)^{\frac
{1}{\alpha}}\leq M(t)\mu_{1}(t).
\]

Note that equation (\ref{eq4-9}) is the explicit formula of the linear BSDE
(\ref{Myeq3-5}) with solution $(\langle P(s)x(s),x(s)\rangle+\sigma
(s),\mathcal{Z}(s))\in L_{\mathbb{F}}^{\alpha}(t,T)\times L_{\mathbb{F}%
}^{2,\alpha}(t,T)$. The uniqueness of $(\sigma,\mathcal{Z)}$ in the equation
(\ref{Myeq3-5}) follows from the basic results of BSDEs. On the other hand,
\begin{align*}
&  \big(\mathbb{E}[|\langle\xi x(T),x(T)\rangle|^{\alpha}|\mathcal{F}%
_{t}]\big)^{\frac{1}{\alpha}}\leq C\big(\mathbb{E}[\Vert\xi\Vert
_{\mathfrak{L}(H)}^{2\alpha}|\mathcal{F}_{t}]\big)^{\frac{1}{2\alpha}%
}\big(\mathbb{E}[\Vert x(T)\Vert_{H}^{4\alpha}|\mathcal{F}_{t}]\big)^{\frac
{1}{2\alpha}}\\
&  \ \ \ \ \ \ \ \leq C\big(\mathbb{E}[\Vert\xi\Vert_{\mathfrak{L}%
(H)}^{2\alpha}|\mathcal{F}_{t}]\big)^{\frac{1}{2\alpha}}\Big\{\left\Vert
x_{0}\right\Vert _{H}^{2}+\Big\{\mathbb{E}\Big[\Big(\int_{t}^{T}\left\Vert
\gamma_{1}(s)\right\Vert _{H}ds\Big)^{4\alpha}\B|\mathcal{F}_{t}\Big]+\mathbb{E}%
\Big[\Big(\int_{t}^{T}\left\Vert \gamma_{2}(s)\right\Vert _H^{2}ds\Big)^{2\alpha}\Big|\mathcal{F}_{t}\Big]\Big\}^{\frac{1}{2\alpha}%
}\Big\}\\
&  \ \ \ \ \ \ \ \leq M(t)\mu_{2}(t),
\end{align*}
and similarly,
\[
\Big\{\mathbb{E}\Big[\Big(\int_{s}^{T}\big|\langle f(r,P(r))x(r),x(r)\rangle
\big|ds\Big)^{\alpha}\Big|\mathcal{F}_{t}\Big]\Big\}^{\frac{1}{\alpha}}\leq
M(t)\mu_{2}(t).
\]
Observing that $\mu_{1}(t)\leq2\mu_{2}(t),$ we obtain from the basic
estimates of BSDEs that
\[
\Big\{\mathbb{E}\Big[\Big(\int_{t}^{T}
|\mathcal{Z}(s)|^{2}
ds\Big)^{\frac{\alpha}{2}%
}
\Big|\mathcal{F}_{t}\Big]\Big\}^{\frac{1}{\alpha}}\leq
M(t)\mu_{2}(t).
\]
The proof is complete.
\end{proof}

We also need the following corollary of DPP. 
\begin{lemma}
\label{Le4-4} Assume $(H2)$ and $(H4)$. If $(\bar{X}^{t,x;\bar{u}}(\cdot),\bar{Y}^{t,x;\bar{u}}%
(\cdot),\bar{Z}^{t,x;\bar{u}}(\cdot),\bar{u}(\cdot))$ are optimal for Problem
$(S_{t,x}),$ then for any $\delta\in\lbrack0,T-t],$
\[
V(t+\delta,\bar{X}^{t,x;\bar{u}}(t+\delta))=\bar{Y}^{t,x;\bar{u}}
(t+\delta),\quad P\text{-a.s}.
\]

\end{lemma}

\begin{proof}
First we have
\[
\bar{Y}^{t,x;\bar{u}}(s)=\bar{Y}^{t,x;\bar{u}}(t+\delta)+\int_{s}^{t+\delta
}k(r,X^{t,x;\bar{u}}(r),Y^{t,x;\bar{u}}(r),Z^{t,x;\bar{u}}(r),\bar
{u}(r))dr-\int_{s}^{t+\delta}Z^{t,x;\bar{u}}(r)dw({r}),\text{ }s\in\lbrack
t,t+\delta].
\]
We introduce a BSDE
\begin{align*}
y^{t,x;\bar{u}}(s) =  &  V(t+\delta,\bar{X}^{t,x;\bar{u}}(t+\delta))\\
&  +\int_{s}^{t+\delta}k(r,X^{t,x;\bar{u}}(r),y^{t,x;\bar{u}}(r),z^{t,x;\bar
{u}}(r),\bar{u}(r))dr-\int_{s}^{t+\delta}z^{t,x;\bar{u}}(r)dw({r}),\quad s
\in[t,t+\delta].
\end{align*}
From  (\ref{Myeq4-9})   and Proposition \ref{pro3-5} in the Appendix, we know that
\[
\bar{Y}^{t,x;\bar{u}}(t+\delta)=\bar{Y}^{t+\delta,\bar{X}^{t,x;\bar{u}%
}(t+\delta);\bar{u}}(t+\delta)\geq V(t+\delta,\bar{X}^{t,x;\bar{u}}%
(t+\delta)).
\]
Then by  Theorem \ref{DPP} and comparison theorem of  BSDEs,
\begin{equation}\label{eq4-3}
V(t,x)\leq G_{t,t+\delta}^{t,x;\bar{u}}\big[V(t+\delta,X^{t,x;\bar{u}%
}(t+\delta))\big]=y^{t,x;\bar{u}}(t)\leq\bar{Y}^{t,x;\bar{u}}(t).
\end{equation}
But $V(t,x)=\bar{Y}^{t,x;\bar{u}}(t).$ So, the two inequalities in (\ref{eq4-3}) are in
fact equalities. Thus, from the strict comparison theorem of  BSDEs,
we have
\[
\bar{Y}^{t,x;\bar{u}}(t+\delta)=V(t+\delta,\bar{X}^{t,x;\bar{u}}(t+\delta)),
\]
which completes the proof.
\end{proof}

\subsection{Differential in spatial variable}
In this subsection, we study the relationship between the generalized derivatives of the value function in DPP and the first- and second-order adjoint processes in MP in the spatial variables.

Let us first recall the
notions of super- and subdifferentials.
For $v\in C([0,T]\times H)$ and $(t,x)\in\lbrack0,T)\times H$, the
second-order parabolic partial superdifferential of $v$ with respect to $x$ is
defined as:
\begin{align*}
D_{x}^{2,+}v(t,x)  &  =\Big\{(p,P)\in H\times\mathcal{S}(H)\Big|\\
&  v(t,y)\leq v(t,x)+\langle p,y-x\rangle+\frac{1}{2}\langle
P(y-x),y-x\rangle+o(\Vert y-x\Vert_{H}^{2}),\mbox{ as }y\rightarrow
x\Big\}.
\end{align*}
Similarly, the second-order parabolic partial subdifferential of $v$ with
respect to $x$ is defined as:
\begin{align*}
D_{x}^{2,-}v(t,x)  &  =\Big\{(p,P)\in H\times\mathcal{S}(H)\Big|\\
&  v(t,y)\geq v(t,x)+\langle p,y-x\rangle+\frac{1}{2}\langle
P(y-x),y-x\rangle+o(\Vert y-x\Vert_{H}^{2}),\mbox{ as }y\rightarrow
x\Big\}.
\end{align*}

\begin{theorem}
\label{diff-space} Assume $(H4),(H5)$ and $(H6)$. Suppose $(\bar{X}%
(\cdot),\bar{Y}(\cdot),\bar{Z}(\cdot),\bar{u}(\cdot))$ are the optimal 4-tuple
of Problem $(S_{x})$ and $(p(\cdot),q(\cdot))$ and $P(\cdot)$ are the solutions of
corresponding adjoint equations (\ref{adjoint1})  and (\ref{adjoint2}), respectively.  Let $V\in C([0,T]\times H)$ be defined as
in (\ref{value-f}). Then
\begin{equation}\label{eq3-3-1}
\{p(t)\}\times\lbrack P(t),+\infty)\subset D_{x}^{2,+}V(t,\bar{X}(t)),\quad
t\in\lbrack0,T],\text{ }P\text{-a.s.}%
\end{equation}
and%
\begin{equation}\label{eq3-3-2}
D_{x}^{2,-}V(t,\bar{X}(t))\subset\{p(t)\}\times(-\infty,P(t)],\quad
t\in\lbrack0,T],\text{ }P\text{-a.s.}%
\end{equation}
\end{theorem}

\begin{proof} The proof is lengthy, and consists of five steps.
	
\textbf{Step one: Variational state equations.} Fix any $t\in\lbrack0,T]$ and
$x^{1}\in H,$ let $X^{x^{1}}$ be the solution of the following SEE on $[t,T]:$
\begin{equation}
	\left\{
	\begin{aligned}
		dX^{x^{1}}(s)= &\,\, \big[A(s)X^{x^{1}}(s)+a(s,X^{x^{1}}(s),\bar{u}(s))\big]ds \\
		& +\big[B(s)X^{x^{1}}(s)+b(s,X^{x^{1}}(s),\bar{u}(s))\big]dw(s),\quad s\in[t,T], \\
		X^{x^{1}}(t)= &\,\, x^{1}.
	\end{aligned}
	\right.
	\label{EQ3-1}
\end{equation}

We denote $\hat{x}(s):=X^{x^{1}}(s)-\bar{X}(s),$ $s\in\lbrack t,T].$ In
particular, $\hat{x}(t)=X^{x^{1}}(t)-\bar{X}(t)=x^{1}-\bar{X}(t).$ Then
\begin{align*}
\hat{x}(s)  &  =\hat{x}(t)+\int_{t}^{s}\big[A(r)\hat{x}(r)+a(r,\hat{x}%
(r)+\bar{X}(r),\bar{u}(r))-a(r,\bar{X}(r),\bar{u}(r))\big]{d}r\\
&  +\int_{t}^{s}\big[B(r)\hat{x}(r)+b(r,\hat{x}(r)+\bar{X}(r),\bar
{u}(r))-b(r,\bar{X}(r),\bar{u}(r))\big]{d}w(r).
\end{align*}
From Lemma \ref{apriori-see}, we first have
\begin{equation}
\mathbb{E}\Big[  \sup_{t\leq s\leq T}\left\Vert \hat{x}(s)\right\Vert
_{H}^{2\alpha}\Big|\mathcal{F}_{t}\Big]  \leq C\left\Vert \hat{x}(t)\right\Vert
_{H}^{2\alpha},\quad P\text{-a.s.} \label{eq4-13}%
\end{equation}
With $\bar{A}$ and $\bar{B}$ being defined as in Subsection 2.3,  the
equation of $\hat{x}$ reads
\begin{equation}
\hat{x}(s)=\hat{x}(t)+\int_{t}^{s}[\bar{A}(r)\hat{x}(r)+\varepsilon_{1}%
(r)]{d}r+\int_{t}^{s}[\bar{B}(r)\hat{x}(r)+\varepsilon_{2}(r)]{d}w(r)
\label{var-x-1}%
\end{equation}
and
\begin{equation}%
\begin{split}
\hat{x}(s)=\hat{x}(t)+  &  \int_{t}^{s}[\bar{A}(r)\hat{x}(r)+\frac{1}{2}%
\bar{a}_{xx}(r)(\hat{x}(r),\hat{x}(r))+\varepsilon_{3}(r)]\,{d}r\\
&  +\int_{t}^{s}[\bar{B}(r)\hat{x}(r)+\frac{1}{2}\bar{b}_{xx}(r)(\hat
{x}(r),\hat{x}(r))+\varepsilon_{4}(r)]\,{d}w(r),
\end{split}
\label{var-x-2}%
\end{equation}
where
\begin{align*}
\varepsilon_{1}(r)  &  :=\int_{0}^{1}\big\langle a_{x}\left(  r,\bar
{X}(r)+\mu\hat{x}(r),\bar{u}(r)\right)  -\bar{a}_{x}\left(  r\right)  ,\hat
{x}(r)\big\rangle d\mu,\\
\varepsilon_{2}(r)  &  :=\int_{0}^{1}\left\langle b_{x}\left(  r,\bar
{x}(r)+\mu\hat{x}(r),\bar{u}(r)\right)  -\bar{b}_{x}\left(  r\right)  ,\hat
{x}(r)\right\rangle d\mu,\\
\varepsilon_{3}(r)  &  :=\int_{0}^{1}(1-\mu)\big[  a_{xx}\left(  r,\bar
{x}(r)+\mu\hat{x}(r),\bar{u}(r)\right)  -\bar{a}_{xx}\left(  r\right)
\big]  \left(  \hat{x}(r),\hat{x}(r)\right)  d\mu,\\
\varepsilon_{4}(r)  &  :=\int_{0}^{1}(1-\mu)\left[  b_{xx}\big(  r,\bar
{x}(r)+\mu\hat{x}(r),\bar{u}(r)\big)  -\bar{b}_{xx}\left(  r\right)
\right]  \left(  \hat{x}(r),\hat{x}(r)\right)  d\mu.
\end{align*}

\textbf{Step two: Estimates of higher-order remainders.} We have the following
estimates: for any $\alpha\geq2,$
\begin{equation}%
\begin{split}
\mathbb{E}\Big[\int_{t}^{T}\left\Vert \varepsilon_{1}(r)\right\Vert
_{H}^{\alpha}dr\Big|\mathcal{F}_{t}\Big]  &  =o\left(  \left\Vert \hat
{x}(t)\right\Vert _{H}^{\alpha}\right),\quad P\text{-a.s.,}\\
\mathbb{E}\Big[\int_{t}^{T}\left\Vert \varepsilon_{2}(r)\right\Vert
_{H}^{\alpha}dr\Big|\mathcal{F}_{t}\Big]  &  ={o}\left(  \left\Vert \hat
{x}(t)\right\Vert _{H}^{\alpha}\right) ,\quad P\text{-a.s.,}\\
\mathbb{E}\Big[\int_{t}^{T}\left\Vert \varepsilon_{3}(r)\right\Vert
_{H}^{\alpha}dr\Big|\mathcal{F}_{t}\Big]  &  ={o}(\left\Vert \hat{x}(t)\right\Vert
_{H}^{2\alpha}),\quad  P\text{-a.s.,}\\
\mathbb{E}\Big[\int_{t}^{T}\left\Vert \varepsilon_{4}(r)\right\Vert
_{H}^{\alpha}dr\Big|\mathcal{F}_{t}\Big]  &  ={o}(\left\Vert \hat{x}(t)\right\Vert
_{H}^{2\alpha}),\quad P\text{-a.s.}%
\end{split}
\label{Est-1}%
\end{equation}
We only prove the first and third ones, and the others 
can be derived in a similar way. Applying (\ref{eq4-13}),
\begin{align*}
\mathbb{E}\Big[\int_{t}^{T}\left\Vert \varepsilon_{1}(r)\right\Vert
_{H}^{\alpha}dr\Big|\mathcal{F}_{t}\Big]  &  \leq\int_{t}^{T}\mathbb{E}%
\Big[\int_{0}^{1}\b\Vert a_{x}\left(  r,\bar{X}(r)+\mu\hat{x}(r),\bar
{u}(r)\right)  -\bar{a}_{x}\left(  r\right)  \b\Vert _{H}^{\alpha}%
d\mu\left\Vert \hat{x}(r)\right\Vert _{H}^{\alpha}\Big|\mathcal{F}_{t}\Big]dr\\
&  \leq\int_{t}^{T}\mathbb{E}[\left\Vert \hat{x}(r)\right\Vert _{H}^{2\alpha
}|\mathcal{F}_{t}]dr\\
&  \leq C\left\Vert \hat{x}(t)\right\Vert _{H}^{2\alpha}\\
&  =o(\left\Vert \hat{x}(t)\right\Vert _{H}^{\alpha}),\quad P\text{-a.s.}%
\end{align*}
and
\begin{align*}
&  \mathbb{E}\Big[\int_{t}^{T}\left\Vert \varepsilon_{3}(r)\right\Vert
_{H}^{\alpha}dr\Big|\mathcal{F}_{t}\Big]\\
&  \leq\int_{t}^{T}\mathbb{E}\Big[  \int_{0}^{1}(1-\mu)\big|[  a_{xx}\left(
r,\bar{X}(r)+\mu\hat{x}(r),\bar{u}(r)\right)  -\bar{a}_{xx}\left(  r\right)
]  \left(  \hat{x}(r),\hat{x}(r)\right)  \big|^{\alpha}d\mu\Big|\mathcal{F}%
_{t}\Big]  dr\\
&  \leq\Big(\int_{t}^{T}\mathbb{E}\Big[\int_{0}^{1}\big\Vert a_{xx}\left(
r,\bar{X}(r)+\mu\hat{x}(r),\bar{u}(r)\right)  -\bar{a}_{xx}\left(  r\right)
\big\Vert_{\mathfrak{L}(H)}^{2\alpha}d\mu\Big|\mathcal{F}_{t}\Big]dr\Big)^{\frac{1}%
{2}}\left\Vert \hat{x}(t)\right\Vert _{H}^{2\alpha}\\
&  =o(\left\Vert \hat{x}(t)\right\Vert _{H}^{2\alpha}),\quad P\text{-a.s.}%
\end{align*}

\textbf{Step three: Duality relationship.} Applying It\^{o}'s  formula to
$\left\langle p(r),\hat{x}(r)\right\rangle,$ from (\ref{var-x-2}) we get
\begin{equation}
\langle p(s),\hat{x}(s)\rangle=\langle h_{x}(\bar{X}(T)),\hat{x}%
(T)\rangle+\int_{s}^{T}J_{1}(r)dr-\int_{s}^{T}J_{2}(r)dw({r}),\quad 
s\in\lbrack t,T]. \label{w-adjoint1}%
\end{equation}
where
\begin{align*}
J_{1}(s):=  &  \langle k_{x}(s)+k_{y}(s)p(s)+k_{z}(s)q(s),\hat{x}%
(s)\rangle+k_{z}(s)\langle p(s),\bar{B}(s)\hat{x}(s)\rangle-\langle
p(s),\varepsilon_{3}(s)\rangle-\langle q(s),\varepsilon_{4}(s)\rangle\\
&  -\frac{1}{2}[\langle p(s),(\bar{a}_{xx}(s)(\hat{x}(s),\hat{x}%
(s))\rangle+\langle q(s),\bar{b}_{xx}(s)(\hat{x}(s),\hat{x}(s))\rangle],\\
J_{2}(s):=  &  \langle p(s),\bar{B}(s)(\hat{x}(s))\rangle+\langle q(s),\hat
{x}(s)\rangle+\langle p(s),\varepsilon_{4}(s)\rangle+\frac{1}{2}\langle
p(s),\bar{b}_{xx}(s)(\hat{x}(s),\hat{x}(s))\rangle.
\end{align*}
Next, applying Theorem \ref{Myth3-7} to $P$ and $\hat{x}$ in (\ref{var-x-1}) (with $\gamma_{1}=\varepsilon_{1}$, $\gamma_{2}=\varepsilon_{2}$, $x_{0}=\hat{x}(t)$),
and noting from Step two that, for the quantities $\mu_1,\mu_2$ appearing in Theorem \ref{Myth3-7},
\[
\mu_{1}(t)=o(\left\Vert \hat{x}(t)\right\Vert _{H}^{2})\text{ and }\mu
_{2}(t)=O(\left\Vert \hat{x}(t)\right\Vert _{H}^{2}),\quad P\text{-a.s.},
\]
we obtain
\begin{equation}%
\begin{split}
&  \langle P(s)\hat{x}(s),\hat{x}(s)\rangle+\sigma(s)=\langle h_{xx}(\bar
{X}(T))\hat{x}(T),\hat{x}(T)\rangle+\int_{s}^{T}\b[k_{y}(s)\langle
P(s)\hat{x}(s),\hat{x}(s)\rangle\\
&  \ \ \ \ \ \ \ \ +k_{z}(s)\mathcal{Z}(s)+\langle G(s)\hat{x}(s),\hat
{x}(s)\rangle\b]{d}s-\int_{s}^{T}\mathcal{Z}(s){d}w(s),
\end{split}
\label{w-adjiont2}%
\end{equation}
for some processes $(\sigma,\mathcal{Z})\in L_{\mathbb{F}}^{\alpha}(t,T)\times
L_{\mathbb{F}}^{2,\alpha}(t,T)$ satisfying, for any  $\alpha\geq2$,
\begin{equation}
\sup_{s\in\lbrack t,T]}\mathbb{E}[|\sigma(s)|^{\alpha}|\mathcal{F}%
_{t}]=o(\left\Vert \hat{x}(t)\right\Vert _{H}^{2\alpha})\ \ \text{and}%
\ \ \mathbb{E}\Big[\Big(\int_{t}^{T}|\mathcal{Z}(t)|^{2}dt\Big)^{\frac{\alpha
}{2}}\B|\mathcal{F}_{t}\Big]=O(\left\Vert \hat{x}(t)\right\Vert _{H}^{2\alpha
}),\quad P\text{-a.s.} \label{Myeq5-12}%
\end{equation}
Therefore,
\begin{align*}
&  \langle p(t),\hat{x}(t)\rangle+\frac{1}{2}\langle P(t)\hat{x}(t),\hat
{x}(t)\rangle+\frac{1}{2}\sigma(t)=\langle h_{x}(\bar{X}(T)),\hat{x}%
(T)\rangle\\
&  \ \ \ \ \ \ \ \ +\frac{1}{2}\langle h_{xx}(\bar{X}(T))\hat{x}(T),\hat
{x}(T)\rangle+\int_{t}^{T}I_{1}(s){d}s-\int_{t}^{T}I_{2}(s){d}w(s),
\end{align*}
where
\begin{align*}
I_{1}(s)  &  :=\langle k_{x}(s)+k_{y}(s)p(s)+k_{z}(s)q(s),\hat{x}%
(s)\rangle+k_{z}(s)\langle p(s),\bar{B}(s)\hat{x}(s)\rangle+\frac{1}%
{2}\Big\langle\Big\{k_{y}(s)P(s)\\
&  \ \ \ \ +D^{2}k(s)([I_{d},p(s),\bar{B}^{\ast}(s)p(s)+q(s)],[I_{d}%
,p(s),\bar{B}^{\ast}(s)p(s)+q(s)])+k_{z}(s)\langle p(s),\bar{b}_{xx}%
(s)\rangle\Big\}\hat{x}(s),\hat{x}(s)\Big\rangle\\
&  \ \ \ \ +\frac{1}{2}k_{z}(s)\mathcal{Z}(s)-\langle p(s),\varepsilon
_{3}(s)\rangle-\langle q(s),\varepsilon_{4}(s)\rangle,\\
I_{2}(s)  &  :=\langle p(s),\bar{B}(s)\hat{x}(s)\rangle+\langle q(s),\hat
{x}(s)\rangle+\langle p(s),\varepsilon_{4}(s)\rangle+\frac{1}{2}\langle
p(s),\bar{b}_{xx}(s)(\hat{x}(s),\hat{x}(s))\rangle+\frac{1}{2}\mathcal{Z}(s).
\end{align*}

\textbf{Step four: Variational equation for BSDE.} We denote, on $[t,T]$,
\[
Y^{x^{1}}(s)=h(X^{x^{1}}(T))+\int_{s}^{T}k(r,X^{x^{1}}(r),Y^{x^{1}%
}(r),Z^{x^{1}}(r),\bar{u}(r))dr-\int_{s}^{T}Z^{x^{1}}(r)dw({r}).
\]
Then we have
\[%
\begin{split}
&  \hat{y}(s)-\frac{1}{2}\sigma(s)=h(X^{x^{1}}(T))-h(\bar{X}(T))-\langle
h_{x}(\bar{X}(T)),\hat{x}(T)\rangle-\frac{1}{2}\langle h_{xx}(\bar{X}%
(T))\hat{x}(T),\hat{x}(T)\rangle\\
&  +\int_{s}^{T}\Big\{k(r,X^{x^{1}}(r),Y^{x^{1}}(r),Z^{x^{1}}(r),\bar
{u}(r))-k(r,\bar{X}(r),\bar{Y}(r),\bar{Z}(r),\bar{u}(r))-I_{1}(r)\Big\}dr-\int%
_{s}^{T}\hat{z}(r){d}w(r),
\end{split}
\label{Myeq4-13}%
\]
where
\begin{align*}
\hat{y}(s)  &  :=Y^{x^{1}}(s)-\bar{Y}(s)-\langle p(s),\hat{x}(s)\rangle
-\frac{1}{2}\langle P(s)\hat{x}(s),\hat{x}(s)\rangle,\\
\hat{z}(s)  &  :=Z^{x^{1}}(s)-\bar{Z}(s)-I_{2}(s).
\end{align*}
We denote%
\[
I_{3}(s):=\langle p(s),\hat{x}(s)\rangle+\frac{1}{2}\langle P(s)\hat
{x}(s),\hat{x}(s)\rangle.
\]
From  Taylor's expansion,
\begin{equation}%
\begin{split}
&  \hat{y}(s)-\frac{1}{2}\sigma(s)=J_{4}+\int_{s}^{T}\Big\{\tilde{k}%
_{y}(r)(\hat{y}(r)-\frac{1}{2}\sigma(r))+\tilde{k}_{z}(r)\hat{z}(r)+\frac
{1}{2}J_{5}(r)+\frac{1}{2}\tilde{k}_{y}(r)\sigma(r)\\
&  \ \ \ \ \ \ +k_{z}(r)\langle p(r),\varepsilon_{4}(r)\rangle+\langle
p(r),\varepsilon_{3}(r)\rangle+\langle q(r),\varepsilon_{4}(r)\rangle
\Big\}dr-\int_{s}^{T}\hat{z}(r){d}w(r),
\end{split}
\label{EQ3-4}%
\end{equation}
where%
\begin{align*}
&  \tilde{k}_{y}(s):=\int_{0}^{1}k_{y}(s,X^{x^{1}}(s),\bar{Y}(s)+I_{3}%
(s)+\mu\hat{y}(s),\bar{Z}(s)+I_{2}(s)+\mu\hat{z}(s),\bar{u}(s))d\mu,\\
&  \tilde{k}_{z}(s):=\int_{0}^{1}k_{z}(s,X^{x^{1}}(s),\bar{Y}(s)+I_{3}%
(s)+\mu\hat{y}(s),\bar{Z}(s)+I_{2}(s)+\mu\hat{z}(s),\bar{u}(s))d\mu,\\
&  \tilde{D}^{2}k(s):=2\int_{0}^{1}\int_{0}^{1}\mu D^{2}k(s,\bar{X}(s)+\mu
\nu\hat{x}(s),\bar{Y}(s)+\mu\nu I_{3}(s),\bar{Z}(s)+\mu\nu I_{2}(s),\bar
{u}(s))d\mu d\nu,\\
&  J_{3}(s):=k_{z}(s)\langle p(s),\varepsilon_{4}(s)\rangle+\langle
p(s),\varepsilon_{3}(s)\rangle+\langle q(s),\varepsilon_{4}(s)\rangle,\\
&  J_{4}:=h(X^{x^{1}}(T))-h(\bar{X}(T))-\langle h_{x}(\bar{X}(T)),\hat
{x}(T)\rangle-\frac{1}{2}\langle h_{xx}(\bar{X}(T))\hat{x}(T),\hat
{x}(T)\rangle,\\
&  J_{5}(s):=\tilde{D}^{2}k(s)\b ([\hat{x}(s),I_{3}(s),I_{2}(s)],[\hat
{x}(s),I_{3}(s),I_{2}(s)]\b)\\
&  \ \ \ \ \ \ \ \ \ \ \,-\B\langle D^{2}k(s)\b([I_{d},p(s),\bar{B}^{\ast
}(s)p(s)+q(s)],[I_{d},p(s),\bar{B}^{\ast}(s)p(s)+q(s)]\b)\hat{x}(s),\hat
{x}(s)\B\rangle.
\end{align*}
First, we have
\begin{align*}
  \mathbb{E}[|J_4||\mathcal{F}_{t}] & =\frac{1}{2}\mathbb{E}\Big[\b|\langle(\tilde{h}_{xx}(T)-h_{xx}(\bar
{X}(T)))\hat{x}(T),\hat{x}(T)\rangle\b|\Big|\mathcal{F}_{t}\Big]\\
&  \leq\frac{1}{2}\Big(\mathbb{E}\b[
\Vert \tilde{h}_{xx}(T)-h_{xx}(\bar
{X}(T)) \Vert _{\mathfrak{L}(H)}^{2}
\b|\mathcal{F}_{t}\b]\Big)^{\frac{1}{2}%
}\Big(\mathbb{E}\b[ 
\Vert \hat{x}(T) \Vert
_{H}^{4}\b|\mathcal{F}_{t}\b]\Big)^{\frac{1}{2}}\\
&  =o(\left\Vert \hat{x}(t)\right\Vert _{H}^{2}),\ P\text{-a.s.}%
\end{align*}
where%
\[
\tilde{h}_{xx}(T)=2\int_{0}^{1}\int_{0}^{1}\mu h_{xx}\big(\bar{X}(T)+\mu
\nu\hat{x}(T)\big)d\mu d\nu.
\]
Moreover, it is direct to check that $\mathbb{E}\b[\int_{t}^{T}|J_{3}%
(s)|ds\b|\mathcal{F}_{t}\b]=o(\left\Vert \hat{x}(t)\right\Vert _{H}^{2})$. We
write $J_{5}(s)=J_{6}(s)+J_{7}(s),$ where%
\begin{align*}
J_{6}(s):=  &  \b\langle\tilde{D}^{2}k(s)([I_{d},p(s),\bar{B}^{\ast
}(s)p(s)+q(s)],[I_{d},p(s),\bar{B}^{\ast}(s)p(s)+q(s)])\hat{x}(s),\hat
{x}(s)\b\rangle\\
&  -\b\langle D^{2}k(s)([I_{d},p(s),\bar{B}^{\ast}(s)p(s)+q(s)],[I_{d}%
,p(s),\bar{B}^{\ast}(s)p(s)+q(s)])\hat{x}(s),\hat{x}(s)\b\rangle,\\
J_{7}(s):=  &  \tilde{D}^{2}k(s)\b([\hat{x}(s),I_{3}(s),I_{2}(s)],[\hat
{x}(s),I_{3}(s),I_{2}(s)]\b)\\
&  -\b\langle\tilde{D}^{2}k(s)([I_{d},p(s),\bar{B}^{\ast}(s)p(s)+q(s)],[I_{d}%
,p(s),\bar{B}^{\ast}(s)p(s)+q(s)])\hat{x}(s),\hat{x}(s)\b\rangle.
\end{align*}
We only estimate $J_{6}$ and the treatment for $J_{7}$ is similar. Setting
$$\Vert\tilde{D}^{2}k(s)-D^{2}k(s)\Vert:=\Vert\tilde{D}^{2}k(s)-D^{2}%
k(s)\Vert_{\mathfrak{L}_{2}((H\times\mathbb{R\times R)\times}(H\times
\mathbb{R}\times \mathbb{R});\mathbb{R})},$$
 we then have by (3.17) in \cite{LT21} that
\begin{align*}
&  \mathbb{E}\B[\B(\int_{t}^{T}|J_{6}(s)|ds\B)^{2\alpha}\B |\mathcal{F}%
_{t}\B]\\
\leq &\  C\,\mathbb{E}\B[\B(\int_{t}^{T}\Vert\tilde{D}^{2}%
k(s)-D^{2}k(s)\Vert\left((1+\Vert p(s)\Vert_{H}^{2})\Vert\hat{x}(s)\Vert
_{H}^{2}+\Vert p(s)\Vert_{\mathcal{V}}^{2}\Vert\hat{x}(s)\Vert
_{H}^{2}+\Vert q(s)\Vert_{H}^{2}\Vert\hat{x}(s)\Vert_{H}^{2}%
\right)ds\B)^{2\alpha}\B |\mathcal{F}_{t}\B]\\
\leq&\  C\B(\mathbb{E}\B[\int_{t}^{T}\Vert\tilde{D}%
^{2}k(s)-D^{2}k(s)\Vert^{4\alpha}(1+\Vert p(s)\Vert_{H}^{8\alpha
})ds\B |\mathcal{F}_{t}\B]\B)^{\frac{1}{2}}
\B(\mathbb{E}\B[\int_{t}^{T}\Vert\hat{x}(s)\Vert_{H}^{8\alpha
}ds\B |\mathcal{F}_{t}\B]\B)^{\frac{1}{2}}\\
& \ +C\B(\mathbb{E}\B[\B(\int_{t}^{T}\Vert
\tilde{D}^{2}k(s)-D^{2}k(s)\Vert\Vert p(s)\Vert_{\mathcal{V}}^{2}%
ds\B)^{4\alpha}\B |\mathcal{F}_{t}\B]\B)^{\frac{1}{2}}\B(\mathbb{E}%
\B[\sup_{s\in\lbrack t,T]}\Vert\hat{x}(s)\Vert_{H}^{8\alpha}\B |\mathcal{F}%
_{t}\B]\B)^{\frac{1}{2}}\\
&  \ +C\B(\mathbb{E}\B[\B(\int_{t}^{T}\Vert
\tilde{D}^{2}k(s)-D^{2}k(s)\Vert\Vert q(s)\Vert_{H}^{2}ds\B)^{4\alpha
}\B |\mathcal{F}_{t}\B]\B)^{\frac{1}{2}}\B(\mathbb{E}\B[\sup_{s\in\lbrack
t,T]}\Vert\hat{x}(s)\Vert_{H}^{8\alpha}\B |\mathcal{F}_{t}\B]\B)^{\frac{1}{2}}\\[3mm]
=& \ o(\left\Vert \hat{x}(t)\right\Vert _{H}^{4\alpha}),\quad P\text{-a.s.}%
\end{align*}
So, \[\mathbb{E}\B[\B(\int_{t}^{T}|J_{5}(s)|ds\B)^{2\alpha}\B|\mathcal{F}%
_{t}\B]=o(\left\Vert \hat{x}(t)\right\Vert _{H}^{4\alpha}),\quad  P\text{-a.s.}\] Then from the a
priori estimate for  BSDEs and (\ref{Myeq5-12}),
\[
\sup_{s\in\lbrack t,T]}\mathbb{E}[|\hat{y}(s)-\frac{1}{2}\sigma(s)|^{2\alpha
}|\mathcal{F}_{t}]+\mathbb{E}\Big[\Big(\int_{t}^{T}|\hat{z}(s)|^{2}
ds\Big)^{\alpha}\B|\mathcal{F}_{t}\Big]=o(\left\Vert \hat{x}(t)\right\Vert
_{H}^{4\alpha}),\quad P\text{-a.s.}%
\]
Taking into account of (\ref{Myeq5-12}) again,%
\[
\sup_{s\in\lbrack t,T]}\mathbb{E}[|\hat{y}(s)|^{2\alpha}|\mathcal{F}%
_{t}]+\mathbb{E}\Big[\Big(\int_{t}^{T}   |\hat{z}(s)|^{2}ds\Big)^{\alpha
}\B|\mathcal{F}_{t}\Big]=o(\left\Vert \hat{x}(t)\right\Vert _{H}^{4\alpha
}),\quad P\text{-a.s.}%
\]
In particular, $P$\text{-a.s}.,
\begin{equation}
Y^{x^{1}}(t)-\bar{Y}(t)=\left\langle p(t),\hat{x}(t)\right\rangle +\frac{1}%
{2}\langle P(t)\hat{x}(t),\hat{x}(t)\rangle+o(\left\Vert \hat{x}(t)\right\Vert
_{H}^{2}). \label{EQ3-2}%
\end{equation}

\textbf{Step five: Completion of the proof.} Let $M$ be a countable dense subset
of $H.$ We can find a subset $\Omega_{0}\subset\Omega$ such that $P(\Omega
_{0})=1$ and for each $\omega\in\Omega_{0},$
\begin{align*}
V(t,\bar{X}(t,\omega))  &  =\bar{Y}(t,\omega),\text{ }Y^{x^{1}}(t,\omega)\geq
V(t,x^{1}),\ (\ref{EQ3-2})\ \text{holds for all}\ x^{1}\in M,\\
&  \text{ and}\ p(s,\omega)\in H,\text{ }P(s,\omega)\in\mathfrak{L}%
(H),\ \forall s\in\lbrack0,T].
\end{align*}
Fix any $\omega\in\Omega_{0}.$ Then for any $x^{1}\in M,$
\[
Y^{x^{1}}(t,\omega)-\bar{Y}(t,\omega)=\left\langle p(t,\omega),\hat
{x}(t,\omega)\right\rangle +\frac{1}{2}\langle P(t,\omega)\hat{x}%
(t,\omega),\hat{x}(t,\omega)\rangle+o(\left\Vert \hat{x}(t,\omega)\right\Vert
_{H}^{2}).
\]
Thus
\[
V(t,x^{1})-V(t,\bar{X}(t,\omega))\leq\langle p(t,\omega),\hat{x}%
(t,\omega)\rangle+\frac{1}{2}\langle P(t,\omega)\hat{x}(t,\omega),\hat
{x}(t,\omega)\rangle+o(\left\Vert \hat{x}(t,\omega)\right\Vert _{H}%
^{2}),\quad \text{for all}\ x^{1}\in M.
\]
Note that the term $o(|\hat{x}(t,\omega)|^{2})$ in the above inequality
depends only on the size $\left\Vert \hat{x}(t,\omega)\right\Vert _{H}^{2}$
and is independent of $x^{1}.$ Therefore,\ from the continuity of
$V(t,\cdot),$ we obtain that
\begin{equation}
V(t,x^{1})-V(t,\bar{X}(t,\omega))\leq\langle p(t,\omega),\hat{x}%
(t,\omega)\rangle+\frac{1}{2}\langle P(t,\omega)\hat{x}(t,\omega),\hat
{x}(t,\omega)\rangle+o(\left\Vert \hat{x}(t,\omega)\right\Vert _{H}%
^{2}),\  \text{for all}\ x^{1}\in H. \label{EQ4-6}%
\end{equation}
This proves, from the definition of upper-differentials, the inclusion (\ref{eq3-3-1}).

Now we prove the second one. Fix any $\omega$ such that (\ref{EQ4-6}) hold.
For any $(\hat{p},\hat{P})\in D_{x}^{2,-}V(t,\bar{X}(t)),$ we have from
the definition of subdifferentials and (\ref{EQ4-6}) that%
\begin{align*}
0  &  \leq\underset{x^{1}\rightarrow\bar{X}(t)}{\lim\inf}
\frac{V(t,x^{1})-V(t,\bar{X}(t))-\langle\hat{p},x^{1}-\bar{X}%
(t)\rangle-\frac{1}{2}\langle\hat{P}(x^{1}-\bar{X}(t
)),x^{1}-\bar{X}(t)\rangle}{\left\Vert x^{1}-\bar{X}(t
)\right\Vert _{H}^{2}} \\
&  \leq\underset{x^{1}\rightarrow\bar{X}(t)}{\lim\inf}
\frac{\langle p(t)-\hat{p},x^{1}-\bar{X}(t)\rangle+\frac{1}%
{2}\langle(P(t)-\hat{P})(x^{1}-\bar{X}(t)),x^{1}-\bar
{X}(t)\rangle}{\left\Vert x^{1}-\bar{X}(t)\right\Vert _{H}^{2}%
} .
\end{align*}
Then we have 
\[
\hat{p}=p(t)\ \text{and}\ \hat{P}\leq P(t),
\]
which implies (\ref{eq3-3-2}).
\end{proof}

\subsection{Differential in time variable}
In this subsection, we study the relationship between DPP and  MP in the time variables.

For $v\in C([0,T]\times H)$ and $(t,x)\in\lbrack0,T)\times H$, the
 partial superdifferential of $v$ with respect to $t$ is
defined as:
\[
D_{t+}^{1,+}v(t,x)=\Big\{r\in \mathbb{R}\Big|v(s,x)\leq
v(t,x)+r(s-t)+o(|s-t|),\mbox{ as }s\downarrow t\Big\}.
\]
The  partial subdifferential of $v$ with respect to $t$
is defined as:%
\[
D_{t+}^{1,-}v(t,x)=\Big\{r\in \mathbb{R}\Big|v(s,x)\geq
v(t,x)+r(s-t)+o(|s-t|),\mbox{ as }s\downarrow t\Big\}.
\]

\begin{theorem}
\label{rel-time} Assume $(H4),(H5),(H6)$. Suppose $(\bar{X}(\cdot),\bar
{Y}(\cdot),\bar{Z}(\cdot),\bar{u}(\cdot))$ are the optimal 4-tuple of Problem
$(S_{x})$ and $(p(\cdot),q(\cdot)),P(\cdot)$ are the solutions of corresponding
adjoint equations (\ref{adjoint1})  and (\ref{adjoint2}), respectively. Then
\begin{equation}\label{eq3-4-1}
\Big\lbrack-\langle p(t),A(t)\bar{X}(t)\rangle_{\ast}-\langle q(t),B(t)\bar
{X}(t)\rangle+\mathcal{H}_{1}(t,\bar{X}(t),\bar
{Y}(t),\bar{Z}(t)),+\infty\Big)\subseteq D_{t+}^{1,+}V(t,\bar{X}(t)),\quad \text{a.e.,}\ P\text{-a.s.}
\end{equation}
and
\begin{equation}\label{eq3-4-2}
D_{t+}^{1,-}V(t,\bar{X}(t))\subseteq\Big(-\infty,-\langle p(t),A(t)\bar
{X}(t)\rangle_{\ast}-\langle q(t),B(t)\bar{X}(t)\rangle+\mathcal{H}_{1}(t,\bar{X}(t),\bar{Y}(t),\bar{Z}(t))\Big],\quad \text{a.e.,}\ P\text{-a.s.},
\end{equation}
where
\begin{align*}
\mathcal{H}_{1}(t,\bar{X}(t),\bar{Y}(t),\bar{Z}(t))=  &  -\mathcal{H}%
(s,\bar{X}(s),\bar{Y}(s),\bar{Z}(s),\bar{u}(s),p(t),q(t))\\
\ \ \ \  &  +\frac12\B\langle P(t)\b[B(t)\bar{X}(t)+b(t,\bar{X}(t),\bar{u}(t))\b],B(t)\bar
{X}(t)+b(t,\bar{X}(t),\bar{u}(t))\B\rangle.
\end{align*}

\end{theorem}

\begin{proof}
\textbf{Step one.} For any $t$, we take any $\tau\in(t,T)$ and denote by $X^{\tau}$
the solution of the following SEE:
\begin{equation}
	\left\{
	\begin{aligned}
		dX^{\tau}(s)= &\,\, \big[A(s)X^{\tau}(s)+a(s,X^{\tau}(s),\bar{u}(s))\big]ds \\
		& +\big[B(s)X^{\tau}(s)+b(s,X^{\tau}(s),\bar{u}(s))\big]dw(s),\quad s\in[\tau,T], \\
		X^{\tau}(\tau)= &\,\, \bar{X}(t).
	\end{aligned}
	\right.
	\label{EQ4-1}
\end{equation}

We define
\[
\hat{\xi}_{\tau}(s):=X^{\tau}(s)-\bar{X}(s),\quad s\in\lbrack\tau,T].
\]
In particular,
\[
\hat{\xi}_{\tau}(\tau)=X^{\tau}(\tau)-\bar{X}(\tau)=\bar{X}(t)-\bar{X}(\tau).
\]
Then on $[\tau,T],$
\begin{align*}
\hat{\xi}_{\tau}(s)    =&\hat{\xi}_{\tau}(\tau)+\int_{\tau}^{s}\big[A(r)\hat{\xi
}_{\tau}(r)+a(r,X^{\tau}(r),\bar{u}(r))-a(r,\bar{X}(r),\bar{u}(r))\big]{d}r\\
&  +\int_{\tau}^{s}\big[B(r)\hat{\xi}_{\tau}(r)+b(r,X^{\tau}(r),\bar
{u}(r))-b(r,\bar{X}(r),\bar{u}(r))\big]{d}w(r).
\end{align*}
Applying Lemma \ref{apriori-see},
\begin{equation}
\mathbb{E}\Big[\sup_{\tau\leq s\leq T}\Vert\hat{\xi}_{\tau}(s)\Vert
_{H}^{2\alpha}\Big|\mathcal{F}_{\tau}\Big]\leq C\Vert\hat{\xi}_{\tau}%
(\tau)\Vert_{H}^{2\alpha},\quad P\text{-a.s.} \label{EQ3-3}%
\end{equation}
Note that $\bar{X}(t)\in L^{2\alpha}(\mathcal{F}_{t},\mathcal{V})$ for a.e.
$t$, since $\bar{X}\in L_{\mathbb{F}}^{2,\alpha}(0,T;\mathcal{V})$. Then by
Lemma \ref{conti-est-1}, we have (for a.e. $t$)
\[
\mathbb{E}[\Vert\hat{\xi}_{\tau}(\tau)\Vert_{H}^{2\alpha}|\mathcal{F}%
_{t}]=\mathbb{E}\big[\Vert\bar{X}(\tau)-\bar{X}(t)\Vert_{H}^{2\alpha
}\big|\mathcal{F}_{t}\big]\leq C\big(1+\Vert \bar{X}(t)\Vert
_{\mathcal{V}}^{2\alpha}\big)|\tau-t|^{\alpha},\quad P\text{-a.s.,}%
\]
and thus,
\begin{equation}
\mathbb{E}\Big[\sup_{\tau\leq r\leq T}\Vert\hat{\xi}_{\tau}(r)\Vert
_{H}^{2\alpha}\Big|\mathcal{F}_{t}\Big]\leq\mathbb{E}\Big[\mathbb{E}%
\Big[\sup_{\tau\leq r\leq T}\Vert\hat{\xi}_{\tau}(r)\Vert_{H}^{2\alpha
}\Big|\mathcal{F}_{\tau}\Big]\Big|\mathcal{F}_{t}\Big]\leq C\mathbb{E}%
[\Vert\hat{\xi}_{\tau}(\tau)\Vert_{H}^{2\alpha}|\mathcal{F}%
_{t}]\leq C|\tau-t|^{\alpha},\quad P\text{-a.s.} \label{eq4-13-2}%
\end{equation}
We can write the equation of $\hat{\xi}_{\tau}(s)$ as
\begin{equation}
\hat{\xi}_{\tau}(s)=\hat{\xi}_{\tau}(\tau)+\int_{\tau}^{s}[\bar{A}(r)\hat{\xi
}_{\tau}(r)+\varepsilon_{1}(r)]{d}r+\int_{\tau}^{s}[\bar{B}(r)\hat{\xi}_{\tau
}(r)+\varepsilon_{2}(r)]{d}w(r) \label{var-t-1}%
\end{equation}
and
\begin{equation}%
\begin{split}
\hat{\xi}_{\tau}(s)=\hat{\xi}_{\tau}(\tau)+  &  \int_{\tau}^{s}[\bar{A}%
(r)\hat{\xi}_{\tau}(r)+\frac{1}{2}\bar{a}_{xx}(r)(\hat{\xi}_{\tau}(r),\hat
{\xi}_{\tau}(r))+\varepsilon_{3}(r)]\,{d}r\\
&  +\int_{\tau}^{s}[\bar{B}(r)\hat{\xi}_{\tau}(r)+\frac{1}{2}\bar{b}%
_{xx}(r)(\hat{\xi}_{\tau}(r),\hat{\xi}_{\tau}(r))+\varepsilon_{4}%
(r)]\,{d}w(r),
\end{split}
\label{var-t-2}%
\end{equation}
where
\begin{align*}
\varepsilon_{1}(r)  &  :=\int_{0}^{1}\langle a_{x}(r,\bar{X}(r)+\mu\hat{\xi
}_{\tau}(r),\bar{u}(r))-\bar{a}_{x}\left(  r\right)  ,\hat{\xi}_{\tau
}(r)\rangle d\mu,\\
\varepsilon_{2}(r)  &  :=\int_{0}^{1}\langle b_{x}(r,\bar{X}(r)+\mu\hat{\xi
}_{\tau}(r),\bar{u}(r))-\bar{b}_{x}(r),\hat{\xi}_{\tau}(r)\rangle d\mu,\\
\varepsilon_{3}(r)  &  :=\int_{0}^{1}(1-\mu)[a_{xx}(r,\bar{X}(r)+\mu\hat{\xi
}_{\tau}(r),\bar{u}(r))-\bar{a}_{xx}\left(  r\right)  ](\hat{\xi}_{\tau
}(r),\hat{\xi}_{\tau}(r))d\mu,\\
\varepsilon_{4}(r)  &  :=\int_{0}^{1}(1-\mu)[b_{xx}(r,\bar{X}(r)+\mu\hat{\xi
}_{\tau}(r),\bar{u}(r))-\bar{b}_{xx}\left(  r\right)  ](\hat{\xi}_{\tau
}(r),\hat{\xi}_{\tau}(r))d\mu.
\end{align*}
\textbf{Step two.} We have, for any $\alpha\geq2$,
\begin{equation}%
\begin{split}
\mathbb{E}\Big[\int_{\tau}^{T}\left\Vert \varepsilon_{1}(r)\right\Vert
_{H}^{\alpha}dr\Big|\mathcal{F}_{\tau}\Big]  &  \leq C\Vert\hat{\xi}_{\tau}%
(\tau)\Vert_{H}^{2\alpha},\quad P\text{-a.s.,}\\
\mathbb{E}\Big[\int_{\tau}^{T}\left\Vert \varepsilon_{2}(r)\right\Vert
_{H}^{\alpha}dr\Big|\mathcal{F}_{\tau}\Big]  &  \leq C\Vert\hat{\xi}_{\tau}%
(\tau)\Vert_{H}^{2\alpha},\quad P\text{-a.s.,}\\
\mathbb{E}\Big[\int_{\tau}^{T}\left\Vert \varepsilon_{3}(r)\right\Vert
_{H}^{\alpha}dr\Big|\mathcal{F}_{t}\Big]  &  =o(|\tau-t|^{\alpha}),\quad P\text{-a.s.,}\\
\mathbb{E}\Big[\int_{\tau}^{T}\left\Vert \varepsilon_{4}(r)\right\Vert
_{H}^{\alpha}dr\Big|\mathcal{F}_{t}\Big]  &  =o(|\tau-t|^{\alpha}),\quad P\text{-a.s.}%
\end{split}
\label{Est-1-2}%
\end{equation}
Indeed, from (\ref{EQ3-3}), we obtain
\begin{align*}
\mathbb{E}\Big[\int_{\tau}^{T}\left\Vert \varepsilon_{1}(r)\right\Vert
_{H}^{\alpha}dr\Big|\mathcal{F}_{\tau}\Big]  &  \leq\int_{\tau}^{T}\mathbb{E}%
\Big[\int_{0}^{1}\Vert a_{x}(r,\bar{X}(r)+\mu\hat{\xi}_{\tau}(r),\bar
{u}(r))-\bar{a}_{x}\left(  r\right)  \Vert_{H}^{\alpha}d\mu\Vert\hat{\xi
}_{\tau}(r)\Vert_{H}^{\alpha}\Big|\mathcal{F}_{\tau}\Big]dr\\
&  \leq\int_{\tau}^{T}\mathbb{E}[\Vert\hat{\xi}_{\tau}(r)\Vert
_{H}^{2\alpha}|\mathcal{F}_{\tau}]dr\\
&  \leq C\Vert\hat{\xi}_{\tau}(\tau)\Vert_{H}^{2\alpha},\quad P\text{-a.s.,}%
\end{align*}
and by (\ref{eq4-13-2}),
\begin{align*}
&  \mathbb{E}\Big[\int_{\tau}^{T}\left\Vert \varepsilon_{3}(r)\right\Vert
_{H}^{\alpha}dr\Big|\mathcal{F}_{t}\Big]\\
&  \leq\Big(\int_{\tau}^{T}\mathbb{E}\Big[\int_{0}^{1}\big\Vert a_{xx}%
(r,\bar{X}(r)+\mu\hat{\xi}_{\tau}(r),\bar{u}(r))-\bar{a}_{xx}\left(
r\right)  \big\Vert_{\mathfrak{L}(H)}^{2\alpha}d\mu\Big|\mathcal{F}%
_{t}\Big]dr\Big)^{\frac{1}{2}}\Big(\int_{\tau}^{T}\mathbb{E}\Big[\int_{0}%
^{1}\Vert\hat{\xi}_{\tau}(r)\Vert_{H}^{4\alpha}d\mu\Big|\mathcal{F}%
_{t}\Big]dr\Big)^{\frac{1}{2}}\\
&  =o(|\tau-t|^{\alpha}),\quad P\text{-a.s.}%
\end{align*}
The other two estimates can be derived similarly.

\textbf{Step three.} Applying It\^{o}'s  formula to $\langle p(r),\hat{\xi}_{\tau
}(r)\rangle,$ from (\ref{var-t-2}) we obtain
\begin{equation}
\langle p(s),\hat{\xi}_{\tau}(s)\rangle=\langle h_{x}(\bar{X}(T)),\hat{\xi
}_{\tau}(T)\rangle+\int_{s}^{T}J_{1}(r)dr-\int_{s}^{T}J_{2}(r)dw({r}).
\label{T-adjoint1}%
\end{equation}
where
\begin{align*}
J_{1}(s):=  &  \langle k_{x}(s)+k_{y}(s)p(s)+k_{z}(s)q(s),\hat{\xi}_{\tau
}(s)\rangle+k_{z}(s)\langle p(s),\bar{B}(s)\hat{\xi}_{\tau}(s)\rangle-\langle
p(s),\varepsilon_{3}(s)\rangle-\langle q(s),\varepsilon_{4}(s)\rangle\\
&  -\frac{1}{2}\b[\langle p(s),(\bar{a}_{xx}(s)(\hat{\xi}_{\tau}(s),\hat{\xi
}_{\tau}(s))\rangle+\langle q(s),\bar{b}_{xx}(s)(\hat{\xi}_{\tau}(s),\hat{\xi
}_{\tau}(s))\rangle\b],\\
J_{2}(s):=  &  \langle p(s),\bar{B}(s)\hat{\xi}_{\tau}(s)\rangle+\langle
q(s),\hat{\xi}_{\tau}(s)\rangle+\langle p(s),\varepsilon_{4}(s)\rangle
+\frac{1}{2}\langle p(s),\bar{b}_{xx}(s)(\hat{\xi}_{\tau}(s),\hat{\xi}_{\tau
}(s))\rangle.
\end{align*}
Applying Theorem \ref{Myth3-7} to $P$ and $\hat{\xi}_{\tau}$ in
(\ref{var-t-1}),  with $\tilde{t}=\tau,$ $\gamma_{1}=\varepsilon_{1}$, $\gamma
_{2}=\varepsilon_{2}$, $x_{0}=\hat{\xi}_{\tau}(\tau)$, and noting that the quantities $\mu_1,\mu_2$ given in that theorem satisfy, according
to Step 2,
\[
\mu_{1}(\tau)=O(\Vert\hat{\xi}_{\tau}(\tau)\Vert_{H}^{3})\text{ and }\mu
_{2}(\tau)=O(\Vert\hat{\xi}_{\tau}(\tau)\Vert_{H}^{2}), \quad P\text{-a.s.}
\]
we obtain
\begin{equation}%
\begin{split}
&  \langle P(s)\hat{\xi}_{\tau}(s),\hat{\xi}_{\tau}(s)\rangle+\sigma
(s)=\langle h_{xx}(\bar{X}(T))\hat{\xi}_{\tau}(T),\hat{\xi}_{\tau}%
(T)\rangle+\int_{s}^{T}\b[k_{y}(r)\langle P(r)\hat{\xi}_{\tau}(r),\hat{\xi
}_{\tau}(r)\rangle\\
&  \ \ \ \ \ \ \ \ +k_{z}(r)\mathcal{Z}(r)+\langle G(r)\hat{\xi}_{\tau
}(r),\hat{\xi}_{\tau}(r)\rangle\b]{d}r-\int_{s}^{T}\mathcal{Z}(r){d}w(r),\quad
s\in\lbrack\tau,T],
\end{split}
\label{T-adjiont2}%
\end{equation}
for some processes $(\sigma,\mathcal{Z})\in L_{\mathbb{F}}^{\alpha}%
(\tau,T)\times L_{\mathbb{F}}^{2,\alpha}(\tau,T)$ satisfying,  for any $\alpha\geq2,$
\[
\sup_{s\in\lbrack\tau,T]}\Big(\mathbb{E}\big[|\sigma(s)|^{\alpha
}\big|\mathcal{F}_{\tau}\big]\Big)^{\frac{1}{\alpha}}\leq C\Vert \hat
{\xi}_{\tau}(\tau)\Vert _{H}^{3}\ \text{and}\ \Big(\ \mathbb{E}%
\Big[\Big(\int_{\tau}^{T}
|\mathcal{Z}(s)|^{2}
ds\Big)^{\frac{\alpha}{2}%
}\Big|\mathcal{F}_{\tau}\Big]\Big)^{\frac{1}{\alpha}}\leq C\Vert \hat
{\xi}_{\tau}(\tau)\Vert _{H}^{2}, \quad P\text{-a.s.}
\]
Moreover, according to (\ref{eq4-13-2}),    for any $\alpha\geq2,$
\[
\sup_{s\in\lbrack\tau,T]}\Big(\mathbb{E}\big[|\sigma(s)|^{\alpha
}\big|\mathcal{F}_{t}\big]\Big)^{\frac{1}{\alpha}}=O(|\tau-t|^{\frac{3}{2}%
})\ \text{and}\ \Big(  \mathbb{E}\Big[\Big(\int_{\tau}^{T}|\mathcal{Z}(s)|^{2}
ds\Big)^{\frac{\alpha}{2}%
}\Big|\mathcal{F}_{\tau}\Big]\Big)^{\frac
{1}{\alpha}}=O(|\tau-t|), \quad P\text{-a.s.}
\]
Consequently, on $[\tau,T]$
\begin{align*}
&  \langle p(s),\hat{\xi}_{\tau}(s)\rangle+\frac{1}{2}\langle P(s)\hat{\xi
}_{\tau}(s),\hat{\xi}_{\tau}(s)\rangle+\frac{1}{2}\sigma(s)=\langle h_{x}%
(\bar{X}(T)),\hat{\xi}_{\tau}(T)\rangle\\
&  \ \ \ \ \ \ \ \ +\frac{1}{2}\langle h_{xx}(\bar{X}(T))\hat{\xi}_{\tau
}(T),\hat{\xi}_{\tau}(T)\rangle+\int_{s}^{T}I_{1}(s){d}s-\int_{s}^{T}%
I_{2}(s){d}w(s),
\end{align*}
where
\begin{align*}
I_{1}(s)  &  :=\langle k_{x}(s)+k_{y}(s)p(s)+k_{z}(s)q(s),\hat{\xi}_{\tau
}(s)\rangle+k_{z}(s)\langle p(s),\bar{B}(s)\hat{\xi}_{\tau}(s)\rangle+\frac
{1}{2}\B\langle\B\{k_{y}(s)P(s)\\
&  \ \ \ \ +D^{2}k(s)\b([I_{d},p(s),\bar{B}^{\ast}(s)p(s)+q(s)],[I_{d}%
,p(s),\bar{B}^{\ast}(s)p(s)+q(s)]\b)\\&  \ \ \ \ +k_{z}(s)\langle p(s),\bar{b}_{xx}%
(s)\rangle\B\}\hat{\xi}_{\tau}(s),\hat{\xi}_{\tau}(s)\B\rangle
 +\frac{1}{2}k_{z}(s)\mathcal{Z}(s)-\langle p(s),\varepsilon
_{3}(s)\rangle-\langle q(s),\varepsilon_{4}(s)\rangle,\\
I_{2}(s)  &  :=\langle p(s),\bar{B}(s)\hat{\xi}_{\tau}(s)\rangle+\langle
q(s),\hat{\xi}_{\tau}(s)\rangle+\langle p(s),\varepsilon_{4}(s)\rangle
+\frac{1}{2}\langle p(s),\bar{b}_{xx}(s)(\hat{\xi}_{\tau}(s),\hat{\xi}_{\tau
}(s))\rangle+\frac{1}{2}\mathcal{Z}(s).
\end{align*}

We denote, on $[\tau,T]$,
\[
Y^{\tau}(s)=h(X^{\tau}(T))+\int_{s}^{T}k(r,X^{\tau}(r),Y^{\tau}(r),Z^{\tau
}(r),\bar{u}(r))dr-\int_{s}^{T}Z^{\tau}(r)dw({r}).
\]
Then
\begin{equation}%
\begin{split}
&  \hat{y}(s)-\frac{1}{2}\sigma(s)=h(X^{x^{1}}(T))-h(\bar{X}(T))-\langle
h_{x}(\bar{X}(T)),\hat{\xi}_{\tau}(T)\rangle-\frac{1}{2}\langle h_{xx}(\bar
{X}(T))\hat{\xi}_{\tau}(T),\hat{\xi}_{\tau}(T)\rangle\\
&  +\int_{s}^{T}\B\{k(r,X^{\tau}(r),Y^{\tau}(r),Z^{\tau}(r),\bar{u}%
(r))-k(r,\bar{X}(r),\bar{Y}(r),\bar{Z}(r),\bar{u}(r))-I_{1}(r)\B\}dr-\int%
_{s}^{T}\hat{z}(r){d}w(r),
\end{split}
\label{Myeq4-13-2}%
\end{equation}
where
\begin{align*}
\hat{y}(s)  &  :=Y^{\tau}(s)-\bar{Y}(s)-\langle p(s),\hat{\xi}_{\tau
}(s)\rangle-\frac{1}{2}\langle P(s)\hat{\xi}_{\tau}(s),\hat{\xi}_{\tau
}(s)\rangle,\\
\hat{z}(s)  &  :=Z^{\tau}(s)-\bar{Z}(s)-I_{2}(s).
\end{align*}
Similar to Step 4 in the proof of Theorem \ref{diff-space},
we obtain (for a.e. $t)$%
\[
\sup_{s\in\lbrack {\tau},T]}\mathbb{E}[|\hat{y}(s)|^{2\alpha}|\mathcal{F}%
_{t}]+\mathbb{E}\Big[\Big(\int_{{\tau}}^{T}
|\hat{z}(s)|^{2}
 ds\Big)^{\alpha
}\Big|\mathcal{F}_{t}\Big]=o(|\tau-t|^{2\alpha}),\quad P\text{-a.s.}
\]
This implies
\begin{equation}
\mathbb{E}\b[  Y^{\tau}(\tau)-\bar{Y}(\tau)\b|\mathcal{F}_{t}\b]
=\mathbb{E}\b[\langle p(\tau),\hat{\xi}_{\tau}(\tau)\rangle+\frac{1}{2}\langle
P(\tau)\hat{\xi}_{\tau}(\tau),\hat{\xi}_{\tau}(\tau)\rangle\b|\mathcal{F}%
_{t}\b]+o(|\tau-t|),\quad P\text{-a.s.} \label{EQ3-2-2}%
\end{equation}

\textbf{Step four.}  We estimate the right-hand side of (\ref{EQ3-2-2}). From the
formulas of $p$, $\hat{\xi}_{\tau}(\tau)$ and the Lebesgue differentiation theorem,
we obtain
\begin{align*}
&  \mathbb{E}[\langle p(\tau),\hat{\xi}_{\tau}(\tau)\rangle|\mathcal{F}%
_{t}]=\mathbb{E}[\langle p(t),\hat{\xi}_{\tau}(\tau)\rangle\big|\mathcal{F}%
_{t}]+\mathbb{E}[\langle p(\tau)-p(t),\hat{\xi}_{\tau}(\tau)\rangle\big|\mathcal{F}%
_{t}]\\
&  =-\b[\langle p(t),A(t)\bar{X}(t)\rangle_{\ast}+\langle p(t),a(t,\bar
{X}(t),\bar{u}(t))\rangle\b](\tau-t)-\b\langle q(t),B(t)\bar{X}(t)+b(t,\bar
{X}(t),\bar{u}(t))\b\rangle(\tau-t)\\
& \ \ \ +o(|\tau-t|),\quad \text{as}\ \tau\downarrow t,\text{ for a.e. }t,\text{
}P\text{-a.s.}%
\end{align*}
In a similar manner, we also have
\begin{align*}
&  \mathbb{E}\Big[\langle P(\tau)\hat{\xi}_{\tau}(\tau),\hat{\xi}_{\tau}(\tau)\rangle\\
&  \ \  -\B\langle P(\tau)\int_{t}^{\tau}\b[B(r)\bar{X}(r)+b(r,\bar{X}(r),\bar
{u}(r))\b]{d}w(r),\int_{t}^{\tau}\b[B(r)\bar{X}(r)+b(r,\bar{X}(r),\bar{u}%
(r))\b]{d}w(r)\B\rangle\Big|\mathcal{F}_{t}\Big]\\
&  =\mathbb{E}\Big[\langle P(\tau)(-\hat{\xi}_{\tau}(\tau)),-\hat{\xi}_{\tau}(\tau
)\rangle-\B\langle P(\tau)\int_{t}^{\tau}\b[B(r)\bar{X}(r)+b(r,\bar{X}(r),\bar
{u}(r))\b]{d}w(r),-\hat{\xi}_{\tau}(\tau)\B\rangle\Big|\mathcal{F}_{t}\Big]\\
&\ \  +\mathbb{E}\Big[\Big\langle P(\tau)\int_{t}^{\tau}[B(r)\bar{X}(r)+b(r,\bar
{X}(r),\bar{u}(r))]{d}w(r),-\hat{\xi}_{\tau}(\tau)\Big\rangle\\
&  \ \ -\B\langle P(\tau)\int_{t}^{\tau}\b[B(r)\bar{X}(r)+b(r,\bar{X}(r),\bar
{u}(r))\b]{d}w(r),\int_{t}^{\tau}\b[B(r)\bar{X}(r)+b(r,\bar{X}(r),\bar{u}%
(r))\b]{d}w(r)\B\rangle\Big|\mathcal{F}_{t}\Big]\\
&  =\mathbb{E}\Big[\B\langle P(\tau)\int_{t}^{\tau}\b[A(r)\bar{X}(r)+a(r,\bar
{X}(r),\bar{u}(r))\b]{d}r,-\hat{\xi}_{\tau}(\tau)\B\rangle\B|\mathcal{F}_{t}\Big]\\
&\ \  +\mathbb{E}\Big[\Big\langle P(\tau)\int_{t}^{\tau}\b[B(r)\bar{X}%
(r)+b(r,\bar{X}(r),\bar{u}(r))\b]{d}w(r),\int_{t}^{\tau}\b[A(r)\bar{X}%
(r)+a(r,\bar{X}(r),\bar{u}(r))\b]{d}r\Big\rangle\Big|\mathcal{F}_{t}\Big]\\
&  \leq\b(\mathbb{E}[\Vert P(\tau)\Vert_{\mathfrak{L}(H)}^{4}%
|\mathcal{F}_{t}]\b)^{\frac{1}{4}}\Big(\mathbb{E}\Big[\Big\Vert\int%
_{t}^{\tau}\b[A(r)\bar{X}(r)+a(r,\bar{X}(r),\bar{u}(r))\b]{d}r\Big\Vert_{H}%
^{2}\Big|\mathcal{F}_{t}\Big]\Big)^{\frac{1}{2}}\big(\mathbb{E}%
[\Vert\hat{\xi}_{\tau}(\tau)\Vert_{H}^{4}|\mathcal{F}_{t}%
]\big)^{\frac{1}{4}}\\
& \ \ +\big(\mathbb{E}[\Vert P(\tau)\Vert_{\mathfrak{L}(H)}^{4}%
|\mathcal{F}_{t}]\big)^{\frac{1}{4}}\Big(\mathbb{E}\Big[\Big\Vert\int%
_{t}^{\tau}\b[B(r)\bar{X}(r)+b(r,\bar{X}(r),\bar{u}(r))\b]{d}w(r)\Big\Vert_{H}%
^{4}\Big|\mathcal{F}_{t}\Big]\Big)^{\frac{1}{4}}\\
&  \ \ \times \Big(\mathbb{E}\Big[\Big\Vert\int_{t}^{\tau}\b[A(r)\bar{X}%
(r)+a(r,\bar{X}(r),\bar{u}(r))\b]{d}r\Big\Vert_{H}^{2}\Big|\mathcal{F}%
_{t}\Big]\Big)^{\frac{1}{2}}\\
&  =o(|\tau-t|),\quad \text{as}\ \tau\downarrow t,\text{ for a.e. }t,\text{
}P\text{-a.s.,}%
\end{align*}
and 
\begin{align*}
&  \mathbb{E}\Big[\Big\langle P(\tau)\int_{t}^{\tau}\b[B(r)\bar{X}(r)+b(r,\bar
{X}(r),\bar{u}(r))\b]{d}w(r),\int_{t}^{\tau}\b[B(r)\bar{X}(r)+b(r,\bar{X}%
(r),\bar{u}(r))\b]{d}w(r)\Big\rangle\\
&\ \   -\Big\langle P(\tau)\b[B(t)\bar{X}(t)+b(t,\bar{X}(t),\bar{u}(t))\b](%
w(\tau)-w(t)),\b[B(t)\bar{X}(t)+b(t,\bar{X}(t),\bar{u}(t))\b](w(\tau
)-w(t))\B\rangle\B|\mathcal{F}_{t}\Big]\\
&  \leq(\tau-t)^{\frac{1}{2}}\Big(\mathbb{E}\big[\Vert P(\tau)\Vert
_{\mathfrak{L}(H)}^{2}\b|\mathcal{F}_{t}\big]\Big)^{\frac{1}{2}}\\
&\ \ \times\Big(\mathbb{E}\Big[\int_{t}^{\tau}\big\Vert B(r)\bar
{X}(r)+b(r,\bar{X}(r),\bar{u}(r))-B(t)\bar{X}(t)-b(t,\bar{X}(t),\bar
{u}(t))\big\Vert_H^{4}dr\Big|\mathcal{F}_{t}%
\Big]\Big)^{\frac{1}{4}}\\
&  \ \ \times\Big\{
\Big(\mathbb{E}\Big[\int_{t}^{\tau}\big\Vert B(r)\bar
{X}(r)+b(r,\bar{X}(r),\bar{u}(r))\big\Vert_H%
^{4}dr\Big|\mathcal{F}_{t}\Big]\Big)^{\frac{1}{4}} +\Big(\mathbb{E}\Big[\int_{t}^{\tau}\big\Vert B(t)\bar
{X}(t)+b(t,\bar{X}(t),\bar{u}(t))\big\Vert_H^{4}%
{d}r\Big|\mathcal{F}_{t}\Big]\Big)^{\frac{1}{4}}\Big\}\\
&  =o(|\tau-t|),\quad \text{as}\ \tau\downarrow t,\ \text{for a.e.}\ t,\text{
}P\text{-a.s.,}%
\end{align*}
and from Proposition \ref{Myth2-7-2},
\begin{align*}
&\mathbb{E}\Big[\Big\langle P(\tau)\b[B(t)\bar{X}(t)+b(t,\bar{X}(t),\bar
{u}(t))\b](w(\tau)-w(t)),\b[B(t)\bar{X}(t)+b(t,\bar{X}(t),\bar{u}(t))\b](w(\tau
)-w(t))\Big\rangle\\
&\ \ \ -\Big\langle P(t)\b[B(t)\bar{X}(t)+b(t,\bar{X}(t),\bar{u}(t))\b](w(\tau
)-w(t)),\b[B(t)\bar{X}(t)+b(t,\bar{X}(t),\bar{u}(t))\b](w(\tau
)-w(t))\Big\rangle\Big|\mathcal{F}_{t}\Big]\\
&=\mathbb{E}\Big[\Big\langle (P(\tau)-P(t))\b[B(t)\bar{X}(t)+b(t,\bar{X}(t),\bar
{u}(t))\b],B(t)\bar{X}(t)+b(t,\bar{X}(t),\bar{u}(t))\Big\rangle(w(\tau
)-w(t))^2\Big|\mathcal{F}_{t}\Big]\\
&\leq(\tau-t){\small \Big(\mathbb{E}\Big[\Big|\Big\langle (P(\tau)-P(t) ) [B(t)\bar{X}%
(t)+b(t,\bar{X}(t),\bar{u}(t)) ], B(t)\bar{X}(t)+b(t,\bar{X}(t),\bar
{u}(t)) \Big\rangle\Big|^{2}\Big|\mathcal{F}_{t}\Big]\Big)^{\frac{1}{2}}}
\\&=o(|\tau-t|),\quad \text{as}\ \tau\downarrow t,\ P\text{-a.s.}
\end{align*}
Moreover,  
\begin{align*}
&  \mathbb{E}\Big[\B\langle P(t)]\b[B(t)\bar{X}(t)+b(t,\bar{X}(t),\bar{u}%
(t))\b](w(\tau)-w(t)),\b[B(t)\bar{X}(t)+b(t,\bar{X}(t),\bar{u}(t))\b](
w(\tau)-w(t))\B\rangle\B|\mathcal{F}_{t}\Big]\\
&  =\B\langle P(t)\b[B(t)\bar{X}(t)+b(t,\bar{X}(t),\bar{u}(t))\b],B(t)\bar
{X}(t)+b(t,\bar{X}(t),\bar{u}(t))\B\rangle(\tau-t),\quad P\text{-a.s.}%
\end{align*}
Therefore,
\begin{align*}
&  \mathbb{E}\left[  Y^{\tau}(\tau)-\bar{Y}(\tau)\b|\mathcal{F}_{t}\right] \\	&=  \mathbb{E}[\langle p(\tau),\xi_{\tau}(\tau)\rangle|\mathcal{F}%
	_{t}]+ \frac12\mathbb{E}[  \langle P(\tau)\xi_{\tau}(\tau),\xi_{\tau}(\tau
	)\rangle|\mathcal{F}_{t}] \\
	&  =\Big\{-\langle p(t),A(t)\bar{X}(t)\rangle_{\ast}-\langle p(t),a(t,\bar
	{X}(t),\bar{u}(t))\rangle-\b\langle q(t),B(t)\bar{X}(t)+b(t,\bar
	{X}(t),\bar{u}(t))\b\rangle\\
	&\ \  \ +\frac12\B\langle P(t)\b[B(t)\bar{X}(t)+b(t,\bar{X}(t),\bar{u}(t))\b],B(t)\bar
	{X}(t)+b(t,\bar{X}(t),\bar{u}(t))\B\rangle\B\}(\tau-t)+o(|\tau-t|),\  \text{a.e.}\ t,\ P\text{-a.s.}%
\end{align*}
Thus,  for a.e. $t$, we have $P$\text{-a.s.}
\begin{equation}%
\begin{split}
&  \mathbb{E}\b[  Y^{\tau}(\tau)-\bar{Y}(t)\b|\mathcal{F}_{t}\b]
=\mathbb{E}\b[  Y^{\tau}(\tau)-\bar{Y}(\tau)\b|\mathcal{F}_{t}\b]
  +\mathbb{E}\b[  \bar{Y}(\tau)-\bar{Y}(t)\b|\mathcal{F}_{t}\b] \\
& 
=\Big\{-\langle p(t),A(t)\bar{X}(t)\rangle_{\ast}-\langle p(t),a(t,\bar
{X}(t),\bar{u}(t))\rangle-\b\langle q(t),B(t)\bar{X}(t)+b(t,\bar
{X}(t),\bar{u}(t))\b\rangle\\
&\ \  \  +\frac12\B\langle P(t)\b[B(t)\bar{X}(t)+b(t,\bar{X}(t),\bar{u}(t))\b],B(t)\bar
{X}(t)+b(t,\bar{X}(t),\bar{u}(t))\B\rangle
\\
&\ \  \ -k(r,\bar{X}(r),\bar{Y}(r),\bar{Z}(r),\bar{u}(r))\B\}(\tau-t)  +o(|\tau-t|)\\
&  =\B\{-\langle p(t),A(t)\bar{X}(t)\rangle_{\ast}-\langle q(t),B(t)\bar
{X}(t)\rangle+\mathcal{H}_{1}(t,\bar{X}(t),\bar{Y}(t),\bar{Z}%
(t))\B\}(\tau-t) +o(|\tau-t|).
\end{split}
\label{Eq5-5}%
\end{equation}

\textbf{Step five.}  For a.e. $t$, we can find a
subset $\Omega_{0}\subset\Omega$ such that $P(\Omega_{0})=1$ and for each
$\omega\in\Omega_{0},$
\begin{align*}
V(t,\bar{X}(t,\omega))  &  =\bar{Y}(t,\omega),\ Y^{\tau}(\tau
,\omega)\geq V(\tau,\bar{X}(t,\omega)),\ {(\ref{Eq5-5})\ \text{holds for all
rational}}\ \tau>t,\ \text{and}\\
p(s,\omega)  &  \in H,\ P(s,\omega)\in\mathfrak{L}(H),\ \forall
s\in\lbrack0,T].
\end{align*}
Fix any $\omega\in\Omega_{0}.$ We have along all rational $\tau>t$ that
\begin{equation}%
\begin{split}
&  V(\tau,\bar{X}(t,\omega))-V(t,\bar{X}(t,\omega))\\
&  \leq\Big\{-\langle p(t,\omega),A(t)\bar{X}(t,\omega)\rangle_{\ast
}-\langle q(t,\omega),B(t)\bar{X}(t,\omega)\rangle+\mathcal{H}_{1}(t,\bar
{X}(t,\omega),\bar{Y}(t,\omega),\bar{Z}(t,\omega))\Big\}(\tau-t)\\
& \ \  \  +o(|\tau-t|).
\end{split}
\label{eq5-5}%
\end{equation}
From the continuity of $V,$ we obtain that the above relationship holds for all
$\tau\in(t,T]$. This proves (\ref{eq3-4-1}).

Now we consider (\ref{eq3-4-2}). Fix an $\omega\in\Omega$ such that (\ref{eq5-5})
holds for any $\tau\in(t,T]$. Then for any $\hat{q}\in D_{t+}^{1,-}V(t,\bar
{X}(t))$, we have
\begin{align*}
0  &  \leq\liminf\limits_{\tau\downarrow t} \frac{V(\tau,\bar
{X}(t))-V(t,\bar{X}(t))-\hat{q}(\tau-t)}{|\tau-t|} \\
&  \leq\liminf\limits_{\tau\downarrow t} \frac{\left[-\langle
p(t),A(t)\bar{X}(t)\rangle_{\ast
}-\langle q(t),B(t)\bar
{X}(t)\rangle+\mathcal{H}_{1}(t,\bar{X}(t),\bar{Y}%
(t),\bar{Z}(t))-\hat{q}\right](\tau-t)}{|\tau-t|} .
\end{align*}
Then we have
\[
\hat{q}\leq-\langle p(t),A(t)\bar{X}(t)\rangle-\langle q(t),B(t
)\bar{X}(t)\rangle+\mathcal{H}_{1}(t,\bar{X}(t),\bar{Y}(t),\bar
{Z}(t)).
\]
This proves (\ref{eq3-4-2}).
\end{proof}

\section{Regularity of the value function}

In this part,  we shall derive a first-order
regularity of $V$ under some additional assumptions. It provides  a more precise understanding on the
relationship obtained in the last section.

We first present the notions of semiconcavity and semiconvexity.
\begin{definition}
A function $w:H\rightarrow\mathbb{R}$ is called semiconcave if there is a
constant $C>0$ such that
\[
\lambda w(x)+(1-\lambda)w(x^{\prime})-w(\lambda x+(1-\lambda)x^{\prime})\leq
C\lambda(1-\lambda)\Vert x-x^{\prime}\Vert_{H}^{2},
\]
for all $\lambda\in\lbrack0,1]$ and $x,x^{\prime}\in H$. A family of funtions
$w^{\alpha}:H\rightarrow\mathbb{R}$ is called semiconcave uniformly in
$\alpha,$ if $C$ in the above formulation is independent of $\alpha.$ We say a function  $w:H\rightarrow\mathbb{R}$   is  semiconvex if $-w$  
is  semiconcave;  we say  a family of functions $w^{\alpha}:H\rightarrow\mathbb{R}$ is  semiconvex  uniformly in $\alpha$  if  $-w^\alpha$ 
is  semiconcave  uniformly  in $\alpha$
.
\end{definition}

In this  section, we assume that $U$ is a convex subset of some separable
Hilbert space $H_{1}$ and make the following assumption.
\begin{description}
\item[$(B1)$] $a(t,x,v):[0,T]\times H\times U\rightarrow H$ and $b(t,x,v):[0,T]\times
H\times U\rightarrow H$ are Fr\'{e}chet differentiable in $x$
and there exists a constant $C>0$ such that
\[
\Vert a_{x}(t,x,v)-a_{x}(t,x^{\prime},v)\Vert_{\mathfrak{L}(H)}+\Vert b_{x}(t,x,v)-b_{x}(t,x^{\prime},v)\Vert_{\mathfrak{L}(H)}\leq C\Vert
x-x^{\prime}\Vert_{H},
\]
for all $t\in[0,T]$, $x,x^{\prime}\in H$ and $v\in U$.
\end{description}

For any $x_{0},x_{1}\in H$ and $u_{0}(\cdot),u_{1}(\cdot)\in\mathcal{U}$, we
define
\begin{equation}%
\begin{cases}
X_{0}(s)=X^{t,x_{0};u_{0}}(s),\\
X_{1}(s)=X^{t,x_{1};u_{1}}(s).
\end{cases}
\label{x1x0differentcontrols}%
\end{equation}
For any $\lambda\in\lbrack0,1],$ we denote
\begin{equation}%
\begin{cases}
u_{\lambda}(s)=\lambda u_{1}(s)+(1-\lambda)u_{0}(s),\\
x_{\lambda}=\lambda x_{1}+(1-\lambda)x_{0},\\
X_{\lambda}(s)=X^{t,x_{\lambda};u_{\lambda}}(s),\\
X^{\lambda}(s)=\lambda X_{1}(s)+(1-\lambda)X_{0}(s).
\end{cases}
\label{lambdadefinition}%
\end{equation}

\begin{lemma}
\label{Le3-9} Let Assumptions $(H4),(H6)$ and $(B1)$ be satisfied. Take
$u_{0}(\cdot)=u_{1}(\cdot)=u(\cdot)$, for any $u(\cdot
)\in\mathcal{U}$. Then there exist constant
$C\geq0$ such that 
\[
\mathbb{E}\B[  \sup_{s\in\lbrack t,T]}\Vert X^{\lambda}(s)-X_{\lambda
}(s)\Vert_{H}\B]  \leq C\lambda(1-\lambda)\Vert x_{1}-x_{0}\Vert_{H}^{2},
\]
for   all $t\in\lbrack0,T]$,   $\lambda\in\lbrack0,1]$, $x_{0},x_{1}\in H$ and $u(\cdot
)\in\mathcal{U}$.
\end{lemma}

\begin{proof}
 Note that, for any $x_{1},x_{0}\in H,v\in U,$
\begin{align*}
&  \left\Vert \lambda a(s,x_{1},v)+(1-\lambda)a(s,x_{0},v)-a(s,x_{\lambda
},v)\right\Vert _{H}\\
&  =\lambda(1-\lambda)\int_{0}^{1}a_{x}(s,x_{\lambda}+\theta(1-\lambda
)(x_{1}-x_{0}),v)(x_{1}-x_{0})d\theta\\
&  \ \ \ +\lambda(1-\lambda)\int_{0}^{1}a_{x}(s,x_{\lambda}+\theta\lambda
(x_{0}-x_{1}),v)(x_{0}-x_{1})d\theta\\
&  \leq\lambda(1-\lambda)\left\Vert x_{1}-x_{0}\right\Vert _{H}\int_{0}%
^{1}\b\Vert a_{x}(s,x_{\lambda}+\theta(1-\lambda)(x_{1}-x_{0}),v)-a_{x}%
(s,x_{\lambda}+\theta\lambda(x_{0}-x_{1}),v)\b\Vert_{\mathfrak{L}(H)}d\theta\\
&  \leq C\lambda(1-\lambda)\left\Vert x_{1}-x_{0}\right\Vert _{H}^{2}.
\end{align*}
Similarly,
\begin{align*}
&  \b\Vert \lambda b(s,x_{1},v)+(1-\lambda)b(s,x_{0},v)-b(s,x_{\lambda
},v)\b\Vert _H \leq C\lambda(1-\lambda)\left\Vert x_{1}-x_{0}\right\Vert _{H}^{2}.
\end{align*}
Then from the above two inequalities and the B-D-G inequality,
\begin{align*}
&  \mathbb{E}\Big[\sup_{s\in\lbrack t,T]}\Vert X^{\lambda}(s)-X_{\lambda
}(s)\Vert_{H}^{2}\Big]\\
&  \leq C\mathbb{E}\Big[\sup_{s\in\lbrack t,T]}\int_{t}^{s}\left\Vert
\lambda a(r,X_{1}(r),u(r))+(1-\lambda)a(r,X_{0}(r),u(r))-a(r,X^{\lambda
}(r),u(r))\right\Vert _{H}^{2}dr\Big]\\
&  \ \ \ +C\mathbb{E}\Big[\int_{t}^{T}\left\Vert \lambda b(r,X_{1}(r),u(r))+(1-\lambda
)b(r,X_{0}(r),u(r))-b(r,X^{\lambda}(r),u(r))\right\Vert _{H
}^{2}dr\Big]\\
&  \ \ \ +C\mathbb{E}\Big[\sup_{s\in\lbrack t,T]}\int_{t}^{s}\b\Vert a(r,X^{\lambda
}(r),u(r))-a(r,X_{\lambda}(r),u(r))\b\Vert_{H}^{2}dr\Big]\\
&  \ \ \ +C\mathbb{E}\Big[\int_{t}^{T}\b\Vert b(r,X^{\lambda
}(r),u(r))-b(r,X_{\lambda}(r),u(r))\b\Vert_H^{2}dr\Big]\\
&  \leq C\lambda^{2}(1-\lambda)^{2}\mathbb{E}\Big[\int_{t}^{T}\Vert X_{1}%
(r)-X_{0}(r)\Vert_{H}^{4}dr\Big]+C\mathbb{E}\Big[\int_{t}^{T}\Vert X^{\lambda}(r)-X_{\lambda
}(r)\Vert_{H}^{2}dr\Big]\\
&  \leq C\lambda^{2}(1-\lambda)^{2}\Vert x_{1}-x_{0}\Vert_{H}^{4}%
+C\mathbb{E}\Big[\int_{t}^{T}\Vert X^{\lambda}(r)-X_{\lambda}(r)\Vert_{H}%
^{2}dr\Big].
\end{align*}
Applying Gr\"{o}nwall's inequality, we get%
\[
\mathbb{E}\Big[\sup_{s\in\lbrack t,T]}\Vert X^{\lambda}(s)-X_{\lambda}%
(s)\Vert_{H}^{2}\Big]\leq C\lambda^{2}(1-\lambda)^{2}\Vert x_{1}-x_{0}%
\Vert_{H}^{4}.
\]
Therefore,
\[
\mathbb{E}\Big[\sup_{s\in\lbrack t,T]}\Vert X^{\lambda}(s)-X_{\lambda}%
(s)\Vert_{H}\Big]\leq C\lambda(1-\lambda)\Vert x_{1}-x_{0}\Vert_{H}^{2},
\]
which completes the proof.
\end{proof}

Now we assume:
\begin{description}
\item[$(B2)$] $k(t,x,y,z,v)$ is semiconcave in $x$ and concave in $(y,z)$, uniformly in $(t,v)$, that is,  for some constant $C_1>0$,
\begin{align*}
&\lambda k(t,x,y,z,v)+(1-\lambda)k(t,x^{\prime},y^{\prime},z^{\prime},v)-k(t,\lambda x+(1-\lambda)x^{\prime},\lambda y+(1-\lambda)y^{\prime},\lambda z+(1-\lambda)z^{\prime},v)\\
&\leq
C_1\lambda(1-\lambda)\Vert x-x^{\prime}\Vert_{H}^{2},
\end{align*} for all $\lambda\in\lbrack0,1]$, $x,x^{\prime}\in H$, $t\in [0,T]$ and $v\in U$; $h(x)$
is semiconcave in $x$  for some constant $C_2>0$.
\end{description}
\begin{lemma}\label{Le3-7}
Let Assumptions $(H4)$, $(H6)$, $(B1)$ and $(B2)$   be satisfied, and take
$u_{0}(\cdot)=u_{1}(\cdot)=u(\cdot)$,  for any $u(\cdot
)\in\mathcal{U}$.  For $\lambda\in\lbrack0,1]$, we denote the solution
to BSDE (\ref{BSDE1-2}) for $(t,x_{1}),$ $(t,x_{0})$ and $(t,x_{\lambda
})$ by $(Y_{1}(\cdot),Z_{1}(\cdot)),$ $(Y_{0}(\cdot),Z_{0}(\cdot))$ and
$(Y_{\lambda}(\cdot),Z_{\lambda}(\cdot)),$ respectively. We also denote%
\[
(Y^{\lambda}(\cdot),Z^{\lambda}(\cdot))=(\lambda Y_{1}(\cdot)+(1-\lambda
)Y_{0}(\cdot),\lambda Z_{1}(\cdot)+(1-\lambda)Z_{0}(\cdot)).
\]
Then
\[
Y^{\lambda}(t)-Y_{\lambda}(t)\leq C\lambda(1-\lambda)\Vert x_{1}-x_{0}%
\Vert_{H}^{2}.
\]
for all $\lambda\in\lbrack0,1]$, $x_{0},x_{1}\in H$ and $u(\cdot
)\in\mathcal{U}$.
\end{lemma}

\begin{proof}
We have on $[t,T]$ that
\begin{equation}%
\begin{split}
&  Y^{\lambda}(s)=\lambda h(X_{1}(T))+(1-\lambda)h(X_{0}(T))\\
& \ \  +\int_{s}^{T}\big[\lambda k(X_{1}(r),Y_{1}(r),Z_{1}(r),u(r))+(1-\lambda
)k(X_{0}(r),Y_{0}(r),Z_{0}(r),u(r))\big]dr -\int_{s}^{T}Z^{\lambda}(r)dw(r)
\end{split}
\label{BSDE4-1}%
\end{equation}
and%
\begin{equation}%
\begin{split}
&  Y_{\lambda}(s)=h(X_{\lambda}(T))+\int_{s}^{T}k(X_{\lambda}(r),Y_{\lambda
}(r),Z_{\lambda}(r),u(r))dr  -\int_{s}^{T}Z_{\lambda}(r)dw(r).
\end{split}
\label{BSDE4-2}%
\end{equation}
We  introduce a new BSDE on $[t,T]$:
\begin{equation}%
\begin{split}
&  \tilde{y}(s)=h(X^{\lambda}(T))-h(X_{\lambda}(T))+C_2\lambda(1-\lambda)\Vert
X_{1}(T)-X_{0}(T)\Vert_{H}^{2}\\
&  \ \ +\int_{s}^{T}\big[k(X^{\lambda}(r),Y_{\lambda}(r)+\tilde{y}(r),Z_{\lambda}(r)+\tilde{z}(r),u(r))-k(X_{\lambda}(r),Y_{\lambda}(r),Z_{\lambda}(r),u(r))\\
& \ \  +C_1\lambda(1-\lambda)\Vert X_{1}(r)-X_{0}(r)\Vert_{H}^{2}\big]dr -\int_{s}^{T}\tilde{z}(r)dw(r).
\end{split}
\label{BSDE4-3}%
\end{equation}
From (\ref{BSDE4-2}) and (\ref{BSDE4-3}), we get
\begin{equation}%
\begin{split}
  Y_{\lambda}(s)+\tilde{y}(s)=&h(X^{\lambda}(T))+C_2\lambda(1-\lambda)\Vert
X_{1}(T)-X_{0}(T)\Vert_{H}^{2}+\int_{s}^{T}\big[k(X^{\lambda}(r),Y_{\lambda}(r)+\tilde{y}(r),Z_{\lambda}(r)+\tilde{z}(r),u(r))\\
&  +C_1\lambda(1-\lambda)\Vert X_{1}(r)-X_{0}(r)\Vert_{H}^{2}\big]dr  -\int_{s}^{T}(Z_{\lambda}(r)+\tilde{z}(r))dw(r).
\end{split}
\label{BSDE4-4}%
\end{equation}
From Assumption (B2), we know that%
\[
\lambda h(X_{1}(T))+(1-\lambda)h(X_{0}(T))\leq h(X^{\lambda}(T))+C_2\lambda
(1-\lambda)\Vert X_{1}(T)-X_{0}(T)\Vert_{H}^{2}%
\]
and
\begin{align*}
&  \lambda k(X_{1}(r),Y_{1}(r),Z_{1}(r),u(r))+(1-\lambda)k(X_{0}%
(r),Y_{0}(r),Z_{0}(r),u(r))\\
&  \leq k(X^{\lambda}(r),Y^{\lambda}(r),Z^{\lambda}(r),u(r))+C_1\lambda
(1-\lambda)\Vert X_{1}(r)-X_{0}(r)\Vert_{H}^{2}.
\end{align*}
Thus, applying the comparison theorem for  BSDEs to (\ref{BSDE4-1})
and (\ref{BSDE4-4}), we derive
\begin{equation}\label{EQ4-11}
Y^{\lambda}(t)\leq Y_{\lambda}(t)+\tilde{y}(t).
\end{equation}
Next, apply the a priori estimate for  BSDEs to (\ref{BSDE4-3}), we have
from Lemma \ref{Le3-9}  that
\begin{align*}
&  \mathbb{E}\Big[\sup\limits_{t\leq s\leq T}|\tilde{y}(s)|^{2}+\int_{t}%
^{T}|\tilde{z}(s)|^{2}ds\Big]\\
&  \leq C\mathbb{E}\b[\Vert h(X^{\lambda}(T))-h(X_{\lambda}(T))\Vert_{H}^{2}+C\lambda(1-\lambda)\Vert
X_{1}(T)-X_{0}(T)\Vert_{H}^{2}\b]\\
& \ \ +C\mathbb{E}\B[\int_{t}^{T}\b|k(X^{\lambda}(s),Y_{\lambda}(s),Z_{\lambda
}(s),u(s))-k(X_{\lambda}(s),Y_{\lambda}(s),Z_{\lambda}(s),u(s))\b|^{2}ds\B]\\
&  \leq C\mathbb{E}\b[\Vert X^{\lambda}(T)-X_{\lambda}(T)\Vert_{H}%
^{2}\b]+C\lambda(1-\lambda)\mathbb{E}\b[\Vert
X_{1}(T)-X_{0}(T)\Vert_{H}^{2}\b]
+C\mathbb{E}\B[\int_{t}^{T}\Vert X^{\lambda}(s)-X_{\lambda}(s)\Vert_{H}%
^{2}ds\B]\\
&  \leq C\lambda(1-\lambda)\Vert x_{1}-x_{0}\Vert_{H}^{2}.
\end{align*}
So,
\[
\tilde{y}(t)\leq C\lambda(1-\lambda)\Vert x_{1}-x_{0}\Vert_{H}^{2}.
\]
Combining this with (\ref{EQ4-11}), we get the desired result.
\end{proof}

\begin{theorem}
Let Assumptions  $(H4)$, $(H6)$, $(B1)$ and $(B2)$  be satisfied. Then the function $V(t,\cdot)$ is semiconcave, uniformly
in $t\in\lbrack0,T].$
\end{theorem}

\begin{proof}
Fix any $x_{1},x_{0}\in H$ and $\lambda\in\lbrack0,1].\ $For any $\varepsilon
>0,$ we can find $u_{\varepsilon}\in\mathcal{U}^{t}$ such that
\[
J(t,x_{\lambda};u_{\varepsilon}(\cdot))<V(t,x_{\lambda})+\varepsilon.
\]
Applying Lemma \ref{Le3-7} with $u=u_\varepsilon$ and using the notations therein, we obtain
\begin{align*}
&  \lambda V(t,x_{1})+(1-\lambda)V(t,x_{0})-V(t,x_{\lambda})-\varepsilon\\
&  \leq\lambda J(t,x_{1};u_{\varepsilon}(\cdot))+(1-\lambda)J(t,x_{0}%
;u_{\varepsilon}(\cdot))-J(t,x_{\lambda};u_{\varepsilon}(\cdot))\\
&  =\lambda Y_{1}(t)+(1-\lambda)Y_{0}(t)-Y_{\lambda}(t)\\
&  =Y^{\lambda}(t)-Y_{\lambda}(t)\\
&  \leq C\lambda(1-\lambda)\Vert x_{1}-x_{0}\Vert_{H}^{2}.
\end{align*}
Letting $\varepsilon\rightarrow0$, we get the desired result.
\end{proof}

Next, we study the convexity of the value function under the following assumption.
\begin{description}
\item[$(B3)$] $a:[0,T]\times H\times H_{1}\rightarrow H$ and $b:[0,T]\times H\times H_{1}%
\rightarrow H$ are linear  in $(x,u)$.
For each  $t\in\lbrack0,T]$, $k(t,\cdot,\cdot,\cdot,\cdot): H\times\mathbb{R}\times \mathbb{R} \times U\rightarrow\mathbb{R}$
and $h(\cdot):H\rightarrow\mathbb{R}$ are convex.
\end{description}
\begin{theorem}
Suppose $(H4)$, $(H6)$, $(B1)$, $(B2)$ and $(B3)$. Then, for every
$t\in\lbrack0,T]$, the function $V(t,\cdot)$ is convex.
\end{theorem}
\begin{proof}
 Given any $x_{1},x_{0}\in
H,$ $\lambda\in\lbrack0,1],\ $and for any $\varepsilon>0,$ we can find
$u_{\varepsilon}^{1},u_{\varepsilon}^{0}\in\mathcal{U}^{t}$ such that
\[
J(t,x_{1};u_{\varepsilon}^{1}(\cdot))<V(t,x_{1})+\varepsilon,
\]%
\[
J(t,x_{0};u_{\varepsilon}^{0}(\cdot))<V(t,x_{0})+\varepsilon,
\]
and define%
\[
u_{\varepsilon}^{\lambda}(\cdot)=\lambda u_{\varepsilon}^{1}(\cdot
)+(1-\lambda)u_{\varepsilon}^{0}(\cdot).
\]
We denote the corresponding solution of the systems with controls
$u_{\varepsilon}^{1}$, $u_{\varepsilon}^{0}$ and $u_{\varepsilon}^{\lambda}$
by $(\tilde{X}_{1}(\cdot),\tilde{Y}_{1}(\cdot),\tilde{Z}_{1}(\cdot))$,
$(\tilde{X}_{0}(\cdot),\tilde{Y}_{0}(\cdot),\tilde{Z}_{0}(\cdot))$ and
$(\tilde{X}_{\lambda}(\cdot),\tilde{Y}_{\lambda}(\cdot),\tilde{Z}_{\lambda
}(\cdot)),$ respectively$.$ We define  $$\tilde{X}^{\lambda}:=\lambda\tilde{X}_{1}+(1-\lambda)\tilde{X}%
_{0},\ \tilde{Y}^{\lambda}:=\lambda\tilde{Y}_{1}+(1-\lambda)\tilde{Y}_{0}\ \text{and}\
 \tilde{Z}^{\lambda}:=\lambda\tilde{Z}_{1}+(1-\lambda)\tilde{Z}_{0}.$$  Then%
\begin{align*}
&  \lambda V(t,x_{1})+(1-\lambda)V(t,x_{0})-V(t,x_{\lambda})+\varepsilon\\
&  \geq\lambda J(t,x_{1};u_{\varepsilon}^{1}(\cdot))+(1-\lambda)J(t,x_{0}%
;u_{\varepsilon}^{0}(\cdot))-J(t,x_{\lambda};u_{\varepsilon}^{\lambda}%
(\cdot))\\
&  =\lambda\tilde{Y}_{1}(t)+(1-\lambda)\tilde{Y}_{0}(t)-\tilde{Y}_{\lambda
}(t).
\end{align*}

Note that   $\tilde{X}_{\lambda} =\tilde{X}^{\lambda} $ due to the linearity of the
coefficients. From the convexity assumptions on $h$ and $k$, we have $$\lambda h(\tilde{X}_{1}(T))+(1-\lambda)h(\tilde{X}_{0}(T))\geq
h(\tilde{X}^{\lambda}(T))= h(\tilde{X}_{\lambda}(T))$$ and
\begin{align*}
&  \lambda k(r,\tilde{X}_{1}(r),\tilde{Y}_{1}(r),\tilde{Z}_{1}(r),u_{\varepsilon
}^{1}(r))+(1-\lambda)k(r,\tilde{X}_{0}(r),\tilde{Y}_{0}(r),\tilde{Z}%
_{0}(r),u_{\varepsilon}^{0}(r))\\
&  \geq k(r,\tilde{X}^{\lambda}(r),\tilde{Y}^{\lambda}(r),\tilde{Z}^{\lambda
}(r),u_{\varepsilon}^{\lambda}(r))\\
&  =k(r,\tilde{X}_{\lambda}(r),\tilde{Y}^{\lambda}(r),\tilde{Z}^{\lambda
}(r),u_{\varepsilon}^{\lambda}(r)),
\end{align*}
Then according to the comparison theorem for  BSDEs, we derive
\[
\tilde{Y}^{\lambda}(s)\geq\tilde{Y}_{\lambda}(s).
\]
Thus,%
\[
\lambda V(t,x_{1})+(1-\lambda)V(t,x_{0})-V(t,x_{\lambda})+\varepsilon\geq0.
\]
Letting $\varepsilon\rightarrow0$, we get the desired result.
\end{proof}
\begin{remark}\label{Rm3-16}
From the previous results,  for all $t\in\lbrack0,T]$, we know that $V(t,\cdot):H\rightarrow\mathbb{R}$ is
semiconcave and convex, then  $V(t,\cdot)\in
C^{1,1}(H)$ according to \cite{LL86}. 

As a direct corollary of Theorem \ref{diff-space}, we have 
	\[
	D_{x}^{1,-}V(t,\bar{X}(t))\subset	\{p(t)\}\subset D_{x}^{1,+}V(t,\bar{X}(t)),\quad
	t\in\lbrack0,T],\quad P\text{-a.s.},
	\]
	where, for $v\in C([0,T]\times H)$ and $(t,x)\in\lbrack0,T)\times H$,
	\begin{align*}
		D_{x}^{1,+}v(t,x)  & : =\Big\{p\in H\Big| v(t,y)\leq v(t,x)+\langle p,y-x\rangle+o(\Vert y-x\Vert_{H}),\mbox{ as }y\rightarrow
		x\Big\},\\
		D_{x}^{1,-}v(t,x)  & : =\Big\{p\in H\Big|v(t,y)\geq v(t,x)+\langle p,y-x\rangle+o(\Vert y-x\Vert_{H}),\mbox{ as }y\rightarrow
		x\Big\}.
	\end{align*}
When $V(t,\cdot)\in C^{1,1}(H)$, it is straightforward to check that
the Fr\'{e}chet derivative $V_{x}(t,\bar{X}(t))\in D_{x}^{1,-}V(t,\bar{X}(t))\cap
D_{x}^{1,+}V(t,\bar{X}(t)).$ Then according to 
\cite[Proposition 6.5.1]{LY95}, we know that
\[
D_{x}^{1,-}V(t,\bar{X}(t))=D_{x}^{1,+}V(t,\bar{X}(t))=\{V_{x}(t,\bar{X}(t))\},\quad
t\in\lbrack0,T],\ P\text{-a.s.}%
\]
So,%
\[
D_{x}^{1,-}V(t,\bar{X}(t))=D_{x}^{1,+}V(t,\bar{X}(t))=\{V_{x}(t,\bar{X}(t)
)\}=\{p(t)\},\quad t\in\lbrack0,T],\ P\text{-a.s.}%
\]
\end{remark}
\begin{remark}\label{rm4-1}
The value function in DPP corresponds to the viscosity solutions of
second-order Hamilton-Jacobi-Bellman equations in Hilbert space for 
recursive control systems (see, e.g., \cite{TZ24-1}).	For the $C^{1,1}$-regularity of the value
	functions of non-recursive stochastic control problems in Hilbert space, when
	the operators in the equations are bounded, it was obtained in
	\cite{BGY24,MS23}; when the equation contains unbounded operators, some partial
	results are available in \cite{FFS24,Go95} and the  $C^{1,1}$-regularity  for conventional optimal control problems was obtained in \cite{FSW23}. The interested readers are referred  to
	\cite{CS04,FGG10,Lions82,Kr80} for the results in a finite dimensional space.
\end{remark}
\section{Smooth case and examples}
\subsection{The smooth case}
In this section, we examine the relationship between the MP and the DPP under the assumption that the value function   $V(t,x)$ is sufficiently
smooth.  To apply the results on infinite-dimensional Hamilton-Jacobi-Bellman (HJB)  eqations, we assume throughout this subsection that $B\equiv 0$.
We begin by the following result.

Consider the following HJB  equation:
\begin{equation}
	\left\{
	\begin{aligned}
		V_{t}(t,x)
		& + \langle A^{\ast}V_{x}(t,x),x\rangle \\
		& \quad + \inf_{v\in U} G\big(t,x,V(t,x),V_{x}(t,x),V_{xx}(t,x),v\big)
		= 0,\quad (t,x)\in[0,T]\times H, \\
		V(T,x)
		& = h(x),\quad x\in H.
	\end{aligned}
	\right.
	\label{Myeq4-8}
\end{equation}
where%
\begin{align*}
G(t,x,r,p,P,v)  &  :=\frac{1}{2}\langle Pb(t,x,v),b(t,x,v)\rangle
+\langle p,a(t,x,v)\rangle+k(t,x,r,\langle
p,b(t,x,v)\rangle,v),\\
 &\quad\quad\quad (t,x,r,p,P,v)  \in\lbrack0,T]\times H\times\mathbb{R\times}H\times
S(H)\times U.
\end{align*}

We denote  the weak topology on $\mathcal{V}$ by $\tau_w$ and define
\begin{align*}
{\Phi}:=  &  \Big\{\varphi\in C^{1,2}([0,T]\times H):\text{if }x\in
\mathcal{V},\ \text{then}\ V_{x}(t,x)\in\mathcal{V};\ x\rightarrow
V_{x}(t,x)\text{ is continuous from}\\
&  (\mathcal{V}, \Vert \cdot\Vert_{\mathcal{V}})\ \text{to}\ (\mathcal{V},\tau_w);\ \Vert V_{x}(t,x)\Vert_{\mathcal{V}}\leq C(1+\Vert
x\Vert_{\mathcal{V}}) \ \text{and}\ A^{\ast}(t)V_{x}(t,x)\in C([0,T]\times H;H)\Big\}.
\end{align*}

The proof of the following result is provided in the Appendix.
\begin{proposition}
\label{Prop5-6} Assume $(H4)$ and $(H6)$. Suppose that the value function
$V\in{\Phi},$ then it is a classical solution of the HJB equation
(\ref{Myeq4-8}).
\end{proposition}
\begin{remark}\label{vis-rm}
	Recently, \cite{TZ24-1} introduced a new notion of viscosity solutions for
	infinite-dimensional HJB equation with unbounded
	operators. It eliminated the need for the so-called $B$-continuity assumption on the coefficients (see \cite{FGS17}) by utilizing an It\^{o}-type inequality for solutions of SEEs. If we assume the operator $L_{A,0}(t,s)$ (see (\ref{Myeq1-5})) is contractive in the sense that $\Vert L_{A,0}(t,s)x\Vert_{H}\leq \Vert u\Vert_{H}$, for $x\in H$ and
	 $t\leq s,$ 
	it can be shown  by the same method that $V$ is the unique viscosity solution of 
	(\ref{Myeq4-8}).
\end{remark}

The following is the relationship between MP and DPP in the smooth case.

\begin{theorem}
\label{Rel-smooth} Assume $(H4)$, $(H6)$ and fix $x\in H.$ Suppose $(\bar
{X}(\cdot),\bar{Y}(\cdot),\bar{Z}(\cdot),\bar{u}(\cdot))$ are the optimal
4-tuple of Problem $(S_{x})$ and $(p(\cdot),q(\cdot)),P(\cdot)$ are the
solutions of corresponding adjoint equations$.$ Suppose that the value
function $V\in{\Phi}$. Then
\begin{equation}%
\begin{split}
&  -V_{t}(t,\bar{X}(t))\\
&  =\langle A(t)V_{x}(t,\bar{X}(t)),\bar{X}(t)\rangle_{\ast}+G(t,\bar
{X}(t),V(t,\bar{X}(t)),V_{x}(t,\bar{X}(t)),V_{xx}(t,\bar{X}(t)),\bar{u}(t))\\
&  =\langle A(t)V_{x}(t,\bar{X}(t)),\bar{X}(t)\rangle_{\ast}+\inf_{v\in
U}G(t,\bar{X}(t),V(t,\bar{X}(t)),V_{x}(t,\bar{X}(t)),V_{xx}(t,\bar
{X}(t)),v),\quad P\text{-a.s. a.e.}%
\end{split}
\label{EQ5-1}%
\end{equation}
If moreover $V\in C^{1,3}([0,T]\times H)$ with $V_{x}\in{\Phi}$, then
\begin{align*}
p(t)  &  =V_{x}(t,\bar{X}(t)),\quad P\text{-a.s. a.e.,}\\
q(t)  &  =V_{xx}(t,\bar{X}(t))b(t,\bar{X}(t),\bar{u}(t)),\quad P\text{-a.s.
a.e.}%
\end{align*}

\end{theorem}

\begin{proof}
From Lemma \ref{Le4-4}, we know that
\[
V(t,\bar{X}(t))=\bar{Y}(t).
\]
Applying It\^{o}'s formula (\cite[Lemma 2.15]{P2021}), we obtain
\begin{align*}
dV(t,\bar{X}(t))  &  =\Big[V_{t}(t,\bar{X}(t))+\langle V_{x}(t,\bar
{X}(t)),a(t,\bar{X}(t),\bar{u}(t))\rangle+\langle A^{\ast}(t)V_{x}(t,\bar
{X}(t)),\bar{X}(t)\rangle\\
&  +\frac{1}{2}\b\langle V_{xx}(t,\bar{X}(t))b(t,\bar{X}(t),\bar{u}%
(t)),b(t,\bar{X}(t),\bar{u}(t))\b\rangle\Big]dt+\langle V_{x}(t,\bar
{X}(t)),b(t,\bar{X}(t),\bar{u}(t)) \rangle dw(t).
\end{align*}
Thus,
\begin{align*}
&  V_{t}(t,\bar{X}(t))+\langle V_{x}(t,\bar{X}(t)),a(t,\bar{X}(t),\bar
{u}(t))\rangle+\langle A^{\ast}(t)V_{x}(t,\bar{X}(t)),\bar{X}(t)\rangle\\
&  +\frac{1}{2}\langle V_{xx}(t,\bar{X}(t))b(t,\bar{X}(t),\bar{u}%
(t)),b(t,\bar{X}(t),\bar{u}(t))\rangle=-k(t,\bar{X}(t),\bar{Y}(t),\bar
{Z}(t),\bar{u}(t))
\end{align*}
and
\[
\langle V_{x}(t,\bar{X}(t)),b(t,\bar{X}(t),\bar{u}(t))\rangle=\bar{Z}(t).
\]
Therefore,
\begin{equation}%
\begin{split}
&  V_{t}(t,\bar{X}(t))+\langle A^{\ast}V_{x}(t,\bar{X}(t)),\bar{X}(t)\rangle+G(t,\bar{X}(t),V(t,\bar{X}(t)),V_{x}(t,\bar{X}(t)),V_{xx}%
(t,\bar{X}(t)),\bar{u}(t))=0.
\end{split}
\label{eq5-8}%
\end{equation}
This proves the first equality in (\ref{EQ5-1}). The second equality follows
since $V$ is the classical solution of HJB equation.

We then consider the second part. Taking into account (\ref{eq5-8}) and the
fact that $V$ is a solution to HJB equation, we have
\[
V_{t}(t,x)+\langle A^{\ast}(t)V_{x}(t,x),x\rangle+G(t,x,V(t,x),V_{x}%
(t,x),V_{xx}(t,x),\bar{u}(t))\geq0.
\]
From this we know that
\[
V_{t}(t,x)+\langle A^{\ast}V_{x}(t,x),x\rangle+G(t,x,V(t,x),V_{x}%
(t,x),V_{xx}(t,x),\bar{u}(t))
\]
attains its minimum at $\bar{X}(t)$. Thus,
\begin{align*}
0  &  =\frac{\partial}{\partial x}\Big\{V_{t}(t,x)+\langle A^{\ast}%
(t)V_{x}(t,x),x\rangle+G(t,x,V\left(  t,x\right)  ,\partial
_{x}V\left(  t,x\right)  ,\partial_{xx}V\left(  t,x\right)  ,\bar
{u}(t))\Big\}\Big|_{x=\bar{X}(t)}\\
&  =V_{tx}(t,\bar{X}(t))+A^{\ast}(t)V_{xx}(t,\bar{X}(t))\bar{X}(t)+A^{\ast
}(t)V_{x}(t,\bar{X}(t))\\
&  \ \ \ +V_{xx}(t,\bar{X}(t))a(t,\bar{X}(t),\bar{u}(t))+V_{x}(t,\bar{X}%
(t))a_{x}(t,\bar{X}(t),\bar{u}(t))\\
&  \ \ \ +\frac{1}{2}V_{xxx}(t,\bar{X}(t))(b(t,\bar{X}(t),\bar{u}(t)),b(t,\bar
{X}(t),\bar{u}(t)))+b_{x}^{\ast}(t,\bar{X}(t),\bar{u}(t))V_{xx}(t,\bar
{X}(t))b(t,\bar{X}(t),\bar{u}(t))\\
&\ \ \  +k_{x}(t)+k_{y}(t)V_{x}(t,\bar{X}(t)) +k_{z}(t)\B(b_{x}(t,\bar{X}(t),\bar{u}(t))V_{x}(t,\bar{X}(t))+b(t,\bar
{X}(t),\bar{u}(t))V_{xx}(t,\bar{X}(t))\B).
\end{align*}
Then, applying It\^{o}'s formula to $V_{x}(t,\bar{X}(t))$ and
combining the above equality$,$ we get
\begin{align*}
dV_{x}(t,\bar{X}(t))  &  =V_{xt}(t,\bar{X}(t))dt+V_{xx}(t,\bar{X}%
(t))\b[\b(A(t)\bar{X}(t)+a(t,\bar{X}(t),\bar{u}(t))\b)dt\\
&  \ \ \ +b(t,\bar{X}(t),\bar{u}(t))dw(t)\b]+\frac{1}{2}V_{xxx}(t,\bar{X}%
(t))\b(b(t,\bar{X}(t),\bar{u}(t)),b(t,\bar{X}(t),\bar{u}(t))\b)dt\\
&  =-\Big\{V_{x}(t,\bar{X}(t))\b[A^*(t)+a_{x}(t,\bar{X}(t),\bar{u}(t))+k_{y}%
(t)+k_{z}(t)b_{x}(t,\bar{X}(t),\bar{u}(t))\b]\\
&  \ \ \ +V_{xx}(t,\bar{X}(t))b(t,\bar{X}(t),\bar{u}(t))b_{x}(t,\bar{X}(t),\bar
{u}(t))\\
&  \ \ \ +k_{x}(t)+k_{z}(t)b(t,\bar{X}(t),\bar{u}(t))V_{xx}(t,\bar{X}%
(t))\Big\}dt+V_{xx}(t,\bar{X}(t))b(t,\bar{X}(t),\bar{u}(t))dw(t).
\end{align*}
Moreover, from the boundary condition in the HJB equation, we have
\[
V_{x}(T,\bar{X}(T))=h_{x}(\bar{X}(T)).
\]
So $\tilde{p}(t)=V_{x}(t,\bar{X}(t))$ and $\tilde{q}(t)=V_{xx}(t,\bar
{X}(t))b(t,\bar{X}(t),\bar{u}(t))$ also solve the first-order adjoint
equation (\ref{adjoint1}). From the uniqueness of solutions, we obtain
\[
p(t)=V_{x}(t,\bar{X}(t))\ \text{and}\ q(t)=V_{xx}(t,\bar{X}(t))b(t,\bar{X}%
(t),\bar{u}(t)).
\]
The proof is complete.
\end{proof}

\subsection{Examples}
In this subsection, we present two illustrative examples.
\begin{example}
 Let $G$ be a
bounded domain in $\mathbb{R}^{n}$. Consider the following super-parabolic stochastic PDE
(cf. \cite{Ro18}):
\begin{equation}
	\left\{
	\begin{aligned}
		dX(t,\zeta)= &\,\, \Big[\sum_{i,j=1}^{n}\partial_{\zeta_{i}}
		\big(\alpha_{ij}(t,\zeta)\partial_{\zeta_{j}}X(t,\zeta)\big)
		+ a\big(t,\zeta,u(t),X(t,\zeta)\big)\Big]dt \\
		& + \Big[\sum_{i=1}^{n}\beta_{i}(t,\zeta)\partial_{\zeta_{i}}X(t,\zeta)
		+ b\big(t,\zeta,u(t),X(t,\zeta)\big)\Big]dw(t),\quad (t,\zeta)\in[0,T]\times G, \\
		X(0,\zeta)= &\,\, x_{0}(\zeta),\quad \zeta\in G, \\
		X(t,\zeta)= &\,\, 0,\quad (t,\zeta)\in[0,T]\times\partial G.
	\end{aligned}
	\right.
\end{equation}
Here $\alpha_{ij},\beta_{i},a,b$ and\ $x_{0}$\ are given coefficients and
initial value, respectively.\ The control $u(t)$ is a progressive process
taking values in some metric space $U$. We aim at minimizing
the cost functional
\[
J(x_0;u(\cdot))=Y(0),
\]
where $Y$ is the recursive utility subjected to a BSDE:
\[
Y(t)=\int_{G}h(\zeta,X(T,\zeta))d\zeta+\int_{t}^{T}\int_{G}k(s,\zeta
,Y(s),Z(s),u(s),X(s,\zeta))d\zeta ds-\int_{t}^{T}Z(s)dw({s}).
\]
We take
\[
H=L^{2}(G),\ \mathcal{V}=H_{0}^{1}(G),\ A=\sum_{i,j=1}^{n}%
\partial_{\zeta_{i}}(\alpha_{ij}(t,\zeta)\partial_{\zeta_{j}}),\ %
B=\sum_{i=1}^{n}\beta_{i}(t,\zeta)\partial_{\zeta_{i}}.
\]
Assume there exist some constants $0<\kappa\leq K$ such that
\[
\kappa I_{n\times n}+(\beta_{i}\beta_{j})_{n\times n}\leq2(\alpha
_{ij})_{n\times n}\leq KI_{n\times n}.
\]
With mild measurability, differentiation and growth conditions on the
coefficients, the assumptions 
 $(H4),(H5),(H6)$  in Theorems \ref{diff-space} and \ref{rel-time}  can be verified.
\end{example}

\begin{example}
\label{newexa1}Consider the following control system: for any fixed
$a\in\mathcal{V},$
\begin{equation}\label{SEE1-1}
	\begin{cases}
		dX(t)
		=A(t)X(t)dt+bu(t)dw(t),\ 0\leq t\leq
		T,\quad X(0)=0, \\[4pt]
		Y(t)=\left\langle a,\,X(T)\right\rangle+\int_{t}^{T}[f(Z(s))-\left\langle
		A^{\ast}(s)a,\,X(s)\right\rangle ]ds-\int_{t}^{T}Z(s)dw(s),\ 0\leq t\leq
		T,
	\end{cases}
\end{equation}
where $A:[0,T]\rightarrow\mathfrak{L}(\mathcal{V}
;\mathcal{V}^{\ast})$ satisfying $(H4)$ and   $A^{\ast}(t)a \in C([0,T];H)$, the control domain $U$ is a separable metric space, the function $f:\mathbb{R} \rightarrow\mathbb{R}$ is  twice  differentiable   with continuous and bounded first- and second-order derivatives,  and $b\in
\mathfrak{L}(U,H)$ is a constant.  The corresponding first-order adjoint equation is given by
\begin{equation}
	\left\{
	\begin{aligned}
		-dp(t)= &\,\, \big[A^{\ast}(t)p(t) + f_{z}(\bar{Z}(t))\, q(t) - A^{\ast}(t)a\big]dt 
		- q(t)\,dw(t),\quad t\in[0,T], \\
		p(T)= &\,\, a.
	\end{aligned}
	\right.
\end{equation}
It has the  solution $(p,q)=(a,0).$ The second-order
adjoint equation is%
\[
P(t)=\mathbb{E}\Big[\int_{t}^{T}\tilde{L}^{\ast}(t,s)k_{zz}(\bar
{Z}(t))\b(q(t),q(t)\b)\tilde{L}(t,s)ds\Big|\mathcal{F}_{t}\Big],\quad0\leq t\leq
T,
\]
with
\[
\tilde{L}(t,s):=L_{\tilde{A},\tilde{B}}(t,s),\quad\text{for}\ \ \tilde
{A}(s):=A(s)-\frac{(f_{z}(\bar
	{Z}(s)))^{2}}{8}I_{d}\ \ \text{and}\ \ \tilde
{B}(s):=\frac{f_{z}(\bar
	{Z}(s))}{2}I_{d}.
\]
We can check that $P\equiv0.$ Thus, the maximum principle in Theorem \ref{SMP}  reads
\begin{equation}
f\b(\bar{Z}(t)+\langle a,b(u-\bar{u}(t))\rangle
\b)-f(\bar{Z}(t))\geq0,\quad \forall u\in U,\  P\text{-a.s.,
a.e.} \label{exa1}%
\end{equation}
Observe that, for any control process $u,$
\begin{equation}\label{eq4-22}
	\begin{split}
&  Y(t)-\left\langle a,\,X(t)\right\rangle=\left\langle
a,\,X(T)\right\rangle-\left\langle a,\,X(t)\right\rangle+\int%
_{t}^{T}\b[f(Z(s))-\left\langle A^{\ast}(s)a,\,X(s)\right\rangle 
\b]ds-\int_{t}^{T}Z(s)dw(s)\\
&  =\Big\langle a,\int_{t}^{T}A(s)X(s)ds+\int_{t}^{T}bu(s)dw(s)\Big\rangle
+\int_{t}^{T}\b[f(Z(s))-\left\langle A^{\ast}(s)a,\,X(s)\right\rangle \b]ds-\int_{t}^{T}Z(s)dw(s)\\
&  =\int_{t}^{T}f(Z(s))ds-\int_{t}^{T}\b(Z(s)-\langle a,bu(s)\rangle\b)dw(s)\\
&  =\int_{t}^{T}f\b((Z(s)-\langle a,bu(s)\rangle)+\langle a,bu(s)\rangle\b)ds-\int_{t}^{T}\b(Z(s)-\langle a,bu(s)\rangle\b)dw(s).
\end{split}
\end{equation}
From this we also know that
\begin{equation}\label{eq4-23}
\bar{Y}(t)-\left\langle a,\,\bar{X}(t)\right\rangle =\int_{t}^{T}f\b(\bar
{Z}(s)-\langle a,b\bar{u}(s)\rangle+\langle a,b\bar{u}(s)\rangle\b)ds-\int%
_{t}^{T}\b(\bar{Z}(s)-\langle a,b\bar{u}(s)\rangle\b)dw(s).
\end{equation}
If $\bar{u}$ satisfies (\ref{exa1}), then for any quadruple $(X,Y,Z,u),$ applying the comparison theorem of   BSDEs to (\ref{eq4-22}) and (\ref{eq4-23}), we deduce
\[
Y(t)-\left\langle a,\,X(t)\right\rangle \geq\bar{Y}(t)-\left\langle
a,\,\bar{X}(t)\right\rangle,\quad t\in\lbrack0,T].
\]
In particular, $Y(0)\geq\bar{Y}(0).$ That is,  Inequality (\ref{exa1}) is also sufficient (for $\bar{u}$ to be optimal). 

Now we take $U=\{0
,e_{1}\}$, where $0,e_1\in H_1$ with $e_1\neq 0$, for some Hilbert space $H_1$. Suppose  $f(0)=0$,
$f(\langle a,be_{1}\rangle)\geq0$. We can check that $(\bar{X},\bar{Y}%
,\bar{Z},\bar{u})=(0,0,0,0)$ satisfies (\ref{exa1}), thus $\bar{u}=0$ is an
optimal control. The HJB equation is%
\begin{equation*}
	\left\{
	\begin{aligned}
		\partial_{t}V(t,x)
		& + \langle A^{\ast}V_{x}(t,x),x\rangle \\
		& \quad + \inf_{v\in U} G\big(t,x,v,V(t,x),V_{x}(t,x),V_{xx}(t,x)\big)
		= 0,\quad (t,x)\in[0,T]\times H, \\
		V(T,x)
		& = x,\quad x\in H.
	\end{aligned}
	\right.
\end{equation*}
where%
\begin{align*}
& G(t,x,v,r,p,P)    :=\frac{1}{2}\langle Pbv,bv\rangle_{H%
}+f(\langle p,bv\rangle)-\left\langle A^{\ast}(s)a,\,x\right\rangle ,\\& \ \ \ \ \ \ \     
(t,x,v,r,p,P)    \in\lbrack0,T)\times H\times U\times\mathbb{R\times}H\times
S(H).
\end{align*}
It is direct to verify that  $V(t,x)=\left\langle a, x\right\rangle, \ (t,x)\in  [0,T]\times H$, is  the solution.
Hence,  we verifies the relationship  $p(t)=V_{x}(t,\bar{X}(t))=a$ in  Theorem  \ref{Rel-smooth}.
\end{example}

\section{Appendix}

\subsection{Proof of Theorem \ref{DPP}}

\begin{proposition}
\label{Le3-2} Let Assumptions $(H4)$ and $(H6)$ be satisfied. Then
\[
V(t,x)=\inf_{u(\cdot)\in\mathcal{U}^{t}[t,T]}Y^{t,x;u}(t), \quad (t,x)\in [0,T]\times H.
\]

\end{proposition}

\begin{proof}
Noting the inclusion $\mathcal{U}^{t}[t,T]\subset\mathcal{U}[t,T],$ we have $V(t,x)\leq
\inf_{u(\cdot)\in\mathcal{U}^{t}[t,T]}Y^{t,x;u}(t).$ On the other hand, for
any $u(\cdot)\in\mathcal{U}[t,T],$ by \cite[Lemma 13]{HJ-17},  
there exists a sequence $u^{m}$ taking the form
\[
u^{m}(s)=\sum_{i=1}^{N_{m}}v^{m}_i(s)I_{A^m_{i}},\quad  s\in\lbrack t,T],
\]
where $\{A^m_{i}\}_{i=1}^{N_{m}}$ is an $\mathcal{F}_{t}$-partition of $\Omega$
and $v^m_{i}\in\mathcal{U}^{t}[t,T],$ such that
\[
{\mathbb{E}}\B[\int_{t}^{T}|u^{m}(s)-u(s)|_{U}^{2}{d}
t\B]\rightarrow 0, \quad \text{ as }m\rightarrow\infty.
\]
From the a priori estimate of  BSDEs,
\begin{equation}%
\begin{split}
&  \mathbb{E}\B[\B\vert Y^{t,x;u^{m}}(t)-Y^{t,x;u}(t)\B\vert^{2}\B]\\
&  \leq C\mathbb{E}\B[\int_{t}^{T}\left\vert g(s,X^{t,x;u^m}(s),Y
^{t,x;u}(s),Z^{t,x;u}(s),u^{m}(s))-g(s,X^{t,x;u}(s),Y^{t,x;u}(s)
,Z^{t,x;u}(s),u(s))\right\vert ^{2}ds\B]\\
&  \leq C{\mathbb{E}}\B [\int_{t}^{T}|u^{m}(s)-u(s)|_{U}^{2}{d}t\B ]
 \rightarrow 0,\quad  \text{as}\ m\rightarrow\infty.
\end{split}
\label{eq-diff-contr}%
\end{equation}
Note that for $s\in\lbrack t,T]$,
\[
\B(X^{t,x;u^{m}}(s),\ Y^{t,x;u^{m}}(s),\ Z^{t,x;u^{m}}(s)\B)=\B(\sum
_{i=1}^{N_m}X^{t,x;v^m_{i}}(s)I_{A^m_{i}},\ \sum_{i=1}^{N_m}Y^{t,x;v^m_{i}}(s)
I_{A^m_{i}},\ \sum_{i=1}^{N_m}Z^{t,x;v^m_{i}}(s)I_{A^m_{i}}\B).
\]
Then
\begin{equation}
Y^{t,x;u^{m}}(t)=\ \sum_{i=1}^{N_m}Y^{t,x;v^m_{i}}(t)I_{A^m_{i}}\ \geq\ \sum_{i=1}
^{N_m}\underset{v\in\mathcal{U}^{t}[t,T]}{\inf}Y^{t,x;v}(t)I_{A^m_{i}
}\ =\ \underset{v\in\mathcal{U}^{t}[t,T]}{\inf}Y^{t,x;v}(t). \label{eq-deff-2}%
\end{equation}
This, combining (\ref{eq-diff-contr}), implies
\[
Y^{t,x;u}(t)\geq\underset{v\in\mathcal{U}^{t}[t,T]}{\inf}Y^{t,x;v}(t).
\]
Thus, $V(t,x)\geq\inf_{v\in\mathcal{U}^{t}[t,T]}Y^{t,x;v}(t)$. Therefore,
$V(t,x)=\inf_{v\in\mathcal{U}^{t}[t,T]}Y^{t,x;v}(t).$
\end{proof}

\begin{lemma}
\label{Le3-3} Assume $(H4)$ and $(H6)$. Then there exists a constant $C>0$
depending on $\delta$, $K_1$,   $L_3$, $L_4$ and $L_5$   such that, for each $u\in\mathcal{U}[t,T]$ and $\xi,\xi^{\prime
}\in L^{2}(\mathcal{F}_{t};H)$,
\begin{equation}
	\begin{split}
&	\mathbb{E}\Big[\sup\limits_{t\leq s\leq T}\left(  \Vert X^{t,\xi;u}(s)
	-X^{t,\xi^{\prime};u}(s)\Vert_{H}^{2}+|Y^{t,\xi;u}(s)-Y^{t,\xi^{\prime
		};u}(s)|^{2}\right)  \\&\ \ \ \ \ \ \ \ \ +\int_{t}^{T} 
| Z^{t,\xi;u}(s)-Z^{t,\xi^{\prime};u}(s)
	|^2
	ds\Big|\mathcal{F}_{t}\Big]\leq C\Vert\xi-\xi^{\prime}\Vert_{H}^{2};
	\end{split}
\end{equation}
\begin{equation}
\mathbb{E}\Big[\sup\limits_{t\leq s\leq T}\left(  \Vert X^{t,\xi;u}(s)
\Vert_{H}^{2}+|Y^{t,\xi;u}(s)|^{2}\right)  +\int_{t}^{T}| Z^{t,\xi
	;u}(s)
|^2ds\Big|\mathcal{F}_{t}\Big]\leq C\big(1+\Vert\xi\Vert_{H}^{2}\big).
\label{est-initial-2}%
\end{equation}

\end{lemma}

\begin{proof}
We only prove the second one, and the first one can be handled similarly. From
Lemma \ref{apriori-see}, we have
\begin{align*}
\mathbb{E}\Big[\sup\limits_{t\leq s\leq T}\Vert X^{t,\xi;u}(s)\Vert_{H}
^{2}\Big|\mathcal{F}_{t}\Big]  &  \leq C\mathbb{E}\Big[\Vert\xi\Vert_{H}%
^{2}+\int_{t}^{T}\Vert a\left(  s,0,u(s)\right)  \Vert_{H}^{2}ds+\int_{t}%
^{T}\Vert b\left(  s,0,u(s)\right)  \Vert_H%
^{2}ds\Big|\mathcal{F}_{t}\Big]\\
&  \leq C\Big(1+\Vert\xi\Vert_{H}^{2}\Big).
\end{align*}
Then from the basic estimate of  BSDEs, we obtain
\begin{align*}
&  \mathbb{E}\Big[\sup\limits_{t\leq s\leq T}|Y^{t,\xi;u}(s)|^{2}+\int%
_{t}^{T}|Z^{t,\xi
	;u}(s)
|^2ds\Big|\mathcal{F}_{t}\Big]\\
&  \leq C\mathbb{E}\Big[\Vert\xi\Vert_{H}^{2}+\Big(\int_{t}^{T}\b|k(
s,X^{t,\xi;u}(s),0,0,u(s)) \b| ds\Big)^{2}\Big|\mathcal{F}_{t}\Big]\\
&  \leq C\big(1+\Vert\xi\Vert_{H}^{2}\big).
\end{align*}
This completes the proof.
\end{proof}

\begin{lemma}
\label{Le3-4} Under $(H4)$ and $(H6)$, we have for some constant $C>0$
depending on $\delta$, $K_1$,   $L_3$, $L_4$ and $L_5$  such that, for each $t\in\lbrack0,T]$ and
$x,x^{\prime}\in H$,
\[
\big|V(t,x)-V(t,x^{\prime})\big|\leq C\Vert x-x^{\prime}\Vert_{H}\quad \text{and}\quad |V(t,x)|\leq C\big(1+\Vert x\Vert_{H}\big).
\]

\end{lemma}

\begin{proof}
Applying Proposition \ref{Le3-2} and Lemma \ref{Le3-3}, we have
\begin{align*}
|V(t,x)-V(t,x^{\prime})|  &  =\Big\vert\underset{v\in\mathcal{U}
^{t}[t,T]}{\inf}Y^{t,x;v}(t)-\underset{v\in\mathcal{U}^{t}[t,T]}{\inf}
Y^{t,x^{\prime};v}(t)\Big\vert\\
&  \leq\underset{v\in\mathcal{U}^{t}[t,T]}{\sup}\Big\vert Y^{t,x;v}(t)
-Y^{t,x^{\prime};v}(t)\Big\vert\\
&  \leq\underset{v\in\mathcal{U}^{t}[t,T]}{\sup}\Big\{\mathbb{E}
\Big[\sup\limits_{t\leq s\leq T}\Big\vert Y^{t,x;v}(s)-Y^{t,x^{\prime}
;v}(s)\Big\vert^{2}\Big|\mathcal{F}_{t}\Big]\Big\}^{\frac{1}{2}}\\
&  \leq C\Vert x-x^{\prime}\Vert_{H}.
\end{align*}
The other assertion is proved  in the same manner.
\end{proof}

\begin{proposition}
\label{pro3-5} Suppose $(H4)$ and $(H6)$. Then for each $\xi\in$
$L^{2}(\mathcal{F}_{t};H)$, we have
\[
V(t,\xi)=\underset{u\in\mathcal{U}[t,T]}{ess\inf}Y^{t,\xi;u}(t).
\]
On the other hand, for each $\varepsilon>0$, there exists an admissible
control $u_{\varepsilon}(\cdot)\in\mathcal{U}[t,T]$ such that
\begin{equation}
V(t,\xi)\geq Y^{t,\xi;u_{\varepsilon}}(t)-\varepsilon, \label{eq4-8}%
\end{equation}

\end{proposition}

\begin{proof}
We take a sequence $\xi^{m}$ $=\sum_{i=1}^{N_m}x_{i}^{m}I_{A_{i}^{m}}$ such that
$\mathbb{E}\left[  \Vert\xi^{m}-\xi\Vert_{H}^{2}\right]  \rightarrow0$ as
$m\rightarrow\infty$, where $\left\{  A_{i}^{m}\right\}  _{i=1}^{N_m}$ is a
$\mathcal{F}_{t}$-partition of $\Omega$ and $x_{i}^{m}\in H$. Note that for
$s\in\lbrack t,T]$,
\[
\big(X^{t,\xi_{m};u}(s),Y^{t,\xi_{m};u}(s),Z^{t,\xi_{m};u}(s)
\big)=\Big(\sum_{i=1}^{N_m}X^{t,x_{i}^{m};u}(s)I_{A_{i}^{m}},\sum_{i=1}
^{N_m}Y^{t,x_{i}^{m};u}(s)I_{A_{i}^{m}},\sum_{i=1}^{N_m}Z^{t,x_{i}^{m}
;u}(s)I_{A_{i}^{m}}\Big).
\]
Then we have
\begin{equation}%
\begin{split}
\underset{u\in\mathcal{U}[t,T]}{ess\inf}Y^{t,\xi_{m};u}(t)  &
=\underset{u\in\mathcal{U}[t,T]}{ess\inf}\sum\limits_{i=1}^{N_m}Y
^{t,x_{i}^{m};u}(t)I_{A_{i}^{m}}\\
&  =\sum\limits_{i=1}^{N_m}\Big(\underset{u\in\mathcal{U}[t,T]}{ess\inf}
Y^{t,x_{i}^{m};u}(t)\Big)I_{A_{i}^{m}}\\
&  =\sum\limits_{i=1}^{N_m}V\left(  t,x_{i}^{m}\right)  I_{A_{i}^{m}}\\
&  =V\left(  t,\xi^{m}\right)  .
\end{split}
\label{eq-new-111}%
\end{equation}
By Lemmas \ref{Le3-3} and \ref{Le3-4}, we have
\begin{equation}%
\Big\vert\underset{u\in\mathcal{U}[t,T]}{ess\inf}Y^{t,\xi_{m}
;u}(t)-\underset{u\in\mathcal{U}[t,T]}{ess\inf}Y^{t,\xi;u}(t)\Big\vert  
\leq\underset{u\in\mathcal{U}[t,T]}{ess\sup}\big\vert Y^{t,\xi_{m}
;u}(t)-Y^{t,\xi;u}(t)\big\vert   \leq C\Vert\xi^{m}-\xi\Vert_{H},
\label{eq-new-112}%
\end{equation}
\begin{equation}
\big\vert V(t,\xi^{m})-V\left(  t,\xi\right)  \big\vert\leq C\Vert\xi^{m}-\xi\Vert_{H}.
\label{eq-new-113}%
\end{equation}
Combining (\ref{eq-new-111}), (\ref{eq-new-112}) and (\ref{eq-new-113}), we
get
\[
\Big\vert\underset{u\in\mathcal{U}[t,T]}{ess\inf}Y^{t,\xi;u}(t)-V\left(
t,\xi\right)  \Big\vert\leq C\Vert\xi^{m}-\xi\Vert_{H}.
\]
Then the desired result is deduced by letting $m\rightarrow\infty$.

Next we consider (\ref{eq4-8}). From Proposition 11 in Chapter 1 of
\cite{Di00}, we can find elementary function $\xi^{\prime}=\sum
\limits_{i=1}^{\infty}1_{A_{i}}x_{i}$, where $x_i\in H$, $i\geq 1$ and  $\{A_{i}\}_{i=1}^{\infty}$ is an $\mathcal{F}_{t}$-partition of $\Omega$, such that
\[
\Vert\xi^{\prime}-\xi\Vert_{H}\leq\varepsilon.
\]
Then from Lemmas \ref{Le3-3} and \ref{Le3-4}, we have
\[
\big|Y^{t,\xi;u}(t)-Y^{t,\xi^{\prime};u}(t)\big|\leq C\varepsilon
,\quad\big|V(t,\xi)-V(t,\xi^{\prime})\big|\leq C\varepsilon.
\]
For each $x_{i}$, by Proposition \ref{Le3-2}, we can take $u^{i}\in
\mathcal{U}^{t}[t,T]$ such that
\[
V(t,x_{i})\geq Y^{t,x_{i};u^{i}}(t)-\varepsilon.
\]
Let $u(\cdot)=\sum\limits_{i=1}^{\infty}1_{A_{i}}u^{i}(\cdot)$. Then
\begin{align*}
Y^{t,\xi;u}(t)  &  =Y^{t,\xi^{\prime};u}(t)+Y^{t,\xi;u}(t)
-Y^{t,\xi^{\prime};u}(t)\\
&  \leq Y^{t,\xi^{\prime};u}(t)+C\varepsilon\\
&  =\sum\limits_{i=1}^{\infty}1_{A_{i}}Y^{t,x_{i};u_{i}}(t)+C\varepsilon\\
&  \leq\sum\limits_{i=1}^{\infty}1_{A_{i}}(V(t,x_{i})+\varepsilon
)+C\varepsilon\\
&  \leq V(t,\xi^{\prime})+C\varepsilon\\
&  \leq V(t,\xi)+C\varepsilon.
\end{align*}
This completes the proof by noting that $\varepsilon$ can be arbitrary.
\end{proof}

\begin{proof}
[Proof of Theorem \ref{DPP}]We first prove $V(t,x)\geq\inf_{u(\cdot
)\in\mathcal{U}^{t}[t,T]}G_{t,t+\delta}^{t,x;u}\big[V(t+\delta,X^{t,x;u}%
(t+\delta))\big].$ Given any $u\in\mathcal{U}[t,T].$ Note that
\begin{align*}
&  Y^{t,x;u}(s) =h(X^{t,x;u}(T))+\int_{s}^{T}k(r,X^{t,x;u}(r),Y^{t,x;u}
(r),X^{t,x;u}(r),u(r))dr-\int_{s}^{T}X^{t,x;u}(r)dw({r})\\
&  =Y^{t,x;u}(t+\delta)+\int_{s}^{t+\delta}k(r,X^{t,x;u}(r),Y^{t,x;u}
(r),X^{t,x;u}(r),u(r))dr-\int_{s}^{t+\delta}X^{t,x;u}(r)dw({r}),\text{\quad
}s\in\lbrack t,t+\delta].
\end{align*}
We derive
\begin{equation}
G_{t,T}^{t,x;u}\big[h(X^{t,x;u}(T))\big]=G_{t,t+\delta}^{t,x;u}\big[Y^{t,x;u}
(t+\delta)\big]. \label{Myeq4-6}%
\end{equation}
On the other hand, by the uniqueness of the solution to (\ref{SEE1-1}), we
have
\[
X^{t,x;u}(s)=X^{t+\delta,X^{t,x;u}(t+\delta);u}(s),\quad s\in\lbrack
t+\delta,T].
\]
Thus,
\begin{align*}
 Y^{t,x;u}(s)  =&h(X^{t+\delta,X^{t,x;u}(t+\delta);u}(T))+\int_{s}^{T}k(r,X^{t+\delta
,X^{t,x;u}(t+\delta);u}(s),Y^{t,x;u}(r),Z^{t,x;u}(r),u(r))dr\\&-\int_{s}
^{T}Z^{t,x;u}(r)dw({r}), \quad  s\in \lbrack t+\delta,T].
\end{align*}
So, from the uniqueness of solutions of BSDEs,
\begin{equation}
Y^{t,x;u}(s)=Y^{t+\delta,X^{t,x;u}(t+\delta);u}(s),\quad s\in\lbrack
t+\delta,T]. \label{Myeq4-9}%
\end{equation}
Consequently,
\begin{equation}
G_{t,t+\delta}^{t,x;u}\big[Y^{t,x;u}(t+\delta)\big]=G_{t,t+\delta}
^{t,x;u}\big[Y^{t+\delta,X^{t,x;u}(t+\delta);u}(t+\delta)\big].
\label{Myeq4-16}%
\end{equation}

From (\ref{Myeq4-6}) and (\ref{Myeq4-16}), 
\begin{equation}%
\begin{split}
V(t,x)  &  =\underset{u(\cdot)\in\mathcal{U}[t,T]}{ess\inf}G_{t,T}
^{t,x;u}\big[h(X^{t,x;u}(T))\big]\\
&  =\underset{u(\cdot)\in\mathcal{U}[t,T]}{ess\inf}G_{t,t+\delta}
^{t,x;u}\big[Y^{t+\delta,X^{t,x;u}(t+\delta);u}(t+\delta)\big].
\end{split}
\label{Myeq4-17-1}%
\end{equation}
and by Proposition \ref{Le3-2},
\begin{equation}%
	\begin{split}
V(t,x)  &  =\inf_{u(\cdot)\in\mathcal{U}^{t}[t,T]}Y^{t,x;u}(t)\\
&  =\inf_{u(\cdot)\in\mathcal{U}^{t}[t,T]}G_{t,T}^{t,x;u}\big[h(X^{t,x;u}
(T))\big]\\
&  =\inf_{u(\cdot)\in\mathcal{U}^{t}[t,T]}G_{t,t+\delta}^{t,x;u}
\big[Y^{t+\delta,X^{t,x;u}(t+\delta);u}(t+\delta)\big].
\end{split}
\label{Myeq4-17-2}%
\end{equation}
Applying Proposition \ref{pro3-5},
\[
Y^{t+\delta,X^{t,x;u}(t+\delta);u}(t+\delta)\geq V(t+\delta,X^{t,x;u}
(t+\delta)).
\]
From the comparison theorem of BSDEs, we have
\[
G_{t,t+\delta}^{t,x;u}\big[Y^{t+\delta,X^{t,x;u}(t+\delta);u}(t+\delta
)\big]\geq G_{t,t+\delta}^{t,x;u}\big[V(t+\delta,X^{t,x;u}(t+\delta))\big].
\]
Thus taking infimum over $u(\cdot)\in\mathcal{U}^{t}[t,T]$ on the both sides,
we get from (\ref{Myeq4-17-2}) that
\begin{equation}\label{eq5-15}
V(t,x)\geq\inf_{u(\cdot)\in\mathcal{U}^{t}[t,T]}G_{t,t+\delta}^{t,x;u}
\big[V(t+\delta,X^{t,x;u}(t+\delta))\big].
\end{equation}

Next we prove $V(t,x)\leq\underset{u(\cdot)\in\mathcal{U}[t,t+\delta
]}{ess\inf}G_{t,t+\delta}^{t,x;u}[V(t+\delta,X^{t,x;u}(t+\delta))].$ Fix any
$u(\cdot)\in\mathcal{U}[t,t+\delta].$ From Proposition \ref{pro3-5}, for any
$\varepsilon>0$, we can find an admissible control $\bar{u}(\cdot
)\in\mathcal{U}[t+\delta,T]$ such that
\[
V(t+\delta,X^{t,x;u}(t+\delta))\geq Y^{t+\delta,X^{t,x;u}(t+\delta);\bar{u}
}(t+\delta)-\varepsilon.
\]
Set $\tilde{u}(\cdot):=u(\cdot)I_{[t,t+\delta]}+\bar{u}(\cdot)I_{(t+\delta
,T]}\in\mathcal{U}[t,T],$ from (\ref{Myeq4-17-1}) and the comparison theorem
of BSDEs, we get
\begin{align*}
V(t,x)  &  =\underset{u(\cdot)\in\mathcal{U}[t,T]}{ess\inf}G_{t,t+\delta
}^{t,x;u}\big[Y^{t+\delta,X^{t,x;u}(t+\delta);u}(t+\delta)\big]\\
&  \leq\underset{u(\cdot)\in\mathcal{U}[t,T]}{ess\inf}G_{t,t+\delta
}^{t,x;\tilde{u}}\big[Y^{t+\delta,X^{t,x;\tilde{u}}(t+\delta);\tilde{u}
}(t+\delta)\big]\\
&  =\underset{u(\cdot)\in\mathcal{U}[t,T]}{ess\inf}G_{t,t+\delta}
^{t,x;u}\big[Y^{t+\delta,X^{t,x;u}(t+\delta);\bar{u}}(t+\delta)\big]\\
&  \leq\underset{u(\cdot)\in\mathcal{U}[t,T]}{ess\inf}G_{t,t+\delta}
^{t,x;u}\big[V(t+\delta,X^{t,x;u}(t+\delta))+\varepsilon\big].
\end{align*}
Then by the a priori estimate of  BSDEs,
\[
V(t,x)\leq\underset{u(\cdot)\in\mathcal{U}[t,T]}{ess\inf}G_{t,t+\delta
}^{t,x;u}\big[V(t+\delta,X^{t,x;u}(t+\delta))\big]+C\varepsilon.
\]
Letting $\varepsilon\rightarrow0$, we obtain
\[
V(t,x)\leq\underset{u(\cdot)\in\mathcal{U}[t,T]}{ess\inf}G_{t,t+\delta
}^{t,x;u}\big[V(t+\delta,X^{t,x;u}(t+\delta))\big].
\]
This, combines with (\ref{eq5-15}), implies the desired result.
\end{proof}

Now we state the time-continuity  of the value function $V$.
\begin{proposition}
\label{Le3-6} Assume $(H4)$ and $(H6)$ hold. Then $V$ is continuous in $t.$
\end{proposition}

\begin{proof}
For each $(t,x)\in\lbrack0,T)\times H$ and $\delta\in(0,T-t]$, from Theorem
\ref{DPP}, we have
\[
V(t,x)=\inf_{u(\cdot)\in\mathcal{U}^{t}[t,t+\delta]}G_{t,t+\delta}%
^{t,x;u}[V(t+\delta,X^{t,x;u}(t+\delta))].
\]
Then
\[
\left\vert V(t,x)-V\left(  t+\delta,x\right)  \right\vert \leq\underset{u\in
\mathcal{U}^{t}[t,t+\delta]}{\sup}\B\vert G_{t,t+\delta}^{t,x;u}%
[V(t+\delta,X^{t,x;u}(t+\delta))]-V\left(  t+\delta,x\right)  \B\vert .
\]
For any $u\in\mathcal{U}^{t}[t,t+\delta]$, from the definition of
$G_{t,t+\delta}^{t,x;u}\left[  \cdot\right]  $, we have
\[
G_{t,t+\delta}^{t,x;u}\left[  V(t+\delta,X^{t,x;u}(t+\delta))\right]
\newline=\mathbb{E}\Big[V(t+\delta,X^{t,x;u}(t+\delta))+\int_{t}^{t+\delta
}k\left(  s,X^{t,x;u}(s),Y^{t,x;u}(s),Z^{t,x;u}(s),u(s)\right)  ds\Big].
\]
Then applying Lemma \ref{Le3-4}, we obtain
\begin{equation}%
\begin{split}
&  \Big\vert G_{t,t+\delta}^{t,x;u}\left[V(t+\delta,X^{t,x;u}(t+\delta
))\right]-V(t+\delta,x)\Big\vert\\
&  \leq\mathbb{E}\Big[\b\vert V(t+\delta,X^{t,x;u}(t+\delta))-V(t+\delta
,x)\b\vert+{\int}_{t}^{t+\delta}\b\vert k\left(  s,X^{t,x;u}%
(s),Y^{t,x;u}(s),Z^{t,x;u}(s),u(s)\right)  \b\vert ds\Big]\\
&  \leq C\mathbb{E}\Big[\Vert X^{t,x;u}(t+\delta)-x\Vert_{H}+{\int}%
_{t}^{t+\delta}\big(1+\Vert X^{t,x;u}%
(s)\Vert_{H}+|Y^{t,x;u}(s)|+  | Z^{t,x;u}(s)   |\big)ds\Big].
\end{split}
\label{eq-new-133}%
\end{equation}
Noting that from Lemma \ref{Le3-3},
\[
\mathbb{E}\Big[\sup\limits_{t\leq s\leq t+\delta}\left(  \Vert X^{t,x;u}%
(s)\Vert_{H}^{2}+|Y^{t,x;u}(s)|^{2}\right)  +{\int}_{t}^{t+\delta}%
| Z^{t,x;u}(s)   |^{2}ds\Big]\leq C\big(1+\Vert x\Vert_{H}^{2}\big),
\]
we have by the Cauchy-Schwarz inequality that
\begin{equation}
\mathbb{E}\Big[\int_{t}^{t+\delta}\big(1+\Vert X^{t,x;u}(s)\Vert
_{H}+\left\vert Y^{t,x;u}(s)\right\vert +| Z^{t,x;u}(s)   |
\big)ds\Big]\leq C\big(  1+\Vert x\Vert_{H}\big)  \delta^{\frac{1}{2}}.
\label{eq-new-134}%
\end{equation}
Moreover, $\mathbb{E}\big[\Vert X^{t,x;u}(t+\delta)-x\Vert_{H}\big]\rightarrow0$ due
to Remark \ref{Rm3-6}.  Combining the above analysis, we obtain the desired result.
\end{proof}

\subsection{Proof of Lemma \ref{apriori-see}}
The proof of Assertion (i) is similar to the estimate following (2.24) in \cite[p.~3841]{LT21} and is therefore omitted. Assertion (ii) can be proved by adapting the argument of \cite[Lemma~4.4]{LT21} and a sketch of the proof is given below.

Let $\varepsilon>0$ and $\gamma>0$ be two  constants to be determined later. For simplicity 
of notation, we denote  by $C_{1}$  a generic constant
independent of $\varepsilon$ and $\gamma$, which may be different from line to
line. By the coercivity condition, 
\[
\Vert B(s)u\Vert_{H}\leq C_{1}\Vert u\Vert_{\mathcal{V}},\text{ for }%
u\in\mathcal{V}.
\]
Then
\begin{equation}%
\begin{split}
&  2\langle A(s)z(s)+a(s,z(s)),z(s)\rangle_{\ast}+\Vert B(s)z(s)+b(s,z(s))\Vert
_H^{2}\\
&  \leq2\langle A(s)z(s),z(s)\rangle_{\ast}+\Vert B(s)z(s)\Vert_H^{2}+2\langle B(s)z(s),b(s,z(s))\rangle+2\langle a(s,z(s)),z(s)\rangle_{\ast}+\Vert b(s,z(s))\Vert_H^{2}\\
&  \leq-\delta\Vert z(s)\Vert_{\mathcal{V}}^{2}+K\Vert z(s)\Vert_{H}%
^{2}+C(K)\Vert z(s)\Vert_{{\mathcal{V}}}\Vert b(s,z(s))\Vert_{H%
}+2\Vert a(s,z(s))\Vert_{\mathcal{V}^{\ast}}\Vert z(s)\Vert_{\mathcal{V}}+\Vert b(s,z(s))\Vert_H^{2}\\
&  \leq-\frac{\delta}{2}\Vert z(s)\Vert_{\mathcal{V}}^{2}+C_{1}\Vert
z(s)\Vert_{H}^{2}+C_{1}\Vert b(s,0)\Vert_H^{2}%
+C_{1}\Vert a(s,0)\Vert_{\mathcal{V}^{\ast}}^{2},
\end{split}
\label{a1-eq}%
\end{equation}
and from the quasi-skew-symmetry condition,
\begin{equation}
\begin{split}
\big|\langle B(s)z(s)+b(s,z(s)),z(s)\rangle\big|^{2}&\leq 
	2\big|\langle B(s)z(s),z(s)\rangle\big|^{2}
		+2\big|\langle b(s,z(s)),z(s)\rangle\big|^{2}\\
&\leq C_{1}\Vert z(s)\Vert_{H}%
^{4}+2\Vert b(s,z(s))|_H^{2}\Vert z(s)\Vert_{H}^{2}\\
	&\leq C_{1}\Vert z(s)\Vert_{H}%
	^{4}+4\Vert b(s,0)\Vert_H^{2}\Vert z(s)\Vert_{H}^{2}.
	 \label{a3-eq}
\end{split}
\end{equation}
Moreover, we have by the H\"{o}lder inequality and the Young's inequality that%
\begin{equation}%
\begin{split}
&  \mathbb{E}\Big[\int_{t}^{T}e^{-\gamma s}\Vert z(s)\Vert_{H}^{2(\alpha
-1)}\Vert a(s,0)\Vert_{\mathcal{V}^{\ast}}^{2}{d}s\Big|\mathcal{F}_{t}\Big]\\
&  \leq\varepsilon^{2}\mathbb{E}\Big[\sup_{s\in\lbrack t,T]}e^{-\gamma s}\Vert
z(s)\Vert_{H}^{2\alpha}\Big|\mathcal{F}_{t}\Big]+\frac{C_{1}}{\varepsilon^{2}%
}\mathbb{E}\Big[\Big(\int_{t}^{T}e^{-\frac{\gamma s}{\alpha}}\Vert
a(s,0)\Vert_{\mathcal{V}^{\ast}}^{2}{d}s\Big)^{\alpha}\Big|\mathcal{F}%
_{t}\Big]\\
&  \leq\varepsilon^{2}\mathbb{E}\Big[\sup_{s\in\lbrack t,T]}e^{-\gamma s}\Vert
z(s)\Vert_{H}^{2\alpha}\Big|\mathcal{F}_{t}\Big]+\frac{C_{1}}{\varepsilon^{2}%
}\mathbb{E}\Big[\Big(\int_{t}^{T}\Vert a(s,0)\Vert_{\mathcal{V}^{\ast}}^{2}%
{d}s\Big)^{\alpha}\Big|\mathcal{F}_{t}\Big],
\end{split}
\label{a2-eq}%
\end{equation}
and similarly,
\begin{equation}
\mathbb{E}\Big[\int_{t}^{T}e^{-\gamma s}\Vert z(s)\Vert_{H}^{2(\alpha-1)}\Vert
b(s,0)\Vert_H^{2}{d}s\Big|\mathcal{F}_{t}\Big]\leq
\varepsilon^{2}\mathbb{E}\Big[\sup_{s\in\lbrack t,T]}e^{-\gamma s}\Vert
z(s)\Vert_{H}^{2\alpha}\Big|\mathcal{F}_{t}\Big]+\frac{C_{1}}{\varepsilon^{2}%
}\mathbb{E}\Big[\Big(\int_{t}^{T}\Vert b(s,0)\Vert_H%
^{2}{d}s\Big)^{\alpha}\Big|\mathcal{F}_{t}\Big]. \label{a4-eq}%
\end{equation}
Therefore, we can calculate from (\ref{a3-eq}) and (\ref{a4-eq}) that
\begin{equation}%
\begin{split}
&  \mathbb{E}\Big[\sup_{s\in\lbrack t,T]}\Big|\int_{s}^{T}e^{-\gamma s}\Vert
z(s)\Vert_{H}^{2(\alpha-1)}\big\langle B(s)z(s)+b(s,z(s)),z(s)\big\rangle{d}
w(s) \Big|\Big|\mathcal{F}_{t}\Big]\\
&  \leq C_{1}\mathbb{E}\Big[\Big(\int_{t}^{T}e^{-2\gamma s}\Vert z(s)\Vert
_{H}^{4\alpha-4}
\big|\langle B(s)z(s)+b(s,z(s)),z(s)\rangle\big|^{2}
\,{d}s\Big)^{\frac{1}{2}%
}\Big|\mathcal{F}_{t}\Big]\\
&  \leq C_{1}\mathbb{E}\Big[\sup_{s\in\lbrack t,T]}e^{-\frac{\gamma s}{2}%
}\Vert z(s)\Vert_{H}^{\alpha}\Big(\int_{t}^{T}e^{-\gamma s}\big(\Vert z(s)\Vert
_{H}^{2\alpha}\,+\Vert z(s)\Vert_{H}^{2\alpha-2}\Vert b(s,z(s))\Vert
_H^{2}\big){d}s\Big)^{\frac{1}{2}}\Big|\mathcal{F}_{t}\Big]\\
&  \leq\varepsilon\mathbb{E}\Big[\sup_{s\in\lbrack t,T]}e^{-\gamma s}\Vert
z(s)\Vert_{H}^{2\alpha}\Big|\mathcal{F}_{t}\Big]+\frac{C_{1}}{\varepsilon}%
\mathbb{E}\Big[\int_{t}^{T}e^{-\gamma s}\big(\Vert z(s)\Vert_{H}^{2\alpha}\,+\Vert
z(s)\Vert_{H}^{2\alpha-2}\Vert b(s,z(s))\Vert_H^{2}%
\big){d}s\Big|\mathcal{F}_{t}\Big]\\
&  \leq\varepsilon\mathbb{E}\Big[\sup_{s\in\lbrack t,T]}e^{-\gamma s}\Vert
z(s)\Vert_{H}^{2\alpha}\Big|\mathcal{F}_{t}\Big]+\frac{C_{1}}{\varepsilon}%
\mathbb{E}\Big[\int_{t}^{T}e^{-\gamma s}\big(\Vert z(s)\Vert_{H}^{2\alpha}\,+\Vert
z(s)\Vert_{H}^{2\alpha-2}\Vert b(s,0)\Vert_H^{2}%
\big){d}s\Big|\mathcal{F}_{t}\Big]\\
&  \leq2\varepsilon\mathbb{E}\Big[\sup_{s\in\lbrack t,T]}e^{-\gamma s}\Vert
z(s)\Vert_{H}^{2\alpha}\Big|\mathcal{F}_{t}\Big]+\frac{C_{1}}{\varepsilon}\mathbb{E}\Big[\int_{t}^{T}e^{-\gamma s}\Vert
z(s)\Vert_{H}^{2\alpha}\,{d}s\Big|\mathcal{F}_{t}\Big]\\
& \ \ \ +C_{1}\frac{1}{\varepsilon^{3}}\mathbb{E}%
\Big[\Big(\int_{t}^{T}\Vert b(s,0)\Vert_H^{2}%
{d}s\Big)^{\alpha}\Big|\mathcal{F}_{t}\Big].
\end{split}
\label{a5-eq}%
\end{equation}

Applying It\^{o}'s formula to $e^{-\gamma s}\Vert z(s)\Vert_{H}^{2\alpha}$ on
$[t,T]$, we obtain
\begin{align*}
&  e^{-\gamma s}\Vert z(s)\Vert_{H}^{2\alpha}+\gamma\int_{t}^{s}e^{-\gamma
u}\Vert z(u)\Vert_{H}^{2\alpha}\,{d}u\\
&  =\Vert z_{0}\Vert_{H}^{2\alpha}+\alpha\int_{t}^{s}e^{-\gamma u}\Vert
z(u)\Vert_{H}^{2(\alpha-1)}\big(2\langle A(u)z(u)+a(u,z(u)),z(u)\rangle_{\ast}+\Vert
B(u)z(u)+b(u)\Vert_H^{2}\big)\,{d}u\\
& \ \ \ +2\alpha(\alpha-1)\int_{t}^{s}e^{-\gamma u}\Vert z(u)\Vert_{H}%
^{2(\alpha-2)}\big|\langle B(u)z(u)+b(u,z(u)),z(u)\rangle
\big|^{2}\,{d}u\\
& \ \ \ +2\alpha\int_{t}^{s}e^{-\gamma u}\Vert z(u)\Vert_{H}^{2(\alpha-1)}%
\big\langle B(u)z(u)+b(u,z(u)),z(u)\big\rangle{d}w(u)\\
&  \leq\Vert z_{0}\Vert_{H}^{2\alpha}+C_{1}\int_{t}^{s}e^{-\gamma u}\Vert
z(u)\Vert_{H}^{2(\alpha-1)}\Big(-\frac{\delta}{2}\Vert z(u)\Vert_{\mathcal{V}}^{2}%
+C_{1}\Vert z(u)\Vert_{H}^{2}+C_{1}\Vert b(u,0)\Vert_{H
}^{2}+C_{1}\Vert a(u,0)\Vert_{\mathcal{V}^{\ast}}^{2}\Big)\,{d}u\\
&  \ \ \ +C_{1}\int_{t}^{s}e^{-\gamma u}\Vert z(u)\Vert_{H}^{2(\alpha-2)}\big(\Vert
z(u)\Vert_{H}^{4}+\Vert b(u,0)\Vert_H^{2}\Vert
z(u)\Vert_{H}^{2}\big)\,{d}u\\
& \ \ \  +2\alpha\int_{t}^{s}e^{-\gamma u}\Vert z(u)\Vert_{H}^{2(\alpha-1)}%
\big\langle Bz(u)+b(u,z(u)),z(u)\big\rangle{d}w(u).
\end{align*}
Thus,%
\begin{equation}%
\begin{split}
&  e^{-\gamma s}\Vert z(s)\Vert_{H}^{2\alpha}+\gamma\int_{t}^{s}e^{-\gamma
u}\Vert z(u)\Vert_{H}^{2\alpha}\,{d}u+\frac{\delta}{2}C_{1}\int_{t}%
^{s}e^{-\gamma s}\Vert z(u)\Vert_{H}^{2(\alpha-1)}\Vert z(u)\Vert_{\mathcal{V}}^{2}%
\,{d}u\\
&  \leq\Vert z_{0}\Vert_{H}^{2\alpha}+C_{1}\int_{t}^{s}e^{-\gamma u}\Vert
z(u)\Vert_{H}^{2(\alpha-1)}\big(\Vert z(u)\Vert_{H}^{2}+\Vert b(u,0)\Vert
_H^{2}+\Vert a(u,0)\Vert_{\mathcal{V}^{\ast}}^{2}\big)\,{d}u\\
&  \ \ \ +C_{1}\int_{t}^{s}e^{-\gamma u}\Vert z(u)\Vert_{H}^{2(\alpha-2)}\big(\Vert
z(u)\Vert_{H}^{4}+\Vert b(u,0)\Vert_H^{2}\Vert
z(u)\Vert_{H}^{2}\big)\,{d}u\\
&  \ \ \ +2\alpha\int_{t}^{s}e^{-\gamma u}\Vert z(u)\Vert_{H}^{2(\alpha-1)}%
\big\langle Bz(u)+b(u,z(u)),z(u)\big\rangle{d}w(u).
\end{split}
\label{Eq6-2}%
\end{equation}
Taking supremum over $[t,T]$ and conditional expectation on both sides, from (\ref{a3-eq}),
(\ref{a4-eq}) and (\ref{a5-eq}) we get%
\begin{align*}
&  \mathbb{E}\Big[\sup_{s\in\lbrack t,T]}e^{-\gamma s}\Vert z(s)\Vert_{H}%
^{2\alpha}\Big|\mathcal{F}_{t}\Big]+\gamma\mathbb{E}\Big[\int_{t}^{T}e^{-\gamma
s}\Vert z(s)\Vert_{H}^{2\alpha}\,{d}s\Big|\mathcal{F}_{t}\Big]\\
&  \leq C_{1}(\varepsilon+\varepsilon^{2})\mathbb{E}\Big[\sup_{s\in\lbrack
t,T]}e^{-\gamma s}\Vert z(s)\Vert_{H}^{2\alpha}\Big|\mathcal{F}_{t}\Big]+\Vert
z_{0}\Vert_{H}^{2\alpha}+C_{1}\Big(1+\frac{1}{\varepsilon}\Big)\mathbb{E}\Big[\int_{t}^{T}e^{-\gamma s}\Vert z(s)\Vert_{H}^{2\alpha
}\,{d}s\Big|\mathcal{F}_{t}\Big]\\
& \ \ \  +\frac{C_1}{\varepsilon
	^{2}}\mathbb{E}\Big[\Big(\int_{t}^{T}\Vert a(s,0)\Vert_{\mathcal{V}^{\ast}}%
^{2}{d}s\Big)^{\alpha}\Big|\mathcal{F}_{t}\Big]+C_{1}\Big(\frac{1}{\varepsilon^2}+\frac{1}{\varepsilon^3}\Big)\mathbb{E}\Big[\Big(\int_{t}^{T}\Vert
b(s,0)\Vert_H^{2}{d}s\Big)^{\alpha}\Big|\mathcal{F}_{t}\Big].
\end{align*}
Choosing $\varepsilon$ small first and then $\gamma$ large, we obtain
\begin{equation}
\mathbb{E}\Big[\sup_{s\in\lbrack t,T]}\Vert z(s)\Vert_{H}^{2\alpha}\Big|\mathcal{F}%
_{t}\Big]\leq C\Big\{\Vert z_{0}\Vert_{H}^{2\alpha}+\mathbb{E}\Big[\Big(\int_{t}^{T}\Vert
a(s,0)\Vert_{\mathcal{V}^{\ast}}^{2}{d}s\Big)^{\alpha}\Big|\mathcal{F}%
_{t}\Big]+\mathbb{E}\Big[\Big(\int_{t}^{T}\Vert b(s,0)\Vert_H%
^{2}{d}s\Big)^{\alpha}\Big|\mathcal{F}_{t}\Big]\Big\}. \label{Eq6-3}%
\end{equation}

Now let $\alpha=1$ in (\ref{Eq6-2}), we get
\begin{align*}
\int_{t}^{T}\Vert z(s)\Vert_{\mathcal{V}}^{2}\,{d}s  &  \leq\Vert z_{0}\Vert_{H}%
^{2}+C\int_{t}^{T}\Vert z(s)\Vert_{H}^{2}\,{d}s\\
&  +C \int_{t}^{T}\big(\Vert a(s,0)\Vert_{\mathcal{V}^{\ast}}^{2}+\Vert
b(s,0)\Vert_H^{2}\big)\,{d}s+C_2 \int_{t}^{T}e^{-\gamma
s}\big\langle B(s)z(s)+b(s),z(s)\big\rangle{d}w(s).
\end{align*}
Then by  (\ref{a3-eq}),
\begin{align*}
  \mathbb{E}\Big[\Big(\int_{t}^{T}\Vert z(s)\Vert_{\mathcal{V}}^{2}%
\,{d}s\Big)^{\alpha}\Big|\mathcal{F}_{t}\Big]
&  \leq\Vert z_{0}\Vert_{H}^{2\alpha}+C\mathbb{E}\Big[\int_{t}^{T}\Vert
z(s)\Vert_{H}^{2\alpha}\,{d}s\Big|\mathcal{F}_{t}\Big]\\
&  \ \ \ +C\mathbb{E}\Big[\B(\int_{t}^{T}\b(\Vert a(s,0)\Vert_{\mathcal{V}^{\ast}}%
^{2}+\Vert b(s,0)\Vert_H^{2}\b){d}%
s\B)^{\alpha}\Big|\mathcal{F}_{t}\Big]\\
& \ \ \ +C\mathbb{E}\Big[\B(\int_{t}^{T}\langle B(s)z(s)+b(s),z(s)\rangle{d}%
w(s)\Big)^\alpha\Big|\mathcal{F}_{t}\Big]\\
&  \leq\Vert z_{0}\Vert_{H}^{2\alpha}+C\mathbb{E}\Big[\int_{t}^{T}\Vert
z(s)\Vert_{H}^{2\alpha}\,{d}s\Big|\mathcal{F}_{t}\Big]\\
&\ \ \ +C\mathbb{E}\Big[\B(\int_{t}^{T}\b(\Vert a(s,0)\Vert_{\mathcal{V}^{\ast}}%
^{2}+\Vert b(s,0)\Vert_H^{2}\b){d}%
s\B)^{\alpha}\Big|\mathcal{F}_{t}\Big]\\
&\ \ \  +C\mathbb{E}\Big[\B(\int_{t}^{T}\big|\langle B(s)z(s)+b(s,z(s)),z(s)\rangle\big|^{2}ds\Big)^{\frac{\alpha}{2}}\Big|\mathcal{F}_{t}\Big]\\
	&  \leq\Vert z_{0}\Vert_{H}^{2\alpha}+C\mathbb{E}\Big[\int_{t}^{T}\Vert
	z(s)\Vert_{H}^{2\alpha}\,{d}s\Big|\mathcal{F}_{t}\Big]\\
	&\ \ \ +C\mathbb{E}\Big[\B(\int_{t}^{T}\b(\Vert a(s,0)\Vert_{\mathcal{V}^{\ast}}%
	^{2}+\Vert b(s,0)\Vert_H^{2}\b){d}%
	s\B)^{\alpha}\Big|\mathcal{F}_{t}\Big]\\
		&\ \ \ +C\mathbb{E}\Big[\sup_{s\in\lbrack t,T]}\Vert z(s)\Vert_{H}^{2\alpha}\Big|\mathcal{F}%
		_{t}\Big].
\end{align*}
Combining this with (\ref{Eq6-3}), we get the desired result.

\subsection{Proof of Proposition \ref{Prop5-6}}

\begin{proof}
Fix any $(t,x)\in\lbrack0,T]\times H$ and $\delta\in(0,T-t].$ From Theorem
\ref{DPP}, we know that
\[
V(t,x)=\underset{u(\cdot)\in\mathcal{U}[t,t+\delta]}{ess\inf}G_{t,t+\delta
}^{t,x;u}[V(t+\delta,X^{t,x;u}(t+\delta))]=\underset{u(\cdot)\in
\mathcal{U}^{t}[t,t+\delta]}{\inf}G_{t,t+\delta}^{t,x;u}[V(t+\delta
,X^{t,x;u}(t+\delta))].
\]
For any fixed control $u\in\mathcal{U}^{t}[t,t+\delta],$ let $X^u(s):=X^{t,x;u}%
(s)$, $s\geq t$ and let $(Y^u,Z^u)$ be the solution of (\ref{EQ6}) with $\eta=V(t+\delta
,X^u({t+\delta}))$ on $[t,t+\delta]$. Applying It\^{o}'s formula  (see \cite[Lemma 2.15]{P2021}), we have that%
\begin{align*}
 V(s,X^u({s})) =&V(t+\delta,X^u({t+\delta}))-\int_{s}^{t+\delta}\b[V_{r}%
(r,X^u({r}))+\langle V_{x}(r,X^u({r})),a(r,X^u({r}),u(r))\rangle\\
&+\langle A^{\ast}%
V_{x}(r,X^u({r})),X({r})\rangle  +\frac{1}{2}\b\langle V_{xx}(r,X^u({r}))b(r,X^u({r}),u(r)),b(r,X^u({r}),u(r))\b\rangle
\b]dr\\
&-\int_{s}^{t+\delta}\langle V_{x}(r,X^u({r})),b(r,X^u({r}),u(r))\rangle
{d}w(r),\quad  s\in\lbrack t,t+\delta].
\end{align*}
Thus,
\begin{equation}%
\begin{split}
 Y^u(s)-V(s,X^u({s}))   =&\int_{s}^{t+\delta}\b[k(r,X^u(r),Y^u(r),Z^u(r),u(r))+V_{r}%
(r,X^u({r}))\\
& +\langle V_{x}(r,X^u({r})),a(r,X^u({r}),u(r))\rangle   +\langle A^{\ast}V_{x}(r,X^u({r})),X^u({r})\rangle\\
&  +\frac{1}{2}\b\langle
V_{xx}(r,X^u({r}))b(r,X^u({r}),u(r)),b(r,X^u({r}),u(r))\b\rangle\b]dr\\
& -\int_{s}^{t+\delta}[Z^u(r)-\langle V_{x}(r,X^u({r})),b(r,X^u({r}),u(r))\rangle
]dw({r}),\quad s\in\lbrack t,t+\delta].
\end{split}
\label{EQ55}%
\end{equation}

We set
\[
\hat{Y}^u(s):=Y^u(s)-V(s,X^u({s}))\text{ and }\hat{Z}^u(s):=Z^u(s)-\langle V_{x}%
(s,X^u({s})),b(s,X^u({s}),u(s))\rangle.
\]
Then (\ref{EQ55}) reads
\begin{align*}
\hat{Y}^u(s)  &  =\int_{s}^{T}[k(r,X^u(r),\hat{Y}^u(r)+V(r,X^u({r})),\hat{Z}(r)+\langle
V_{x}(r,X^u({r})),b(r,X^u({r}),u(r))\rangle,u(r))\\
&  +V_{r}(r,X^u({r}))+\langle V_{x}(r,X^u({r})),a(r,X^u({r}),u(r))\rangle+\langle
A^{\ast}V_{x}(r,X^u({r})),X({r})\rangle\\
&  +\frac{1}{2}\b\langle V_{xx}(r,X^u{(r)})b(r,X^u({r}),u(r)),b(r,X^u({r}),u(r))\b\rangle
]dr-\int_{s}^{t+\delta}\hat{Z}^u(r)dw({r}),\quad s\in\lbrack t,t+\delta],
\end{align*}
which is a (real-valued) BSDE with $\hat{Y}^u(s),\hat{Z}^u(s)$ being the solutions.

We consider another BSDE
\begin{equation}%
\begin{split}
\hat{Y}^{1,u}(s)  &  =\int_{s}^{t+\delta}\b[k(r,x,\hat{Y}^{1,u}(r)+V(r,x),\hat
{Z}^u(r)+\langle V_{x}(r,x),b(r,x,u(r))\rangle,u(r))\\
&  +V_{r}(r,x)+\langle V_{x}(r,x),a(r,x,u(r))\rangle+\langle A^{\ast}%
V_{x}(r,x),x\rangle\\
&  +\frac{1}{2}\langle V_{xx}(r,x)b(r,x,u(r)),b(r,x,u(r))\rangle\b]dr-\int%
_{s}^{t+\delta}\hat{Z}^{1,u}(r)dw({r}),\quad s\in\lbrack t,t+\delta].
\end{split}
\label{Eq6-6}%
\end{equation}
From the a priori estimate for  BSDEs and Remark \ref{Rm3-6}, we have
\begin{align*}
|\hat{Y}^u(t)-\hat{Y}^{1,u}(t)|^{2}  &  \leq\delta\int_{t}^{t+\delta}%
\mathbb{E}\Big[\sup_{t\leq r\leq t+\rho}\left\Vert X^u(r)-x\right\Vert _{H}%
^{2}\Big|\mathcal{F}_{t}\Big]dr\\
&  \leq\delta^{2}\mathbb{E}\Big[\sup_{t\leq r\leq t+\rho}\left\Vert
X^u(r)-x\right\Vert _{H}^{2}\Big|\mathcal{F}_{t}\Big]\\
&  =o(\delta^{2}).
\end{align*}
Denote the generator of BSDE (\ref{Eq6-6}) by%
\begin{align*}
F(s,x,y,z,v)  &  :=k(s,x,y+V(s,x),z+\langle V_{x}(s,x),b(s,x,v)\rangle,v)\\
&  +V_{s}(s,x)+\langle V_{x}(s,x),a(s,x,v)\rangle+\langle A^{\ast}%
V_{x}(s,x),x\rangle\\
&  +\frac{1}{2}\langle V_{xx}(s,x)b(s,x,v),b(s,x,v)\rangle.
\end{align*}
We consider the backward ODE
\begin{equation}
\hat{Y}^{0}(s)=\int_{s}^{t+\delta}F_{0}(r,x,\hat{Y}^{0}(r),0)dr,\quad %
s\in\lbrack t,t+\delta], \label{Eq6-7}%
\end{equation}
where
\begin{align*}
F_{0}(s,x,y,z)  &  =\inf_{v\in U}F(s,x,y,z,v)\\
&  =V_{s}(s,x)+\langle A^{\ast}V_{x}(s,x),x\rangle+\inf_{v\in U}\b[\langle
V_{x}(s,x),a(s,x,v)\rangle+\frac{1}{2}\langle V_{xx}%
(s,x)b(s,x,v),b(s,x,v)\rangle\\
& \ \ \ \ +k(s,x,y+V(s,x),z+\langle V_{x}(s,x),b(s,x,v)\rangle,v)\b].
\end{align*}
Note that from Theorem \ref{DPP},
\[
V(t,x)=\inf_{u(\cdot)\in\mathcal{U}^{t}[t,t+\delta]}Y^u(t).
\]
Thus,
\[
\inf_{u(\cdot)\in\mathcal{U}^{t}[t,t+\delta]}\hat{Y}^u(t)=0.
\]
So, from the following Lemma \ref{inf-lem},
\[
\hat{Y}^{0}(t)=\inf_{u(\cdot)\in\mathcal{U}^{t}[t,t+\delta]}\hat{Y}%
^{1,u}(t)=o(\delta).
\]
Dividing by $\delta>0$ on both sides of (\ref{Eq6-7}) for   $s=t,$ and then   letting $\delta\downarrow0,$ from the formula for upper limit integral,
we get
\[
F_{0}(t,x,0,0)=0,
\]
which is just the HJB equation. The proof is complete.
\end{proof}

\begin{lemma}
\label{inf-lem} For $\hat{Y}^{1,u}(t)$ and $\hat{Y}^{0}(t)$ defined as in the
proof of Proposition \ref{Prop5-6}, we have
\[
\inf_{u(\cdot)\in\mathcal{U}^{t}[t,t+\delta]}\hat{Y}^{1,u}(t)=\hat{Y}^{0}(t)
\]

\end{lemma}

\begin{proof}
For each given $u(\cdot)\in\mathcal{U}^{t}[t,t+\delta],$ it holds $F(s,x,0,0,u(s))\geq
F_{0}(s,x,0,0).$ So by the comparison theorem of  BSDEs, we have $\hat{Y}
^{1,u}(t)\geq\hat{Y}^0(t).$ On the other hand, from the measurable selection theorem
and the compactness of $U,$ there exists a (deterministic) measurable function
$\bar{a}:[0,T]\times H\times\mathbb{R}\times \mathbb{R}\rightarrow U$ such
that
\[
F_{0}(s,x,y,z)=F_{1}(s,x,y,z,\bar{a}(s,x,y,z)),\quad (s,x,y,z)\in
\lbrack0,T]\times H\times\mathbb{R}\times \mathbb{R}.
\]
We define $u(s)=\bar{a}(s,x,\hat{Y}^{0}(s),0).$ At this case, from the
uniqueness of solutions of BSDEs, we have $\hat{Y}^{1,u}(s)=\hat{Y}
^{0}(s),s\in\lbrack t,t+\delta],$ in particular, $\hat{Y}^{1,u}(t)=\hat{Y}
^{0}(t)$. Combining the above analysis, we could get the desired result.
\end{proof}

\noindent\textbf{Acknowledgement. }

The second author would like to thank Wei Liu (JSNU) for helpful discussions.

\end{document}